\documentclass[11pt,reqno,a4paper]{amsart}
\usepackage[margin=1in]{geometry}
\usepackage{amsmath, amsthm, amssymb}
\usepackage[colorlinks=true, pdfstartview=FitV, linkcolor=blue,citecolor=blue, urlcolor=blue]{hyperref}
\usepackage[abbrev,lite,nobysame]{amsrefs}
\usepackage{times}
\usepackage[usenames,dvipsnames]{color}
\usepackage{subcaption}
\usepackage{mathtools}
\usepackage{bbm}

\mathtoolsset{showonlyrefs=true}
\usepackage{dsfont}
\usepackage{tikz}
\usetikzlibrary{patterns}
\usetikzlibrary{fpu}
\usetikzlibrary{plotmarks}

\def\dd{{\rm d}}

\def\eps{\varepsilon}
\def\e{{\rm e}} 
\def\de{{\partial}}

\def\RR {\mathbb{R}}
\def\TT {\mathbb{T}}
\def\ZZ {\mathbb{Z}}

\def\SS {\mathbb{S}}

\def\Re{{\rm Re}}

\newcommand{\norm}[1]{\left\lVert #1 \right\rVert}
\newcommand{\normL}[1]{\lVert #1 \rVert}
\newcommand{\jap}[1]{\left\langle #1 \right\rangle}
\newcommand{\abs}[1]{\left| #1 \right|}
\newcommand{\brak}[1]{\left\langle #1 \right\rangle}
\newcommand{\grad}{\nabla}

\def\bP {\boldsymbol{P}}

\newcommand{\sfm}{\mathsf{m}}
\newcommand{\sfe}{\mathsf{e}}

\newcommand{\cF}{\mathcal{F}}
\newcommand{\cL}{\mathcal{L}}
\newcommand{\cD}{\mathcal{D}}
\newcommand{\cI}{\mathcal{I}}
\newcommand{\cM}{\mathcal{M}}
\newcommand{\cA}{\mathcal{A}}
\newcommand{\cK}{\mathcal{K}}
\newcommand{\cQ}{\mathcal{Q}}
\newcommand{\cT}{\mathcal{T}}
\newcommand{\cR}{\mathcal{R}}

\newtheorem{proposition}{Proposition}[section]
\newtheorem{theorem}[proposition]{Theorem}

\newtheorem{lemma}[proposition]{Lemma}
\theoremstyle{definition}

\newtheorem{remark}[proposition]{Remark}

\theoremstyle{definition}

\numberwithin{equation}{section}
\title[Decay in the weakly collisional Boltzmann equation]{Taylor dispersion and phase mixing in the non-cutoff Boltzmann equation on the whole space} 
\author[J. Bedrossian]{Jacob Bedrossian}
\address{Department of Mathematics, University of Maryland, College Park, MD 20742, USA}
\email{jacob@math.umd.edu}
\author[M. Coti Zelati]{Michele Coti Zelati}
\address{Department of Mathematics, Imperial College London, London, SW7 2AZ, UK}
\email{m.coti-zelati@imperial.ac.uk}
\author[M. Dolce]{Michele Dolce}
\address{Institute of Mathematics, EPFL, Station 8, 1015 Lausanne, Switzerland}
\email{michele.dolce@epfl.ch}

\begin{document}
\maketitle
\begin{abstract}
In this paper we describe the long-time behavior of the non-cutoff Boltzmann equation with soft potentials 
near a global Maxwellian background on the whole space in the weakly collisional limit (i.e. infinite Knudsen number $1/\nu\to \infty$). 
Specifically, we prove that for initial data sufficiently small (independent of the Knudsen number), the solution 
displays several dynamics caused by the phase mixing/dispersive effects of the transport operator $v \cdot \grad_x$ and its interplay with the singular collision operator.
For $x$-wavenumbers $k$ with  $|k|\gg\nu$, one sees an \emph{enhanced dissipation} effect wherein the characteristic decay time-scale is accelerated 
to $O(1/\nu^{\frac{1}{1+2s}} \abs{k}^{\frac{2s}{1+2s}})$, where $s \in (0,1]$ is the singularity of the kernel ($s=1$ being the Landau collision operator, 
which is also included in our analysis); for $|k|\ll \nu$, one sees \emph{Taylor dispersion}, wherein the decay is accelerated to $O(\nu/\abs{k}^2)$. 
Additionally, we prove almost-uniform phase mixing estimates. For macroscopic quantities as the density $\rho$, these bounds imply almost-uniform-in-$\nu$ decay of $(t\grad_x)^\beta \rho$ in $L^\infty_x$ due to Landau damping and dispersive decay.
%Additionally, we prove almost-uniform phase mixing estimates which imply almost-uniform-in-$\nu$ decay of $(t\grad_x)^\beta \rho$ in $L^\infty_x$ due to Landau damping and dispersive decay.
\end{abstract}

\setcounter{tocdepth}{1}
\tableofcontents

\section{Introduction}
In this paper we consider the Boltzmann equation for a distribution function $F = F(t,x,v)$,
\begin{align}
	\label{eq:Bolt}
	\de_t F+v\cdot \nabla_xF=\nu \cQ(F,F), \qquad t\geq 0, \, x,v\in \RR^d \, (\text{or}\,  \TT^d),
\end{align}
where $1/\nu \geq 0$ is the Knudsen number or inverse collision frequency.  We are interested in understanding the long-time dynamics of this equation in the limit $\nu \to 0$ , i.e. the weakly collisional limit.
That is, we describe the dynamics simultaneously in $t \to \infty$ and $\nu \to 0$. 
The Boltzmann equation serves as the classical kinetic model for the dynamics of a rarefied, monoatomic gas with $\cQ$ accounting for elastic, binary collisions between the gas molecules. 
The collision operator $\cQ$ is 
\begin{align}
	\label{def:Q}
	\cQ(f,g)(v)=\frac12\int_{\RR^d}\int_{\SS^{d-1}} B(v-v_*,\sigma)(f'g'_*-fg_*)\dd v_*\dd \sigma,
\end{align}
with 
\begin{align}
	\label{def:v'v*}
	\begin{cases}v'=\displaystyle \frac{v+v_*}{2}+\frac{|v-v_*|}{2}\sigma,\\
		v'_*=\displaystyle \frac{v+v_*}{2}-\frac{|v-v_*|}{2}\sigma,
	\end{cases} \qquad f'=f(v'), \quad g'_*=g(v'_*), \quad f=f(v), \quad g_*=g(v_*).
\end{align}
The velocities $v',\, v'_*$ are the result of an elastic collision between particles with equal mass and velocities $v,v_*$. In particular, momentum and energy are conserved,  namely 
\begin{equation}
	\label{eq:consmomen}
	\begin{cases}v'+v'_*=v+v_*,\\
		|v'|^2+|v'_*|^2=|v|^2+|v_*|^2.	
	\end{cases}
\end{equation}
As in  \cite{alexandre2012boltzmann}, we consider the class of non-negative kernels $B=B(z,\sigma):\RR^d\times \SS^{d-1}\to\RR$  depending only on $|z|$ and the angle between the vectors $z$ and $\sigma$. In particular, we take
	\begin{equation}
		\label{hyp:kernel}B(z,\sigma)=\Phi(|z|)b(\cos(\theta)), \qquad \cos(\theta)=\frac{z}{|z|}\cdot \sigma, \quad 0\leq \theta \leq \frac{\pi}{2}.
	\end{equation}
where the kinetic factor is given by 
\begin{equation}
	\label{phi}
	\Phi(|z|)=|z|^\gamma,
\end{equation}
whereas the collision angle contains a singularity 
\begin{equation}
	\label{eq:propb}
	b(\cos(\theta))\approx K\theta^{1-(d+2s)}, \qquad \text{as } \theta \to 0^+, \qquad 0<s<1.
\end{equation}
More precisely, as in \cite{gressman2011sharp}, we assume no angular cutoff, namely for some $s \in (0,1)$ we have
\begin{equation}
	\label{eq:hypb}
	\sin(\theta)b(\cos(\theta))\approx \frac{1}{\theta^{d+2s}}.
\end{equation}
See \cite{cercignani2013mathematical} for more discussion about the derivation of the collision operator. 
The parameter $\gamma$ differentiates \emph{hard} from \emph{soft} potentials in the following manner: 
\begin{itemize}
\item Soft potentials: $-d\leq \gamma+2s\leq 0$ and $\gamma>\max\{-d,-d/2-2s\}$;
\item Moderately soft potentials: $-2s < \gamma < 0 $;
\item Hard potentials: $\gamma \geq 0$.
\end{itemize}
We will treat all of these cases, but mostly focus on the soft potentials, as these are the most difficult, since collisions have a weaker effect on high velocity particles. 
The singular limit $s \to 1$ reduces to the \emph{Landau collision operator}  \cite{BoydSanderson,villani2002review}.
In this limit, the collision operator becomes a second-order elliptic operator, which usually makes the analysis simpler.
This limit is important as it is a significantly more accurate model of charged particle collisions than Boltzmann, and hence is one of the standard collision models used in plasma physics. 
All our results apply to the Landau collision operator (see Remark \ref{rmk:Landau}), however, we focus our attention on the more difficult Boltzmann case.  

The thermal equilibria are described by the  global Maxwellians
\begin{align*}
\mu_{n,u,T}(v)=\frac{n}{(2\pi T)^{\frac{d}{2}}}\e^{-\frac{|v-u|^2}{2T}}, 
\end{align*}
with parameters $n,T > 0$, $u  \in \RR^d$ and all satisfy $\mathcal{Q}(\mu,\mu) = 0$. 
In this work we consider small (to be quantified), localized perturbations to a global thermal equilibrium, specifically we study solutions of the form 
\begin{equation}
	F=\mu+\sqrt{\mu}f, \qquad \mu(v)=\frac{1}{(2\pi)^{\frac{d}{2}}}\e^{-\frac{|v|^2}{2}}.
\end{equation}
The Maxwellian $\mu$ is a stationary solution to \eqref{eq:Bolt} thanks to \eqref{eq:consmomen}.
The equation for the perturbation $f$ is given by 
\begin{equation}
	\de_t f+v\cdot \nabla_x f+ \nu \cL f= \nu\Gamma(f,f), \label{eq:PertEqn}
\end{equation}
where 
\begin{align}
	\label{def:Gamma}
	\Gamma(g,h)&=\frac{1}{{\sqrt{\mu}}}\cQ\left({\sqrt{\mu}}g,{\sqrt{\mu}}h\right)\\
	\label{def:L}
	\cL f&=-\Gamma(\sqrt{\mu},f)-\Gamma(f,\sqrt{\mu}):=\cL_1f+\cL_2f.
\end{align}
In view of the conservation of the energy, we know $\mu \mu_*=\mu'\mu'_*$ and therefore 
\begin{align}
	\label{eq:Gamma}
	\Gamma(g,h)&=\frac12\iint_{\RR^d\times \SS^{d-1}}B(v-v_*,\sigma)\sqrt{\mu_*}\left(g'_*h'-g_*h\right)\dd v_*\dd \sigma,\\
	\label{eq:explL}
	(\cL f)(v)&=-\frac{1}{2\sqrt{\mu}}\iint_{\RR^d\times \SS^{d-1}}B(v-v_*,\sigma)\mu\mu_*\left(\frac{f'}{\sqrt{\mu'}}+\frac{f'_*}{\sqrt{\mu'_*}}-\frac{f}{\sqrt{\mu}}-\frac{f_*}{\sqrt{\mu_*}}\right)\dd v_*\dd \sigma.
\end{align}
From \eqref{eq:explL} and the conservation of mass, momentum and energy, it readily follows that 
\begin{equation}
	\label{eq:KerL}
	\operatorname{Ker}(\cL)=\mathrm{span}\left\{\sqrt{\mu},v\sqrt{\mu},|v|^2\sqrt{\mu}\right\}.
\end{equation}
We will denote the projection onto the kernel as $\bP $ and for the perturbation $f$ we denote
\begin{align}
\bP f(t,x,v) = \left(\rho(t,x) + \sfm(t,x) \cdot v  + \sfe(t,x) (\abs{v}^2-d) \right) \sqrt{\mu}; \label{def:rhoue}
\end{align}
the unknowns $(\rho,\sfm,\sfe)$ correspond respectively to the macroscopic, or hydrodynamic, quantities of density, momentum and energy. 

On $\mathbb T^d$, small data global existence near Maxwellians was proved for the Landau equations by Guo in \cite{guo2002landau} and for non-cutoff Boltzmann in independent works by Gressman and Strain \cite{gressman2011global} (extended to the whole space by Strain \cite{strain2012optimal}) and Alexandre et. al. in \cite{alexandre2010regularizing,alexandre2011global,alexandre2012boltzmann}.   
See \cite{ukai1974existence,ukai1982cauchy,caflisch1980boltzmann} for some of the earlier works on Boltzmann with angular cutoffs and  \cite{guo2004boltzmann,guo2006boltzmann,duan2009stability} for the corresponding stability and decay estimates.  
Many further extensions of these foundational works exist, for example, covering cases which include self-generated electric and/or magnetic fields such as \cite{guo2002vlasov,duan2013vlasov,guo2012vlasov,strain2013vlasov,duan2011optimal,duan2013noncutoff}. 

The study of the $\nu \to 0$ weakly collisional limit in kinetic theory is relatively new in mathematics. 
However, in plasmas, collisions are typically very weak, and hence there are many works in the physics literature considering how the collisions and the phase mixing due to the free transport will interact albeit mostly in the context of Landau or Fokker-Planck collisions  \cite{lenard1958plasma,su1968collisional,ng2006weakly,o1968effect}.
If one considers the toy model 
\begin{align}
\partial_t g + v \cdot \grad_x g = \nu \Delta_v g, \label{eq:ktm}
\end{align}
it is not hard to show that the $x$-Fourier modes $\hat{g}(t,k,v) = \frac{1}{(2\pi)^{d/2}} \int_{\mathbb R^d} \e^{-ik\cdot x} g(t,x,v) \dd x$ undergo
 \emph{enhanced dissipation} in the regime $\abs{k} \gg \nu$, namely for some $\delta > 0$ there holds
\begin{align}
\norm{\hat{g}(t,k)}_{L_v^2} \lesssim \e^{-\delta \nu \abs{k}^{2} t^3} \norm{\hat{g}(0,k)}_{L_v^2}. \label{ineq:EDintro}
\end{align}
This was essentially observed by Kelvin \cite{Kelvin87} for the 2D incompressible Navier-Stokes equations linearized near the Couette flow and was predicted to hold also in plasmas due to charged particle collisions in \cite{lenard1958plasma,su1968collisional}.
The effect here is that the free streaming $v\cdot \grad_x $ term creates large gradients in $v$ of the size $\sim |k|t$, which correspondingly enhance the effect of the $\Delta_v$ dissipation to $\sim |k|^2 t^2$, which explains form in \eqref{ineq:EDintro}. 
As the leading order singularity in the non-cutoff Boltzmann equation \eqref{eq:PertEqn} is similar to a fractional Laplacian of order $s\in(0,1)$ (see e.g. discussions in \cite{alexandre2011global,gressman2011global,silvestre2016new}), the rate of enhanced dissipation will depend on $s$ in the corresponding manner; see Remark \ref{rmk:Rate} below. 

Another important effect observed in kinetic theory with $0 \leq \nu \ll 1$ is \emph{phase mixing} and \emph{Landau damping}, which generally refers to the rapid damping of hydrodynamic fields such as the density. This was first observed in the linearized Vlasov--Poisson equations by Landau in \cite{Landau46} and is now considered a fundamental aspect of collisionless plasma physics (see e.g. \cite{Ryutov99,BoydSanderson}).  
For the kinetic transport equation \eqref{eq:ktm} Landau damping\footnote{Often the term ``Landau damping'' is reserved specifically for the effect in plasma physics, however we will sometimes abuse terminology and use it here as well, as the origin of the damping is the same.} is relatively simple to verify, leading to the following estimates: for all $\beta \geq 0$, $m > d/2$ and uniformly in $\nu$, there holds
\begin{align*} 
 \norm{\abs{t\grad_x}^\beta \rho}_{L^2_t(0,\infty;\dot{H}^{1/2}_x)} & \lesssim \norm{\brak{v}^{m} \brak{\grad_v}^\beta g(0)}_{L^2_{x,v}} \\
 \sup_{t \geq 0} \brak{t}^{d} \norm{ \abs{t \grad_x}^\beta \rho}_{L^\infty_x} & \lesssim \sup_k \norm{\brak{\grad_v}^{2m + \beta}\hat{g}(0,k)}_{L^1_v}.
\end{align*}
The first inequality follows from the Fourier transform formula for \eqref{eq:ktm}; see Lemma \ref{lem:LDLF} below for a proof of the second inequality.
In particular, we see that regularity in $v$ translates directly to decay of $\rho$.
Phase mixing refers to the fact that, provided there is some regularity of the initial distribution function, all of the particles are traveling at different velocities, which tends to correspondingly smooth out the density variations. 
On $\mathbb R^d$, the decay is due to a combination of phase mixing and dispersion.

Landau damping for all sufficiently regular (analytic and sufficiently high Gevrey regularity) initial conditions for $x\in \mathbb T^d$ was proved by Mouhot and Villani \cite{MV11}; the Gevrey classes conjectured to be sharp was obtained in a later simplified proof \cite{BMM13}.
The analogue on $\mathbb R^3$, Landau damping and dispersive decay, was proved in Sobolev spaces for the screened Vlasov-Poisson equations in \cite{BMM16}; see also \cite{han2021asymptotic,huang2022sharp,ionescu2022nonlinear,huang2022nonlinear} for related works. 
In \cite{B17} the enhanced dissipation effect was proved in the Vlasov--Poisson--Fokker--Planck equations on $\mathbb T^d$ near a global Maxwellian, along with uniform-in-$\nu$ Landau damping estimates on the density, provided the initial velocity-weighted $H^\sigma$ norm of the initial condition was $\lesssim \nu^{1/3}$ (see earlier work by Tristani \cite{tristani2017landau} that proves uniform Landau damping for the linearized problem).  
More recently, the case of Vlasov--Poisson--Landau equation on $\mathbb T^d$ was solved by Chaturvedi, Luk, and Nguyen \cite{CLN21}. 
Due to the nonlinear echo phenomenon, the $\mathcal{O}(\nu^{1/3})$ threshold seems potentially sharp in Sobolev spaces \cite{bedrossian2020nonlinear}.  
See also similar recent results on active suspension models for the collective motion of swimming of microorganisms \cite{albritton2022stabilizing,CZHGV22}. 
This kind of stability threshold problem mirrors the work on quantitative stability results based on enhanced dissipation in the incompressible Navier-Stokes equations (see e.g. \cite{BMV16,G18,WZ19,WZZ20,BGM20,CZEW20,CLWZ20}).
However, enhanced dissipation has probably been best studied mathematically in passive scalars (see e.g. \cite{CKRZ08,CZMD20,Z10,BW13,albritton2022enhanced,BCZ15,CZDE20,Vukadinovic} and the references therein). 
For passive scalars, another effect arises, known as \emph{Taylor dispersion}, wherein for $\abs{k} \ll \nu$ the effective dissipation rate for $\hat{f}(t,k)$ of a passive scalar in a shear flow (in say, a channel) would be $\mathcal{O}(\nu^{-1} \abs{k}^2)$.
This effect was first predicted by Taylor in \cite{Taylor53}; see also \cite{Aris59}. 
The work \cite{BCW20} provided the first mathematically rigorous proof using center manifold theory, whereas \cite{CZG21} showed that a unified energy method could be used to treat both $\abs{k} \gg \nu$ and $\abs{k} \ll \nu$. We will discuss this in more detail in Section \ref{sec:outLin} below.
This effect was also noted in the works on active suspensions in the linearized equations \cite{albritton2022stabilizing}.  
 
The goal of this work is to prove that for sufficiently small initial data (independent of $\nu$), the Boltzmann equation \eqref{eq:PertEqn} displays all of these effects: enhanced dissipation, Taylor dispersion, and a nearly-uniform-in-$\nu$ phase mixing/Landau damping (discussed more below).
A careful reading of the work of \cite{strain2012optimal} already shows that global existence of strong solutions holds for initial conditions small independent of $\nu$: 
there exists a $\nu$-independent $ \eps_0 > 0$ such that if the initial datum $f_{in}$ satisfies
\begin{align*}
\sum_{\abs{\alpha} + \abs{\beta} \leq \sigma} \norm{\brak{v}^{m} \partial_\alpha^\beta f_{in} }_{L^2_{x,v}} + \norm{\brak{v}^{m} \partial_\alpha^\beta f_{in} }_{L^2_v L^1_x}  \leq \eps_0,
\end{align*}
for sufficiently large $\sigma,m$, then the solution is global in time and vanish as $t \to \infty$.
Hence, the specific aim of this paper is to obtain the quantitative estimates in the limit $\nu \to 0$. 
Neither the methods of \cite{CLN21} nor \cite{strain2012optimal} apply here although ideas from both are utilized; see Section \ref{sec:Outline} for an outline of the proof. 
Note in particular that the Landau damping is not quite the same as \eqref{eq:ktm} for $t \gg \nu^{-1}$. 
More precise estimates that quantify the phase mixing are discussed in Section \ref{sec:Outline}. 
Below we denote the density by
\begin{align*}
\rho(t,x) = \int_{\mathbb R^d} f(t,x,v) \dd v. 
\end{align*}
Our main result reads as follows.
\begin{theorem} \label{thm:main}
Let $s \in (0,1)$ and $\gamma$ be either soft, moderately soft, or hard.
Let $\sigma > d$ and $M,M ' > d$ be integers. 
There exists $\eps_0 > 0$ (independent of $\nu$) such that if the initial datum $f_{in}$ satisfies
\begin{align*}
  \sum_{\abs{\alpha} + \abs{\beta} \leq \sigma} \norm{\brak{v}^{M+M'} \partial_\alpha^\beta f_{in} }_{L^2_{x,v}} + \norm{\brak{v}^{M+M'} \partial_\alpha^\beta f_{in} }_{L^2_v L^1_x} = \eps \leq \eps_0, 
\end{align*}
then the following holds for the corresponding solution $f$ to \eqref{eq:PertEqn} and for all $\nu \in (0,1)$:

\smallskip 
\noindent {\rm (i)}
For  $\delta_0,\delta_1 > 0$ determined by the proof (depending only on $\sigma,M,M'$), define  
\begin{equation}
  \label{def:lambdanuk_intro}
  \lambda_{\nu,k}=\lambda(\nu,k)=\delta_1\begin{cases}
		\displaystyle  \nu^{\frac{1}{1+2s}}|k|^{\frac{2s}{1+2s}}, \qquad &\nu/|k|\leq \delta_0,\\
		\displaystyle  \nu^{-1}|k|^{2}, \qquad &\nu/|k|> \delta_0.	
\end{cases}
\end{equation}
For an integer $J = J(\sigma,M,M',\gamma,s)$ there holds the Taylor dispersion/enhanced dissipation estimate
\begin{align*}
\sup_{t > 0} \norm{ \brak{v}^M \brak{\lambda(\nu,\grad_x) t}^J f(t)}_{L^2_{x,v}} \lesssim \eps.
\end{align*}
A formula for $J$ can be found in Theorem \ref{thm:LinDecEst}; importantly, for any $N$ we can choose $M,M'$ sufficiently large such that $J > N$. 

\smallskip 
\noindent {\rm (ii)}
If $\frac{2s}{1+2s} J > d$, we have the low frequency decay estimate 
\begin{align*}
\norm{ \brak{v}^{M/2} \brak{\lambda(\nu,\grad_x) t}^{J/2} f(t)}_{L^2_v L^\infty_{x}} \lesssim \left(\brak{\nu t^{1+2s}}^{-d/2s} + \nu^{d/2}\brak{t}^{-d/2} \right) \eps.
\end{align*}

\smallskip 
\noindent {\rm (iii)}
The density satisfies almost-uniform Landau damping, namely for $J$ as in {\rm (ii)} and all $\sigma - d > \beta > 0$, there holds 
\begin{align*}
\norm{\frac{\brak{t \grad_x}^\beta}{\brak{\nu t}^\beta} \brak{\lambda(\nu,\grad_x) t}^J \rho(t)}_{L^\infty_x} \lesssim \left(\brak{t}^{-d} +\nu^{d/2}\brak{t}^{-d/2} \right) \eps.
\end{align*}
\end{theorem}
\begin{remark}
Similar almost-uniform Landau damping estimates hold on other hydrodynamic moments of $f$. Almost-uniformity refers to the fact that one observes the same Landau damping as expected from the kinetic free transport until times $t \approx \nu^{-1}$, at which point the hydrodynamic fields are already $\mathcal{O}(\nu^{d})$. 
\end{remark}

\begin{remark} \label{rmk:Landau}
  All our results hold for the generalized Landau collision operator equation on $d = 3$ (setting $s=1$), i.e. the collision operator given by
\begin{align*}
\mathcal{Q}_L(F,G) = \grad_v \cdot \left( \int_{\mathbb R^3} \phi(v-v')\left[F(v') \grad_v G - G(v) \grad_{v'}F(v') \right] \dd v'\right), 
\end{align*}
where $\phi$ is the non-negative matrix given by
\begin{align*}
\phi_{ij}(v) = \left(I - \frac{v \otimes v}{\abs{v}^2}\right) \abs{v}^{2+\gamma}.
\end{align*}
The classical case is $\gamma = -3$ (which is included in our analysis).
This collision operator is generally easier for our analysis than non-cutoff Boltzmann due to the local nature of the derivatives, which allows for a more direct use of hypocoercivity and simpler commutator estimates for the derivatives. 
\end{remark}

\begin{remark} \label{rmk:Rate}
For $\abs{k} \gg \nu$, the form of $\lambda(\nu,k)$ can be guessed using the simple toy problem
\begin{align*}
\partial_t g + y \partial_x g = -\nu(-\Delta)^s g, 
\end{align*}
which can be solved explicitly using Fourier analysis on $(x,v) \in \mathbb T \times \mathbb R$, yielding for some $\delta > 0$, 
\begin{align*}
\norm{\hat{g}(t,k)}_{L_y^2} \lesssim \e^{- \delta \nu \abs{k}^{2s} t^{1+2s}}\norm{\hat{g}(0,k)}_{L_y^2}. 
\end{align*}
For $\abs{k} \ll \nu$ the rate is the same as the Taylor dispersion rate predicted for passive scalars by Taylor \cite{Taylor53}.
It is interesting to note that this rate does not depend on $s$ (or $\gamma$).   
\end{remark} 

The proof on $\mathbb T^d$ is significantly easier as one can adapt ideas from \cite{CLN21} to the Boltzmann collision operators (see especially Section \ref{sec:Lin} for what is required for this adaptation). 
We state the result for completeness but will not discuss the proof further.
Note that $d=1$ is possible in this theorem as the decay is not limited by low frequencies here. 
\begin{theorem}[Result on $\mathbb T^d$ for $d \geq 1$] \label{thm:Td}
If the initial datum $f_{in}$ satisfies
\begin{align}
\sum_{\abs{\alpha} + \abs{\beta} \leq \sigma} \norm{\e^{q \abs{v}^2} \partial_\alpha^\beta f_{in} }_{L^2_{x,v}} = \eps \leq \eps_0, \label{ineq:2xexp}
\end{align}
then the following holds for the corresponding solution $f$ to \eqref{eq:PertEqn} and for all $\nu \in (0,1)$:

\smallskip 
\noindent {\rm (i)}
There holds the enhanced dissipation estimate 
\begin{align*}
  \norm{ \brak{v}^M \left(f(t) - \int f(t,x,\cdot) \dd x \right) }_{L^2_{x,v}} & \lesssim \eps  \min\left\{\e^{- \delta_0\left(\nu^{\frac{1}{1+2s}} t\right)^{p_{_{\gamma,s}}}}, \e^{- \delta_0\left(\nu t\right)^{\frac{2}{2+|\gamma+2s|}}}\right\} \\
  \norm{ \brak{v}^M \int f(t,x,\cdot) \dd x }_{L^2_{x,v}} &\lesssim \eps  \e^{- \delta_0(\nu t)^{\frac{2}{2+|\gamma+2s|}}},
\end{align*}
where $$p_{_{\gamma,s}}=\frac{2(1+s)}{2(1+s)+|\gamma|(3-s)-\vartheta_s}$$ with $\vartheta_1=0$ and $\vartheta_s>0$  arbitrarily small for any $s\in(0,1)$;

\smallskip 
\noindent {\rm (ii)}
For all $\beta < \sigma-d$,  there holds the uniform Landau damping estimate 
\begin{align*}
\norm{ \brak{t \grad_x}^\beta \rho(t)}_{L^2} \lesssim \eps  \min\left\{\e^{- \delta_0\left(\nu^{\frac{1}{1+2s}} t\right)^{p_{_{\gamma,s}}}}, \e^{- \delta_0\left(\nu t\right)^{\frac{2}{2+|\gamma+2s|}}}\right\}. 
\end{align*}
\begin{remark}
The stretched exponential decay rates can be easily explained by looking at the linearized problem. In particular, as shown by Strain and Guo \cite{MR2366140}, this time-decay is a direct consequence of monotonicity estimates and properties of the collision operator. The bounds we state follows by a straightforward improvement (due to the Gaussian weight) of Lemma \ref{lem:splitdecay} and the proof of Theorem \ref{thm:LinDecEst}. The rate $2/(2+|\gamma+2s|)$ is the one that can be inferred for the Boltzmann equation\footnote{The stretched exponential rate obtained by Strain and Guo in \cite{MR2366140} is $2/(2+|\gamma|)$ for the Botlzmann equation and $2/3$ for the Landau collision operator (corresponding to $2/(2+|\gamma+2s|)$ when $\gamma=-3$ and $s=1$). This improvement in the rate for the Boltzmann case is due to better bounds on $\cL$ obtained in subsequent works \cite{gressman2011global,alexandre2011global}, see Proposition \ref{lemma:coercive} herein.} directly from the monotonicity estimate \eqref{bd:energyED} and the properties of $\cL$ given in Proposition \ref{lemma:coercive}.  Notice that for $s=1$ we have $p_{_{\gamma,1}}=2/(2+|\gamma|)$. Thus, for the Landau collision operator we get a power $p_{_{-3,1}}=2/5$ which is an improvement over the $1/3$ power given in \cite{CLN21}. We believe this improvement is related to the velocity weights we use in our energy for the hypocoercive scheme. In particular, we use a weight $\brak{v}^{-3/2}$ for terms involving $v$-derivatives whereas in \cite{CLN21} a weight $\brak{v}^{-4}$ is used.
\end{remark}
\end{theorem}

\section{Outline of the proof}\label{sec:Outline}

\subsection{Linearized problem}
\label{sec:outLin}
The first step in the proof of Theorem \ref{thm:main} is  to understand the linearized problem
\begin{align*}
\de_t f+v\cdot \nabla_x f+ \nu \cL f= 0, 
\end{align*}
in the singular limit $\nu \to 0$. 
Due to  translation invariance, we Fourier transform in $x$, and obtain
\begin{align}
\de_t \hat{f} + ik \cdot v \hat{f}+ \nu \cL \hat{f} = 0,  \label{def:Link}
\end{align}
where
\begin{align*}
\hat{f}(t,k,v) := \frac{1}{(2\pi)^{d/2}} \int_{\mathbb R^d} \e^{-i k \cdot x} f(t,x,v) \dd x. 
\end{align*}
In analogy with the passive scalar problem \cite{CZG21}, the behavior of \eqref{def:Link} is different depending on the relationship between $\abs{k}$ and $\nu$. 
There are two distinct regimes:
\begin{align*}
& \abs{k} \gg \nu  \quad \textup{the enhanced dissipation regime}; \\ 
& \abs{k} \ll \nu  \quad \textup{the Taylor dispersion regime}. 
\end{align*}
The most significant difference between \eqref{def:Link} and the passive scalar problem (where $\cL$ is replaced by $\Delta_v$) lies  in the Taylor dispersion regime, which will require 
a more complicated energy method than the hypocoercivity employed in  \cite{CZG21}. This is also reflected in the quantification of phase mixing, which we carry out through 
the vector field method as done in several previous works (see e.g. \cite{WZZ20beta,C20,CLN21,CZHGV22}). 
Define the vector field
\begin{align*}
Z = \grad_v + t \grad_x, 
\end{align*}
and observe the commutation property
\begin{align*}
[Z,\partial_t + v \cdot \grad_x ] = 0. 
\end{align*}
This vector field is exactly equivalent to the $\grad_v$ derivatives used in previous works that make the coordinate change $z = x-tv$ (for example \cite{MV11,BMM13}), and so controlling $Z$ implies the Landau damping of $\rho$ and other hydrodynamic moments in exactly the same manner.
While it is convenient for treating Vlasov-Poisson in some contexts, the change of coordinates is not convenient for dealing with the collision operator, so it will be easier to work with $Z$ instead.

\smallskip 
\noindent $\diamond$ \textbf{Enhanced dissipation regime: $\nu/|k|\leq \delta_0$.}
In this range of $k,\nu$, we use the energy functional
\begin{align}
\label{def:Eed}
E^{e.d.}_{M, B}(t,k) := \frac12\sum_{\alpha + |\beta|\leq B}&\frac{2^{-\mathtt{C} \beta} \brak{k}^\alpha}{\brak{\nu t}^{2\beta}}\bigg(\normL{\jap{v}^{M} Z^\beta \hat{f}}^2_{L^2_v}+a_{\nu,k}\normL{\jap{v}^{M+q_{\gamma,s}} Z^\beta\nabla_v \hat{f}}^2_{L^2_v} \notag \\
		&\hspace{-3cm} + 2b_{\nu,k}\Re \jap{\jap{v}^{M+q_{\gamma,s}}Z^\beta ik \hat{f},\jap{v}^{M+q_{\gamma,s}}Z^\beta\nabla_v\hat{f}}_{L^2_v}\bigg).
\end{align}
where $\alpha\in \mathbb{N}$,  $\beta$ is a multi-index, $\mathtt{C}\geq 1$ is sufficiently large depending on the proof, 
\begin{equation}
\label{def:qgammas}
 \gamma/2s-\kappa_0<q_{\gamma,s}<\gamma/2s, \qquad 0<\kappa_0 \ll 1, 
 \end{equation}
 where $\kappa_0$ can be chosen arbitrarily small ($q_{\gamma,1}=\gamma/2$ for $s=1$ in the Landau case) and 
\begin{equation}
	\label{def:abnuk}
	a_{\nu,k}=a_0(\nu/|k|)^{\frac{2}{2s+1}},  \qquad b_{\nu,k}=b_0 
	(\nu/|k|)^{\frac{1}{2s+1}}|k|^{-1}, 
	\end{equation}
with $0 < a_0 \ll b_0 \ll 1$ fixed constants determined by the proof in Section \ref{sec:Lin}.
For the linearized problem \eqref{def:Link} the $\brak{\nu t}^{-2\beta}$ in the norm is only necessary in the Taylor dispersion regime (on $\mathbb T^d$ it is never needed). Its presence is associated with problems emanating from low frequencies that then cause issues also in the nonlinear interactions. 
Notice that
\begin{align*}
E^{e.d.}_{M, B}(t,k) \approx \sum_{\alpha +|\beta| \leq B}&\frac{2^{-\mathtt{C}\beta} \brak{k}^\alpha}{\brak{\nu t}^{2\beta}} \left( \normL{\jap{v}^{M} Z^\beta \hat{f}}^2_{L^2_v} + a_{\nu,k}\normL{\jap{v}^{M+q_{\gamma,s}} Z^\beta\nabla_v \hat{f}}^2_{L^2_v} \right), 
\end{align*}
where the implicit constants are independent of $k$ or $\nu$.  
There is no need to specifically treat the kernel of $\cL$ since ignoring $\mathbf{P} f$ is equivalent to get error terms scaling as $\normL{\mu^\delta f}_{L^2_v}$. In this regime, those are under control since a key point of the hypocoercivity scheme is to produce damping of $\grad_x f$ with the expected decay rate. This automatically implies the desired control over the hydrodynamic fields using that $\abs{k}$ is bounded below.
On $\mathbb T^d$, this energy is essentially all that is required to treat both the linear and nonlinear problem by Poincar\'e inequality; see \cite{CLN21} where essentially this method was used on the Vlasov-Poisson-Landau equations (i.e. $s=1$ case with nonlinear electrostatic interactions) on $\mathbb T^d$. 

The energy comes with a natural dissipation, namely, the negative-definite contributions that come from the time-derivative after the various coercivity properties have been used and the parameters $a_0,b_0$ have been suitably set. 
This quantity is given by
	\begin{align}\label{def:Ded}
	&\mathcal{D}^{e.d.}_{M,B}(t,k):= \sum_{\alpha+|\beta|\leq B}\frac{2^{-\mathtt{C}\beta} \brak{k}^\alpha}{\jap{\nu t}^{2\beta}}\bigg(\nu\cA\left[\jap{v}^{M}Z^\beta \hat{f}\right]+ \nu a_{\nu,k} \cA\left[\jap{v}^{M+q_{\gamma,s}}\nabla_vZ^\beta \hat{f}\right]\\
		\notag&\hspace{6cm}+ b_{\nu,k}|k|^2\norm{\jap{v}^{M+q_{\gamma,s}} Z^\beta \hat{f}}^2_{L^2_v}\bigg),
	\end{align}
    where $\cA[g]$ is an anisotropic norm naturally arising from the linearized operator $\cL$; see \eqref{def:A} below for the definition. 
    It is comparable in some sense to certain weighted Sobolev norms of the type $\normL{\brak{v}^{\gamma/2}g}_{H^s_v}$, see Proposition \ref{lemma:coercive} in Section \ref{sec:preliminaries} for more details.
More precisely, for the linearized problem \eqref{def:Link}, we obtain the monotonicity estimate (see Section \ref{sec:Lin}),  
\begin{align}
		\frac{\dd }{\dd t} E^{e.d.}_{M,B}+\delta_{e} \mathcal{D}^{e.d.}_{M,B}\leq 0,  \label{bd:energyED}
\end{align}
for a fixed universal constant $\delta_e > 0$.
In the case of hard potentials, i.e. $\gamma \geq 0$, one can show that
\begin{align*}
\mathcal{D}^{e.d.}_{M,B} \gtrsim \lambda_{\nu,k} E^{e.d.}_{M,B}, 
\end{align*}
and hence \eqref{bd:energyED} implies exponential decay of the type $\e^{-\lambda_{\nu,k} t}$ for the linearized problem.   
In the case of soft potentials ($\gamma < 0$), as is standard, one needs to a use a weak Poincar\'e-type approach, instead proving that for any $R> 0$ 
\begin{align*}
\mathcal{D}^{e.d.}_{M,B} \gtrsim R^{-q} \lambda_{\nu,k} E^{e.d.}_{M,B} - R^{-(q+2M')} E^{e.d.}_{M + M',B},
\end{align*}
for some $q>0$ depending on $\gamma, s$. Choosing $R$ depending on $t$ in an optimal way, one obtains then a suitable polynomial decay estimate
\begin{align*}
E^{e.d.}_{M,B} (t)\lesssim \brak{\lambda_{\nu,k} t}^{-J} E^{e.d.}_{M + M',B}(0).
\end{align*}
As usual, stretched-exponential decay can be obtained if one uses exponential or Gaussian weights; see Section \ref{subsec:decay} for more details.

\medskip
\noindent $\diamond$ \textbf{Taylor dispersion regime: $\nu/|k|> \delta_0$.}
One of the major challenges in the low frequency regime, relative to the passive scalar case or the enhanced dissipation regime considered above, 
is that we lose control on the kernel of the linearized operator $ \bP f$. Contrary to the passive scalar case \cite{CZG21}, where one can employ
a simple variation of the energy functional $E^{e.d.}_{M,B}$, we here have to combine hypocoercivity and a (quantitative) adaptation of the \textit{micro-macro} 
energy approach introduced by Guo \cite{guo2006boltzmann}, see also  Duan et al. \cite{duan2009stability,duan2011hypocoercivity} and Strain \cite{strain2012optimal}. An energy functional similar to $E^{e.d.}_{M,B}$ is still used to obtain information on microscopic quantities, i.e. $(I-\bP)f$, however a different energy yields estimates on $(\rho,\sfm,\sfe)$.  
Recall that $(\rho,\sfm,\sfe)$ (defined in \eqref{def:rhoue}) solve the following (non-closed) hydrodynamic system (see e.g. \cite{guo2006boltzmann,duan2011hypocoercivity}) 
\begin{subequations} \label{eq:HydroEqns}
\begin{align}
	\label{eq:detrho0}	&\de_t \rho + \nabla_x \cdot \sfm = 0, \\ 
	\label{eq:detm0}	&\de_t \sfm+\nabla_x\rho=-2\nabla_x\sfe-\nabla_x\cdot \Theta[(I-\bP)f],\\
	\label{eq:dte0}	&\de_t \sfe=-\frac13\nabla_x\cdot \sfm-\frac{1}{6}\nabla_x\cdot \Lambda[(I-\bP)f],
\end{align}
\end{subequations} 
where the high-order moment functions $\Theta[g]=(\Theta_{ij}[g])_{d\times d}$ and $\Lambda[g]=(\Lambda_1[g],\dots,\Lambda_d[g] )$ are defined as 
\begin{align}
	\label{def:Theta}
	&\Theta_{ij}[g]=\jap{(v_iv_j-1)\sqrt{\mu},g}_{L^2_v},\\
	\label{def:Lambda}&\Lambda_i[g]=\jap{(|v|^2-(d+2))v_i\sqrt{\mu},g}_{L^2_v}.
\end{align}
This system is used to design an appropriate energy that can transfer collisional damping from $(I - \bP )f$ to $(\rho,\sfm,\sfe)$.
Here we would like to obtain the sharp decay rate predicted by Taylor dispersion and also obtain as much Landau damping as possible. 
This motivates the introduction of the following energy functional 
\begin{align}
&E^{T.d.}_{M,B}(t,k) := \frac{1}{2}\normL{\hat{f}}^2_{L^2_v}\notag \\ & +\frac{1}{2}\sum_{|\beta|\leq B}\sum_{j=0}^1\frac{2^{-\mathtt{C}_{j}\beta} }{\brak{\nu t}^{2\beta}} \bigg(\frac{|k|}{\nu}\normL{Z^\beta \hat{f}}^2_{L^2_v}+c_1\normL{Z^\beta\nabla_v^j(I-\bP)\hat{f}}^2_{L^2_v} \notag \\ 
\label{def:ETd}
& \hspace{7cm}
+c_2\normL{\jap{v}^{M+jq_{\gamma,s}}Z^\beta \nabla_v^j(I-\bP)\hat{f}}^2_{L^2_v}\bigg)  \\
& +\frac{c_0}{\nu}\mathcal{M}+\frac{c_3}{\nu}\sum_{|\beta|\leq B}\frac{2^{-\mathtt{C}_0\beta}}{\brak{\nu t}^{2\beta}}\Re \jap{\jap{v}^{M+q_{\gamma,s}}Z^\beta (ik(I-\bP)\hat{f}),\jap{v}^{M+q_{\gamma,s}}Z^\beta\nabla_v (I-\bP)\hat{f}}_{L^2_v}\bigg), 
\end{align}
where $q_{\gamma,s}$ is defined in \eqref{def:qgammas}, $1\ll \mathtt{C}_0\ll  \mathtt{C}_{1}$, $0<c_{i+1}\ll c_i\ll 1$ with $i=0,\dots, 3$ are small universal constants determined by the proof; the term $\mathcal{M}$ is a mixed inner product involving the macroscopic variables which we define precisely below.
The energy functional in \eqref{def:ETd} contains the terms in the enhanced dissipation regime for the projection out of the kernel (namely the ones multiplied by $c_2$ and $c_3$). These terms do not give any information on  $\bP f$. To recover dissipation on $(\rho,\sfm,\sfe)$ through \eqref{eq:HydroEqns} we use 
	\begin{align}
\cM:=\,&\Re(\Lambda[(I-\bP)\hat{f}]\cdot  (ik\hat{\sfe}) +b_1\Re\big(\big(\Theta[(I-\bP)\hat{f}]+(2\hat{\sfe} I)\big):(ik\hat{\sfm}+(ik\hat{\sfm})^T)\big) \notag \\
	&+b_2\Re \big(\hat{\sfm}\cdot (ik\hat{\rho})\big), \label{def:Mkbeta}
\end{align}
where $0<b_2\ll b_1\ll 1$ are small universal constants determined by the proof. Notice that
\begin{align}
	\label{bd:equivETd}
	E^{T.d.}_{M,B}\approx\, &\normL{\hat{f}}^2_{L^2_v}+\sum_{|\beta|\leq B}\sum_{j=0}^1\frac{2^{-\mathtt{C}_{j}\beta}}{\brak{\nu t}^{2\beta}} \bigg(\frac{|k|}{\nu}\normL{Z^\beta \hat{f}}^2_{L^2_v}+\normL{Z^\beta(\nabla_v)^j(I-\bP)\hat{f}}^2_{L^2_v}\\
	&\hspace{5cm}+\normL{\jap{v}^{M+jq_{\gamma,s}}Z^\beta(\nabla_v)^j(I-\bP)\hat{f}}^2_{L^2_v}\bigg). 
	\end{align}
\begin{remark}
The idea of adding the mixed inner product \eqref{def:Mkbeta} in the energy to recover dissipation for $\bP f$ 
shares analogies with Kawashima-type energy arguments \cite{kawashima1984systems}; indeed one can use $\cM$ to explicitly build a \emph{Kawashima compensator}.
For the Boltzmann equation, this energy method was introduced by Guo  \cite{guo2006boltzmann}. 
\end{remark} 
As in the enhanced dissipation regime, for the linearized problem \eqref{def:Link}, we obtain a monotonicity estimate of the form
\begin{equation}
		\label{bd:energyTD}
		\frac{\dd }{\dd t} E^{T.d.}_{M,B}+\delta_d\mathcal{D}^{T.d.}_{M,B}\leq 0, 
	\end{equation}
for a fixed universal constant $\delta_d > 0$ and the dissipation functional 
	\begin{align}
	\notag	\mathcal{D}^{T.d.}_{M,B}:=\,&\nu\cA[(I-\bP)\hat{f}]+\sum_{|\beta|\leq B} \frac{2^{-\mathtt{C}_0\beta}}{\jap{\nu t}^{2\beta}} \nu^{-1}|k|^2\bigg(\norm{Z^\beta \bP \hat{f}}^2_{L^2_v}+\norm{\jap{v}^{M+q_{\gamma,s}}Z^\beta(I-\bP)\hat{f}}^2_{L^2_v}\bigg)\\
	\label{def:DTD}	&+\sum_{|\beta|\leq B}\sum_{j=0}^1\frac{2^{-\mathtt{C}_j\beta}}{\jap{\nu t}^{2\beta}}\bigg(|k|\cA[Z^\beta (I-\bP)\hat{f}]+\nu \cA[Z^\beta(\nabla_v)^j (I-\bP)\hat{f}]\\
		&\qquad \qquad \qquad + \nu\cA[\jap{v}^{M+jq_{\gamma,s}}Z^\beta(\nabla_v)^j (I-\bP)\hat{f}]\bigg).
	\end{align}

\begin{remark}
Notice the $\nu^{-1}$ prefactor in the last terms of \eqref{def:ETd}.
This is possible due to the power of $k$ and the uniform boundedness of $|k|/\nu$ in this range of $k$,$\nu$.
This scaling is key to obtain the Taylor dispersion estimates in Theorem \ref{thm:main}. 
\end{remark}

\begin{remark}
The factor $\jap{\nu t}^{-2\beta}$ is essential in our estimates in the regime $|k|\ll \nu$. As we show in Lemma \ref{lemma:mixed}, it allows us to obtain dissipation for $Z^\beta \bP \hat{f}$, whose control is necessary for the  `Landau damping'. Having the factor $\jap{\nu t}^{-2\beta}$ in the energy removes the effect of the phase mixing for $t\gg \nu^{-1}$, though this is a time-scale much larger than the one in which we see the dissipation enhancement. In Remark \ref{rem:nutbeta}, we comment about another natural strategy to get dissipation for $Z^\beta \bP \hat{f}$, which however still requires the factor $\jap{\nu t}^{-2\beta}$. We do not know if this is sharp, but it might be an effect related to frequencies $|k|\ll \nu$ where the phase mixing is very weak.
\end{remark}

\medskip
\noindent $\diamond$ \textbf{Linear decay estimates.}
In Section \ref{sec:Lin} we prove the following estimates on the linearized problem. 

\begin{theorem} \label{thm:LinDecEst}
Let $\hat{f}$ solve \eqref{def:Link} and $E^{e.d.}_{M,B}$, $E^{T.d.}_{M,N}$, $\lambda_{\nu,k}$, and $q_{\gamma,s}$ be defined as in \eqref{def:Eed}, \eqref{def:ETd}, \eqref{def:lambdanuk_intro} and \eqref{def:qgammas} respectively. 
Then, the monotonicity estimates \eqref{bd:energyED} and \eqref{bd:energyTD} both hold. 
Define,
\begin{align}
&\mathcal{E}_{M,B}(t,k):= \mathbbm{1}_{\nu/|k|\leq\delta_0}E^{e.d.}_{M,B}(t,k)+\mathbbm{1}_{\nu/|k|>\delta_0}E^{T.d.}_{M,B}(t,k), \label{def:linEnergy} \\  
&\mathcal{D}_{M,B}(t,k):= \mathbbm{1}_{\nu/|k|\leq\delta_0}\mathcal{D}^{e.d.}_{M,B}(t,k)+\mathbbm{1}_{\nu/|k|>\delta_0}\mathcal{D}^{T.d.}_{M,B}(t,k). \label{def:linDiss} 
\end{align}
Then, for any $M>2\max\{|q_{\gamma,s}|,|\gamma|+2s\}$ and $M' > 1$ there holds
 \begin{align}
 	\label{bd:lindecay} \mathcal{E}_{M,B}(t,k)\lesssim \e^{-\delta_p(\lambda_{\nu,k} t)^{1-p}}\mathcal{E}_{M,B}(0,k)+  \frac{1}{\brak{\lambda_{\nu,k} t}^{\widetilde{M}}}\mathcal{E}_{M+M',B}(0,k)
 \end{align}
where $p\in(0,1)$ is a given number, $0<\delta_p<1$ and $\displaystyle \widetilde{M}=\frac{p(1+s)(|\gamma+2s|+2M')}{|\gamma|(2-s)+2s|q_{\gamma,s}|}-p$.
\end{theorem}
\begin{remark}
By a straightforward variant, one can obtain stretched exponential decay estimates as in Theorem \ref{thm:Td} using Gaussian localization estimates as in \eqref{ineq:2xexp}.
These are omitted as this is not possible for $x \in \mathbb R^d$, which is the main focus of this work. 
\end{remark}

Before we continue to the nonlinear problem, let us briefly explain now how to pass from the distribution function estimates to the estimates on hydrodynamic quantities.
In particular, the decay estimate \eqref{bd:lindecay} implies all of the claims in Theorem \ref{thm:main} for the linearized problem \eqref{def:Link}. 
\begin{lemma} \label{lem:LDLF}
Let $\zeta \geq 0$, $\eta > d$, and $N$ be integers. 
For $j \leq \eta$, suppose that $b$ satisfies $\abs{\grad_v^j b(v)} \lesssim_j \brak{v}^N$.
Let $d \geq 2$ and let $f$ satisfy $\sup_{t \geq 0} \mathcal{E}_{N+\eta,\eta}(t,k) < \infty$.
Then, there holds 
\begin{align*}
\norm{\brak{\lambda(\nu,\grad_x) t}^\zeta \int f(t,\cdot,v) b(v) \dd v}_{L^\infty_x} \lesssim_{\zeta,\eta} \left(\frac{1}{\brak{t}^d} + \frac{\nu^{d/2}}{\brak{t}^{d/2}} \right) \sup_{t \geq 0} \sup_{k \in \mathbb R^d} \brak{\lambda_{\nu,k} t}^{\eta + \zeta} \sqrt{\mathcal{E}_{N+\eta,\eta}(t,k)}. 
\end{align*}
\end{lemma}
\begin{proof}
The case $\zeta=0$ can easily be extended to handle $\zeta > 0$, and hence we only consider the following 
\begin{align*}
\norm{\int f(t,\cdot,v) b(v) \dd v}_{L^\infty_x} & \lesssim \int_{\mathbb R^d} \abs{\int\hat{f}(t,k,v)b(v) \dd v} \dd k \\
& \lesssim \int_{\abs{k} \leq \delta_0^{-1} \nu } \abs{\int \hat{f}(t,k,v) b(v) \dd v} \dd k + \int_{\abs{k} > \delta_0^{-1} \nu } \abs{\int\hat{f}(t,k,v) b(v) \dd v} \dd k \\
& \lesssim \left(\int_{\abs{k} \leq \delta_0^{-1} \nu } \frac{1}{\brak{ \nu^{-1} \abs{k}^2 t }^{\eta} } \dd k \right) \sup_{k \in \mathbb R^d} \brak{\lambda_{\nu,k} t}^{\eta} \sqrt{\mathcal{E}_{M+\eta,0}(t,k)} \\ 
& \quad + \int_{\abs{k} > \delta_0^{-1} \nu } \frac{\brak{\nu t}^\eta}{\abs{tk}^{\eta} \brak{\lambda_{\nu,k} t}^\eta }\abs{\int \brak{\lambda_{\nu,k} t}^\eta \frac{Z^\eta}{\brak{\nu t}^\eta} \left( \hat{f}(t,k,v) b(v) \right) \dd v} \dd k \\
& \lesssim \left(\frac{\nu^{d/2}}{\brak{t}^{d/2}} + \frac{1}{\brak{t}^d}\right) \sup_{k \in \mathbb R^d} \brak{\lambda_{\nu,k} t}^\eta \sqrt{\mathcal{E}_{M+\eta,\eta}(t,k)}, 
\end{align*}
where in the last line we used that $\lambda_{\nu,k} \gtrsim \nu$ for $\abs{k} \gtrsim \nu$ and we used $\eta > d$. 
\end{proof}

In a similar vein, let us record the following $ L^2_v L^\infty_x$ decay estimate which follows from Theorem \ref{thm:LinDecEst}. 
\begin{lemma}
  Let $\zeta \geq 0$, $\eta > d$, and $N$ be integers. 
For $j \leq \eta$, suppose that $b$ satisfies $\abs{\grad_v^j b(v)} \lesssim_j \brak{v}^N$.
Let $d \geq 2$ and let $f$ satisfy $\sup_{t \geq 0} \mathcal{E}_{N+d,\zeta}(t,k) < \infty$.
Suppose further that $\frac{2s}{1+2s} \eta > d$. Then, 
\begin{align*}
  \norm{\brak{\lambda(\nu,\grad_x)}^\zeta b f}_{L^2_v L^\infty_x } \lesssim \left(\frac{\nu^{d/2}}{\brak{t}^{d/2}} + \frac{1}{\brak{ \nu t^{1+2s} }^{d/2s}} \right) 
  \sup_{k \in \mathbb R^d} \brak{\lambda_{\nu,k} t}^{\eta+\zeta}  \sqrt{\mathcal{E}_{N+\eta}(t,k)}.  
\end{align*}
\end{lemma} 
\begin{proof}
The $\zeta > 0$ case is a trivial extension of the $\zeta = 0$ case, and so we just consider this one. 
For $t \lesssim \nu^{-\frac{1}{1+2s}}$, we may simply use Sobolev embedding in $x$ and obtain an estimate with no time-decay.
Therefore, consider  $t \gtrsim \nu^{-\frac{1}{1+2s}}$.
By Cauchy-Schwarz and Fubini's theorem, 
\begin{align*}
\norm{b f}_{L^2_v L^\infty_x }^2 & \lesssim  \int_{\mathbb R^d} \abs{ \int_{\mathbb R^d} b(v) \hat{f}(t,k,v) \dd k }^2 \dd v  \\ 
& \lesssim \int_{\mathbb R^d} \int _{\mathbb R^d} \brak{v}^{2N+\eta}\abs{\hat{f}(t,k,v)}^2 \dd v \dd k \\  
& \lesssim \left(\int_{\abs{k} <  \nu/\delta_0} \frac{1}{\brak{\nu^{-1} |k|^2 t}^\eta} \dd k + \int_{\abs{k} \geq  \nu/\delta_0} \frac{1}{\brak{\nu^{\frac{1}{1+2s}} |k|^{\frac{2s}{1+2s}} t}^\eta} \dd k \right)\sup_{k \in \mathbb R^d}\brak{\lambda_{\nu,k} t}^\eta 
\sqrt{\mathcal{E}_{N+\eta}(t,k)}, 
\end{align*}
from which the desired estimate follows. 
\end{proof}

\subsection{Nonlinear problem}

For the nonlinear problem, we need also to estimate $\nu \Gamma(f,f)$, which now couples all of the $x$-frequencies together.
The norms we will use to treat the nonlinear problem are the following: for integers $M,M',M_{J'} > 2d$, $B > B' + d > 3d$  we define
\begin{subequations} \label{def:E}
\begin{align}
\mathcal{E}(t) & := \int_{\abs{k} \leq \delta_0^{-1} \nu} \brak{\lambda_{\nu,k} t}^{2J} E^{T.d.}_{M,B}(t,k) \dd k +  \int_{\abs{k} > \delta_0^{-1} \nu} \brak{\lambda_{\nu,k} t}^{2J} E^{e.d.}_{M,B}(t,k) \dd k   \\ 
\mathcal{E}_{mom}(t) & := \int_{\abs{k} \leq \delta_0^{-1} \nu} E^{T.d.}_{M+M',B}(t,k) \dd k +  \int_{\abs{k} > \delta_0^{-1} \nu} E^{e.d.}_{M+M',B}(t,k) \dd k  \\ 
\mathcal{E}_{LF}(t) & := \sup_{k : \abs{k} \leq \delta_0^{-1} \nu} \brak{\lambda_{\nu,k} t}^{2J'} E^{T.d.}_{M',B'}(t,k) + \sup_{k : \abs{k} > \delta_0^{-1} \nu} \brak{\lambda_{\nu,k} t}^{2J'} E^{e.d.}_{M',B'}(t,k)  \\
\mathcal{E}_{LF,mom}(t) & := \sup_{k : \abs{k} \leq \delta_0^{-1} \nu} E^{T.d.}_{M' + M_{J'} ,B'}(t,k) + \sup_{k : \abs{k} > \delta_0^{-1} \nu} E^{e.d.}_{M' + M_{J'} ,B'}(t,k), 
\end{align}
\end{subequations}
where the parameters are set satisfying certain conditions determined by the proof, namely
\begin{align}
\notag
&M, M'> 2\max\{ (|\gamma|(2-s)+2s|q_{\gamma,s}|)/(1+s), B+|\gamma|+2s\},\\
& M_{J}, M_{J'}> 2\max\{ (|\gamma|(2-s)+2s|q_{\gamma,s}|)/(1+s), B+|\gamma|+2s\}(J+1), \\
 \label{eq:constMJB} & M'+M_{J'}\leq M, \quad 2d<B'< B-1, \quad d<J'<J-1.  
\end{align}
As above, each of these energies are also associated with a natural dissipation functional $\mathcal{D}_\ast$; see \eqref{def:Diss} and Section \ref{sec:Lin}.
The energy $\mathcal{E}$ represents the fundamental $L^2_xL^2_v$ decay estimates we obtain, while the energy $\mathcal{E}_{mom}$ controls higher moments in $v$ but with no decay estimates (to use a weak Poincar\'e approach in the treatment of soft potentials) and $\mathcal{E}_{LF}$ obtains better estimates at low frequencies using $L^\infty_k L^2_v$ (i.e. pointwise-in-frequency) as a more convenient surrogate for $L^1_xL^2_v$-type estimates. This allows to obtain the sharp decay estimates for both $f$ and $\rho$; these sharp decay estimates are also crucial to close the nonlinear argument in $d=2$ (however in $d \geq 3$, they can be obtained a posteriori). 

The majority of the paper is devoted to the following bootstrap estimate, which implies Theorem \ref{thm:main} by a straightforward regularization argument. 
\begin{proposition}
\label{prop:bootstrap}
Let $f$ be a classical solution to \eqref{eq:PertEqn} such that $\brak{v}^N \partial_{x}^\alpha \partial_v^\beta f \in L^2_{x,v}$ for all $t \in [0,T]$ and $N,\alpha,\beta \geq 0$.
Suppose that for $t \in [0,T]$ there holds
\begin{align}
\label{bd:bootEd}\mathcal{E}(t) + \nu\int_0^t \mathcal{D}(\tau) \dd\tau & \leq 4\mathsf{B}_0 \eps^2 \\
\label{bd:bootEmomD}\mathcal{E}_{mom}(t) + \nu\int_0^t \mathcal{D}_{mom}(\tau) \dd\tau & \leq 4 \mathsf{B}_1 \eps^2 \\
\label{bd:bootELF}\mathcal{E}_{LF}(t) + \nu\int_0^t \mathcal{D}_{LF}(\tau) \dd\tau & \leq 4\mathsf{B}_2 \eps^2 \\ 
\label{bd:bootELFmom}\mathcal{E}_{mom,LF}(t) + \nu\int_0^t \mathcal{D}_{mom,LF}(\tau) \dd\tau & \leq 4\mathsf{B}_3 \eps^2,
\end{align}
for universal constants $\mathsf{B}_j$ set by the proof.
Then for $\eps$ sufficiently small (not depending on $t$ or $\nu$), the same estimates hold with $4$ replaced with $2$.
Furthermore, the quantities on the left-hand side take values continuously in time, and therefore, these estimates hold for all $t \in [0,\infty)$. 
\end{proposition}
Moving from the linearized to the nonlinear problem is not too difficult on $\mathbb T^d$, as the nonlinearity $\nu \Gamma(f,f)$ has a power of $\nu$ in front  admits estimates of the general form
\begin{align}
\abs{\brak{Z^{\beta} \partial_x^\alpha \Gamma, Z^{\beta} \partial_x^\alpha f}} \lesssim \sqrt{\mathcal{E}} \mathcal{D}, \label{ineq:EDest}
\end{align}
where $\mathcal{E}$ is one of the energy functionals and $\mathcal{D}$ denotes one the associated dissipation operators. 
This allows it to be absorbed by the dissipation in a relatively straightforward manner.
However, on $\mathbb R^d$, several new difficulties arise connected with the low frequencies if one wants to obtain the sharp decay rates and almost-uniform Landau damping. 
Instead of \eqref{ineq:EDest}, we eventually reduce ourselves to an estimate of the form 
\begin{align*}
\frac{\dd}{\dd t}\mathcal{E} + \delta_\star \mathcal{D} \lesssim \sqrt{\mathcal{E}} \mathcal{D} + \left(\nu^d \mathbf{1}_{t \leq \nu^{-1} }+ \frac{\nu^{d/2}}{\brak{t}^{d/2}} \mathbf{1}_{t > \nu^{-1}}\right) \sqrt{\mathcal{E}_{LF}} \sqrt{\mathcal{E}} \sqrt{\mathcal{D}}. 
\end{align*}
Coupled with suitable estimates on $\mathcal{E}_{LF}$ and $\mathcal{E}_{mom}$ provided $d \geq 2$, this implies Proposition \ref{prop:bootstrap} and hence Theorem \ref{thm:main}.

\subsection{Notation}
Let $\beta=[\beta_1,\dots,\beta_d]$ be a multi-index with $\beta_i \in \mathbb{N}$ and $|\beta|=\beta_1+\dots +\beta_d$. We define
\begin{align}
\label{def:Zbetamulti}
&Z^\beta=(\de_{v_1}+t\de_{x_1})^{\beta_1}\dots (\de_{v_d}+t\de_{x_d})^{\beta_d}, \\
&(\nabla_v)^\beta=\de_{v_1}^{\beta_1}\dots \de_{v_d}^{\beta_d},\\
&(\nabla_x)^\beta=\de_{x_1}^{\beta_1}\dots \de_{x_d}^{\beta_d}.
\end{align}
We denote $\brak{\nabla_x}$ as the operator whose symbol in the Fourier space is $\brak{k}=(1+|k|^2)^\frac12$. For a given constant $c>0$ and a multi-index $\beta$, we will slightly abuse in notation by writing $c^\beta$ instead of $c^{|\beta|}$.

The $L^2_v(\mathbb{R}^d)$ inner product inner product and norm are
\begin{equation}
\brak{g,h}_{L^2_v}=\int_{\mathbb{R}^d}g(v)\bar{h}(v)\dd v, \qquad \norm{g}^2_{L^2_v}=\brak{g,g}_{L^2_v}.
\end{equation}
We use the following notation for weighted Sobolev spaces 
\begin{equation}
\norm{g}_{H^s_{v,q}}=\norm{\brak{v}^qg}_{H^s_v},
\end{equation}
where $H^s_v$ is the usual $H^s(\mathbb{R}^d)$ Sobolev space. When no confunsion arise, we will omit the subscript $v$ for the weighted spaces. 

	We denote the $x$-Fourier transform as 
	\begin{equation}
		\cF(g)(k)=\hat{g}(k)=\frac{1}{(2\pi)^{d/2}}\int_{\RR^d}\e^{-ix\cdot k}g(x)\dd x.
	\end{equation}
	Given $\chi \in C^{\infty}_c(B_2(0))$ with $\chi(k)=1$ for $|k|\leq 1$ we define the homogeneous Littlewood-Paley decomposition in $\mathbb{R}^d$ as 
	\begin{equation}
	\label{def:para0}
		g=\sum_{N\in 2^{\ZZ}}g_N
	\end{equation}
	where $2^{\ZZ}=\{2^{j}:\,j\in \mathbb{Z}\}$ is the set of dyadic numbers and 
	\begin{equation}
		\widehat{g_N}(k)=(\chi(k/N)-\chi(2k/N))\hat{g}(k).
	\end{equation}
	We denote
	\begin{align}
	&	\widehat{P_{\leq N} g}(k)=\widehat{g_{\leq N}}(k)=\chi(k/N)\hat{g}(k),\\	&\widehat{P_{> N} g}(k)=\widehat{g_{> N}}(k)=(1-\chi(k/N))\hat{g}(k).
\end{align}
Moreover 
\begin{equation}
	\label{def:para1}
\norm{g}_{L^2_x}=\norm{\hat{g}}_{L^2_k}\approx \sum_{N\in 2^{\ZZ}}\norm{\widehat{g_N}}_{L^2_k}.
\end{equation}

\section{Preliminaries}
\label{sec:preliminaries}
In this section, we recall some known results and we prove basic estimates that play a crucial role in our subsequent analysis.

\subsection{Linearized operator}
A well-known coercive estimate for the linearized operator, given in \cite{mouhot2006explicit}, is
\begin{equation}
	\jap{\cL g,g}_{L^2_v}\geq C\norm{(I-\bP)g}_{L^2_{v,\gamma/2}}^2.
\end{equation}
This estimate shows a dissipative mechanism of the linearized operator for elements outside the kernel. However, it does not quantify any regularization property, which are in fact carried over by $\cL_1=-\Gamma(\sqrt{\mu},\cdot)$.
In particular, by standard symmetrization arguments one can show
\begin{align}
	\label{eq:L1ff}
 \jap{\cL_1g,g}_{L^2_v}=\frac12 \iint_{\RR^{2d}\times \SS^{d-1}}B\left(\sqrt{\mu_* \mu'_*}\left(g'-g\right)^2+g^2_*\left(\sqrt{\mu}-\sqrt{\mu'}\right)^2\right)\dd v_*\dd v \dd \sigma.
\end{align}
This identity suggests the following splitting, used for instance in \cite{alexandre2011global,alexandre2012boltzmann}, 
\begin{equation}
	\label{eq:splitcL}
	\jap{\cL g,g}_{L^2_v}=\cA [g]+\cK[g],
\end{equation}
where 
\begin{align}
	\label{def:A}
	\cA [g]&=\frac12 \iint_{\RR^{2d}\times \SS^{d-1}}B\left(\mu_*\left(g'-g\right)^2+g^2_*\left(\sqrt{\mu'}-\sqrt{\mu}\right)^2\right)\dd v_*\dd v \dd \sigma\\
		\label{def:K}
	\cK[ g]&=\jap{\cL_2 g,g}_{L^2_v}+\frac12 \iint_{\RR^{2d}\times \SS^{d-1}}B\sqrt{\mu_*}\left(\sqrt{\mu'_*}-\sqrt{\mu_*}\right)\left(g'-g\right)^2\dd v_*\dd v \dd \sigma.
\end{align}
From the results in \cite{alexandre2012boltzmann}, which we recall below, we know that the operator $\cA$ contains the information about the anisotropic dissipation and is comparable to some weighted Sobolev norms. On the other hand, $\cK$ can be thought as a compact perturbation. 
In the sequel, we also need to use the following variation of $\cA$ 
\begin{equation}
	\label{def:Arho}
	\cA^\kappa [f]=\frac12 \iint_{\RR^{2d}\times \SS^{d-1}}B\left(\mu_*^\kappa\left(f'-f\right)^2+f^2_*\left(\left(\mu'\right)^{\kappa/2}-\mu^{\kappa/2}\right)^2\right)\dd v_*\dd v \dd \sigma
	\end{equation}
which is defined for any $\kappa>0$.

In the next proposition we collect some bounds that are proved in \cite[Proposition 2.1-2.2, Lemma 2.12 and Lemma 2.15]{alexandre2012boltzmann}
(in the case of the Landau collision operator, i.e. $s=1$ and $\gamma=-3$, analogues can be found in \cite{guo2002landau}). 
\begin{proposition}
	\label{lemma:coercive}
	Let $0<s<1$, $\gamma>-3$ and $\kappa>0$. Then
	\begin{align}
	\label{bd:coerL} &	\cA[(I-\bP)g]\lesssim \jap{\cL g,g}_{L^2_v}\leq 2\jap{\cL_1g,g}_{L^2_v}\lesssim \cA[g],\\
	\label{bd:equivA} &\norm{g}^2_{H^s_{v,\gamma/2}}+\norm{g}_{L^2_{v,s+\gamma/2}}^2\lesssim \cA[g]\lesssim \norm{g}^2_{H^s_{v,s+\gamma/2}},\\
\label{bd:Arho}&	\cA^\kappa[g]\lesssim \cA[g].
	\end{align}
Moreover, there exists a constant $0<\delta<1$ such that 
\begin{equation}
\label{bd:L2gh}
\jap{\cL_2g,h}\lesssim \normL{\mu^{\delta} g}_{L^2_v}\normL{\mu^{\delta} h}_{L^2_v}.
\end{equation}
\end{proposition}
In fact, an important bound we need, which is one of the main ingredients to prove \eqref{bd:coerL}-\eqref{bd:equivA}, was given in \cite[Proposition 2.16]{alexandre2012boltzmann}.
\begin{proposition}
	\label{prop:equivL1}
	Let $\gamma>-3$. There exists a constant $C>0$ such that
	\begin{equation}
		\frac{1}{10}\cA[g]-C\norm{g}^2_{L^2_{v,\gamma/2}}\leq \jap{\cL_1 g,g}_{L^2_v}\leq \cA[g].
	\end{equation}
\end{proposition}

\subsubsection{Commutation properties}
Since we are going to use an energy method, commutator estimates of the linearized operator with derivatives and weights are crucial.
For standard derivatives, these were already understood in e.g. \cite{guo2006boltzmann,gressman2011global,alexandre2011global,guo2012decay}. 
However, we need to be sure of analogous properties involving the vector field $Z$. Introducing the trilinear operator
\begin{align}
	\label{def:T}
	\cT(g,h,q)=\iint_{\RR^d\times \SS^{d-1}}B(v-v_*,\sigma)q_*\left(g'_*h'-g_*h\right)\dd v_*\dd \sigma,
\end{align}
it is straightforward to check that 
\begin{equation}
	Z(\cT(g,h,q))=\cT(Zg,h,q)+\cT(g,Zh,q)+\cT(g,h,Zq).
\end{equation}
Since $\Gamma(g,h)=\mathcal{T}(g,h,\sqrt{\mu})$, recalling the multi-index notation \eqref{def:Zbetamulti}, we have the Leibniz formula 
\begin{equation}
	\label{eq:ZLeib}
	Z^{\beta}(\Gamma(g,h))=\sum_{|\beta_1|+|\beta_2|+|\beta_3|=|\beta|}C_{\beta_1,\beta_2,\beta_3}\cT\left(Z^{\beta_1}g,Z^{\beta_2}h,(\nabla_v)^{
\beta_3}\sqrt{\mu}\right),
\end{equation}
where we also used that $\mu$ does not depend on $x$. Then we have the following adaptation of estimates obtained in \cite{gressman2011global,alexandre2011global} for the standard derivatives.
\begin{lemma}
	\label{lemma:commutation}
There exist constants $C_{1},C_2>0$ such that for any $\widetilde{\delta}>0$ the following holds true:
	\begin{align}
	\label{bd:commZ}|\langle{[\cL_1,Z^\beta]g,h}\rangle_{L^2_v}|&\leq \widetilde{\delta} \cA[h]+\frac{C_1}{\widetilde{\delta}}\sum_{|\beta_1|\leq |\beta|-1}\cA[Z^{\beta_1}g] \\
	\label{bd:commZL2} |\langle{[\cL_2,Z^\beta]g,h}\rangle_{L^2_v}|&\leq \widetilde{\delta} \normL{\mu^{\frac{1}{C(\beta)}}h}_{L^2_v}^2+\frac{C_2}{\widetilde{\delta}}\sum_{|\beta_1|\leq |\beta|-1}\normL{\mu^{\frac{1}{C(\beta)}}Z^\beta g}^2_{L^2_{v}},
\end{align}
where $C(\beta)=10^4(|\beta|+1)$. Moreover, let $\ell \geq 0$. There exists $0<\delta_c<1, \, C_3, C_4>0$ such that 
\begin{align}
	\label{bd:commWl} 
		|\langle{[\cL_1,\brak{v}^\ell]g,\brak{v}^\ell g}\rangle_{L^2_v}|&\lesssim \normL{\mu^{\delta_c} g}_{L^2}^2\\
	\label{bd:lowerWZ}
	\jap{\brak{v}^\ell Z^\beta(\cL g),\brak{v}^\ell Z^\beta g}&\geq \frac{1}{100}\cA[\brak{v}^\ell Z^\beta g]-C_3\sum_{|\beta_1|\leq \beta-1}\cA[\brak{v}^\ell Z^{\beta_1}g]-C_4\normL{\mu^{\delta_c} Z^\beta g}^2_{L^2_v}.
\end{align}
The same bound holds with $Z$ replaced by $\nabla_v$.
\end{lemma}
\begin{proof}
We first prove the bounds \eqref{bd:commZ} and \eqref{bd:commWl} involving $\cL_1$.		Recall that 
$$
\cL_1g=-\Gamma(\sqrt{\mu},g)=-\cT(\sqrt{\mu},g, \sqrt{\mu}).
$$
 Combining the Leibniz formula \eqref{eq:ZLeib} with the fact that $\mu$ does not depend on $x$, we get
		\begin{equation}
			 \label{bd:leibcomm}[\cL_1,Z^\beta g]=\sum_{\substack{|\beta_1|+|\beta_2|+|\beta_3|=\beta\\
			 |\beta_2|\leq |\beta|-1}}C_{\beta_1,\beta_2,\beta_3} \cT\left((\nabla_v)^{\beta_1}\sqrt{\mu},Z^{\beta_2}g,(\nabla_v)^{\beta_3}\sqrt{\mu}\right).
		\end{equation}
	We thus have to estimate terms like
    \begin{equation}
    I =\jap{\cT\left((\nabla_v)^{\beta_1}\sqrt{\mu},Z^{\beta_2}g,(\nabla_v)^{\beta_3}\sqrt{\mu}\right),h}.
    \end{equation}
From this point on, having $Z$-derivatives is the same as having $(x,v)$-derivatives in the proofs of the commutator estimates in \cite{gressman2011global,alexandre2011global}. For convenience of the reader, we present the main ideas to obtain such bounds. Notice that $\de_{v_i}^{\beta_j}\sqrt{\mu}=P_{\beta_j}(v_i)\sqrt{\mu}$ for a polynomial $P_{\beta_j}$. Hence, with a slight abuse in notation we rewrite the term above as 
\begin{equation}
	I =\iint_{\RR^{2d}\times \SS^{d-1}}B(P_{\beta_3}\sqrt{\mu})_*((P_{\beta_1}\sqrt{\mu})_*'Z^{\beta_2} g'-(P_{\beta_1}\sqrt{\mu})_*Z^{\beta_2} g)h\, \dd v\dd v_*\dd \sigma.
\end{equation}
We split this term as $I =	I^1+	I^2$, where
\begin{align}
	I ^1&=\iint_{\RR^{2d}\times \SS^{d-1}} B(P_{\beta_3}\sqrt{\mu})_*(P_{\beta_1}\sqrt{\mu})_*'(Z^{
\beta_2} g'-Z^{\beta_2}g)h\, \dd v\dd v_*\dd \sigma,\\
	\label{def:I2}I ^2&=\iint_{\RR^{2d}\times \SS^{d-1}}B(P_{\beta_3}\sqrt{\mu})_*((P_{\beta_1}\sqrt{\mu})_*'-(P_{\beta_1}\sqrt{\mu})_*)(Z^{\beta_2}g)h\, \dd v\dd v_*\dd \sigma.
\end{align}
For $I^2$, combining Lemma \ref{lemma:exp} with Lemma \ref{lemma:angular} and using Cauchy-Schwarz we get
\begin{equation}
	|I^2|\lesssim \normL{Z^{\beta_2}g}_{L^2_{v,s+\gamma/2}}\norm{h}_{L^2_{v,s+\gamma/2}}.
\end{equation}
 Concerning $I ^1$, we consider the symmetric and antysimmetric part of the coefficients involving the Maxwellian, namely $I^1=I^{1}_{\mathrm{sym}}+I^{1}_{\mathrm{anti}}$ with
\begin{align}
	I^{1}_{\mathrm{sym}}&=\frac12\iint_{\RR^{2d}\times \SS^{d-1}}BS(Z^{\beta_2}g'-Z^{\beta_2}g)h\, \dd v\dd v_*\dd \sigma,\\
	\label{def:S}S:&=(P_{\beta_3}\sqrt{\mu})_*(P_{\beta_1}\sqrt{\mu})_*'+(P_{\beta_3}\sqrt{\mu})_*'(P_{\beta_1}\sqrt{\mu})_*,\\
		I^{1}_{\mathrm{anti}}&=\frac12\iint_{\RR^{2d}\times \SS^{d-1}}BA(Z^{\beta_2}g'-Z^{\beta_2}g)h\, \dd v\dd v_*\dd \sigma,\\
		\label{def:Anti}A:&=(P_{\beta_3}\sqrt{\mu})_*(P_{\beta_1}\sqrt{\mu})_*'-(P_{\beta_3}\sqrt{\mu})_*'(P_{\beta_1}\sqrt{\mu})_*. 
\end{align}
For the symmetric part, since $S$ and $B$ are invariant under the change $(v,v_*)\to (v',v_*')$, we have 
\begin{equation}
I^{1}_{\mathrm{sym}}=\frac14\iint_{\RR^{2d}\times \SS^{d-1}}BS(Z^{\beta_2}g'-Z^{\beta_2}g)(h-h')\, \dd v\dd v_*\dd \sigma.
\end{equation}
To bound $S$, notice that for some $0<\kappa<1$ (depending on $\beta_3,\beta_1$) we have 
\begin{equation}
	|S|\lesssim \left(\mu_*\mu_*'\right)^{\kappa/2}=\left(\mu\mu'\right)^{\kappa/2}=\mu^\kappa+\mu^{\kappa/2}\left(\left(\mu'\right)^{\kappa/2}-\mu^{\kappa/2}\right).
\end{equation}
Hence, by the Cauchy-Schwarz inequality and the definition of $\cA^\kappa$ in \eqref{def:Arho}, we infer
\begin{equation}
	|I^{1}_{\mathrm{sym}}|\lesssim \cA^\kappa[Z^{\beta_2} g]\cA^\kappa[h]\lesssim \cA[Z^{\beta_2} g]\cA[h],
\end{equation}
where we used \eqref{bd:Arho} in the last inequality. To handle $I^1_{\mathrm{anti}}$, one exploits the fact that under the assumptions in the kernel, the change $v\to v'$ does not change the bounds. Thus, introducing further commutators, one is in a situation analogous to $I_2$. With the bounds above, appealing to Lemma \ref{lemma:coercive} and using Young's inequality, the commutator estimate \eqref{bd:commZ} is proved. 

Turning our attention to \eqref{bd:commWl}, notice that 
\begin{align}
	\jap{[\cL_1,\brak{v}^\ell]g,\brak{v}^\ell g}&=\iint_{\RR^{2d}\times \SS^{d-1}}B\sqrt{\mu_* \mu_*'}g'g(\brak{v'}^\ell-\brak{v}^\ell)\brak{v}^\ell \, \dd v \dd v_* \dd \sigma\\
	&=-\frac12 \iint_{\RR^{2d}\times \SS^{d-1}}B(\mu_* \mu_*')^\frac14 (\mu^\frac14g)'(\mu^\frac14g)(\brak{v'}^\ell-\brak{v}^\ell)^2 \, \dd v \dd v_* \dd \sigma,
\end{align}
where in the last identity we used $\mu_*\mu_*'=\mu\mu'$ and we symmetrized using $v\to v'$. Since $2(\mu^\frac14 g)'(\mu^\frac14 g)\leq ((\mu^\frac14 g)')^2+(\mu^\frac14 g)^2$, exchanging $v\to v'$ in the resulting integral of the first term, we get 
\begin{equation}
		\label{bd:commWL}
	\left|\jap{[\cL_1,\brak{v}^\ell]g,\brak{v}^\ell g}\right|\lesssim \iint_{\RR^{2d}\times \SS^{d-1}}B\left(\mu_* \mu_*'\right)^\frac14 \left(\mu^\frac14g\right)^2\left(\brak{v'}^\ell-\brak{v}^\ell\right)^2 \, \dd v \dd v_* \dd \sigma. 
\end{equation}
By the mean value theorem
\begin{equation}
	\label{bd:mvW}
	|\brak{v'}^\ell-\brak{v}^\ell|\lesssim |v-v'|(\brak{v'}^{\ell-1}+\brak{v}^{\ell-1}).
\end{equation}
Arguing as in \eqref{eq:v*-v*'}, one has
	\begin{equation}
		\label{bd:anvv'}
	|v-v'|\lesssim\sin(\theta/2)|v-v_*|.
\end{equation}
Therefore, combining \eqref{bd:mvW} and \eqref{bd:anvv'} with Lemma \ref{lemma:exp} and Lemma \ref{lemma:angular}, from \eqref{bd:commWL} we infer 
\begin{align}
		|\jap{[\cL_1,\brak{v}^\ell] g,\brak{v}^\ell g}|\lesssim \normL{\mu^{\delta_c} g}_{L^2}^2,
\end{align}
where $\delta_c>0$ is sufficiently small.

To prove \eqref{bd:commZL2}, recalling that $\cL_2g=-\Gamma(g,\sqrt{\mu})=-\cT(g,\sqrt{\mu},\sqrt{\mu})$, we get
 \begin{equation}
 	\label{bd:leibcommL2}[\cL_2,Z^\beta f]=\sum_{\substack{|\beta_1|+|\beta_2|+|\beta_3|=|\beta|\\
			 |\beta_2|\leq |\beta|-1}}\cT\left(Z^{\beta_1}f,(\nabla_v)^{\beta_2}\sqrt{\mu},(\nabla_v)^{\beta_3}\sqrt{\mu}\right).
 \end{equation}
Consequently, we can proceed as in the proof of \cite[Proposition 4.5]{alexandre2011global} (and \cite[Lemma 2.15]{alexandre2011global}) to prove \eqref{bd:commZL2}. More precisely, one has to bound terms like 
\begin{equation}
	II =\iint_{\RR^{2d}\times \SS^{d-1}}B(P_{\beta_3}\sqrt{\mu})_*\left(Z^{\beta_1}g'_*(P_{\beta_2}\sqrt{\mu})'-Z^{\beta_1}g_*(P_{\beta_2}\sqrt{\mu})\right)h\, \dd v\dd v_*\dd \sigma,
\end{equation}
where $P_\beta$ are again suitable polynomials arising from the derivatives of the Maxwellian.
Split $II=II^1+	II^2$  as 
\begin{align}
	II^1&=\iint_{\RR^{2d}\times \SS^{d-1}} B(Z^{\beta_1} g)'_*((P_{\beta_3}\sqrt{\mu})_*(P_{\beta_2}\sqrt{\mu})'-(P_{\beta_3}\sqrt{\mu})_*'(P_{\beta_2}\sqrt{\mu}))h\, \dd v\dd v_*\dd \sigma,\\
	\label{def:II2}II^2&=\iint_{\RR^{2d}\times \SS^{d-1}}B((P_{\beta_3}\sqrt{\mu}Z^{\beta_1}g)'_*-(P_{\beta_3}\sqrt{\mu}Z^{\beta_1}g)_*)(P_{\beta_2}\sqrt{\mu}h)\, \dd v\dd v_*\dd \sigma.
\end{align}
For the term $II^2$ one exploits the cancellation in \cite{alexandre2000entropy}*{Lemma 1} which says  that 
\begin{equation}
	\int_{\RR^{2d}\times \SS^{d-1}}B(g_*'-g_*)\dd v_*\dd \sigma= (K* g)(v), \qquad K\approx |v|^\gamma.
\end{equation}
For the term $II^1$ it is enough to split the term containing the Maxwellians to apply the mean value theorem together with Lemmas \ref{lemma:exp} and \ref{lemma:angular}.

Finally, to prove \eqref{bd:lowerWZ}, notice that 
\begin{align}
			\brak{v}^\ell Z^\beta \cL g=\cL_1(\brak{v}^\ell Z^\beta g)+\cL_2(\brak{v}^\ell Z^\beta g)+[\brak{v}^\ell,\cL_1+\cL_2]Z^\beta g+\brak{v}^\ell[Z^\beta,\cL_1+\cL_2]g.
\end{align}
Thanks to Proposition \ref{prop:equivL1} and \eqref{bd:commZ}-\eqref{bd:commWl}, in the formula above the only term we do not know how to estimate are the last ones containing $\cL_1$ on the right-side. However, multiplying \eqref{bd:leibcomm} by $\brak{v}^\ell$ corresponds in redefining $S$ and $A$ in \eqref{def:S} and \eqref{def:Anti} respectively as 
\begin{align}
	S_\ell&:=(P_{\beta_3}\sqrt{\mu})_*(P_{\beta_1}\sqrt{\mu})_*'\brak{v}^\ell+(P_{\beta_3}\sqrt{\mu})_*'(P_{\beta_1}\sqrt{\mu})_*\brak{v'}^\ell\\
	A_\ell&:=(P_{\beta_3}\sqrt{\mu})_*(P_{\beta_1}\sqrt{\mu})_*'\brak{v}^\ell-(P_{\beta_3}\sqrt{\mu})_*'(P_{\beta_1}\sqrt{\mu})_*\brak{v'}^\ell.
\end{align}
Moreover, also the term $I^2$ in \eqref{def:I2} has an extra $\brak{v}^\ell$ multiplying the Maxwellian $\sqrt{\mu_*}$. Upon introducing further commutators if necessary, for instance in $A_\ell$, one can repeat the arguments done to prove \eqref{bd:commZ} and obtain the same estimates (clearly with worst constants depending on $\ell$).

The fact that the lemma is true if we change $Z$ with $\nabla_v$ is straightforward. Indeed, as observed before, there is no difference between having $Z$ or $\nabla_v$ derivatives (the latter being used for the bounds in \cite{gressman2011global,alexandre2011global}).
\end{proof}
\subsection{Nonlinear operator}
For the nonlinear part, the main trilinear estimate we are going to exploit was derived by Alexandre et al. \cite{alexandre2011global} and Gressman and Strain \cite{gressman2011global,gressman2011sharp} (obtained independently and with different techniques, see the discussion in \cite{gressman2011sharp}).
\begin{theorem}[ \cite{gressman2011sharp}*{Theorem 2.1}, \cite{alexandre2011global}*{Theorem 1.2}]
	\label{th:trilinear}
	Let $0<s<1$ and $\gamma>\max\{-d,-d/2-2s\}$. Then 
	\begin{equation}
		\label{bd:trilinear}
		|\jap{\Gamma(f,g),h}|\lesssim \norm{f}_{L^2_v}\sqrt{\cA[g]\cA[h]}.
	\end{equation}
\end{theorem}
For the commutation properties with the weight we have the following.
\begin{proposition}[Proposition 3.13 \cite{alexandre2011global}]
	\label{prop:trilincomm}
		Let $0<s<1$ and $\gamma>\max\{-d,-d/2-2s\}$. Then, for any $\ell\geq0$ one has 
		\begin{align}
			\label{bd:trilincomm}
			\left|\jap{\jap{v}^\ell\Gamma(f,g)-\Gamma(f,\jap{v}^\ell g),h}\right|\lesssim&\sqrt{\cA[h]} \big(\norm{f}_{L^2_{v,s+\gamma/2}}\normL{\jap{v}^{\ell-s}g}_{L^2_{v,s+\gamma/2}}\\
			\notag &+\min\{\norm{f}_{L^2_v}\normL{\jap{v}^{\ell-s}g}_{L^2_{v,s+\gamma/2}},\, \norm{f}_{L^2_{v,s+\gamma/2}}\normL{\jap{v}^{\ell-s}g}_{L^2_v}\}\big)
		\end{align}
\end{proposition}
Combining Theorem \ref{th:trilinear}, Proposition \ref{prop:trilincomm} and Lemma \ref{lemma:commutation}, we get the following.
\begin{lemma}
	\label{lem:trilBd}
			Let $0<s<1$ and $\gamma>\max\{-d,-d/2-2s\}$. Then, for any $B\geq2d$, $M> B+|\gamma|+2s$,  $|\beta|\leq B$ one has
			\begin{align}
\label{bd:NLZ}\left|\jap{\jap{v}^MZ^\beta\Gamma(f,g),h}\right|&\lesssim\sqrt{\cA[h]}\sum_{|\beta_1|+|\beta_2|\leq|\beta|} \bigg(\normL{\jap{v}^MZ^{\beta_1}f}_{L^2_v}\sqrt{\cA[\jap{v}^MZ^{\beta_2}g]}\\
\notag &\hspace{4cm}+\sqrt{\cA[\jap{v}^MZ^{\beta_1}f]}\normL{\jap{v}^MZ^{\beta_2}g}_{L^2_v}\bigg)	
			\end{align} 
\end{lemma}
\begin{proof}
From \eqref{eq:ZLeib}, we see that we have to control terms like 
	\begin{equation}
		\label{eq:typ}
		\jap{\jap{v}^M\cT\left(Z^{\beta_1}f,Z^{\beta_2}g,(\nabla_v)^{\beta_3}\mu\right),h}.
	\end{equation}	
	We split this term as 
 \begin{align}
 	&\jap{\jap{v}^M\cT\left(Z^{\beta_1}f,Z^{\beta_2}g,(\nabla_v)^{\beta_3}\mu\right),h}
	\\
\label{eq:splittyp}&=\jap{\jap{v}^M\cT\left(Z^{\beta_1}f,Z^{\beta_2}g,(\nabla_v)^{\beta_3}\mu\right)-\cT\left(Z^{\beta_1}f,\jap{v}^MZ^{\beta_2}g,(\nabla_v)^{\beta_3}\mu\right),h}\\
 	&\quad +		\jap{\cT\left(Z^{\beta_1}f,\jap{v}^MZ^{\beta_2}g,(\nabla_v)^{\beta_3}\mu\right),h}.
 \end{align}
  Then, we notice that Theorem \ref{th:trilinear} and Proposition \ref{prop:trilincomm} are stated for $\Gamma(f,g)=\cT(f,g,\sqrt{\mu})$. However, it is not difficult show that the bounds \eqref{th:trilinear}-\eqref{bd:trilincomm} holds true also for $\cT(f,g,\de_{v_i}^{\beta_j}\mu)$. Indeed, the proofs  in \cite{alexandre2011global} rely on a decomposition of the operator $\Gamma$ based on a nice identity to isolate the singularities, see in \cite[Lemma 3.6]{alexandre2011global}. Upon properly symmetrizing to take care of the polynomial (as done for instance in \eqref{def:S}), using \eqref{bd:equivA} and Proposition \ref{lemma:coercive}, one can verify that the estimates \eqref{bd:trilinear}-\eqref{bd:trilincomm} holds also for  $\cT(f,g,\de_{v_i}^{\beta_j}\mu)$. Thus, we can apply the analogoue of Proposition \ref{prop:trilincomm} to the terms in \eqref{eq:splittyp}  and Theorem \ref{th:trilinear} to the remaining ones. 
\end{proof}

\section{Linear Estimates} \label{sec:Lin}
This section is devoted to proving the results announced in Section \ref{sec:outLin}. A first crucial step is to obtain the monotonicity estimates \eqref{bd:energyED} and \eqref{bd:energyTD} which we collect in the next proposition.
\begin{proposition}
\label{prop:mono}
Let $E^{e.d.}_{M,B}$, $\mathcal{D}^{e.d.}_{M,B}$, $E^{T.d.}_{M,B}$ and $\mathcal{D}^{T.d.}_{M,B}$ be the functionals defined in \eqref{def:Eed}, \eqref{def:Ded},  \eqref{def:ETd} and \eqref{def:DTD} respectively. Then, there exists constants $0<\delta_{e}, \delta_{d}<1$
\begin{align}
		&\frac{\dd }{\dd t} E^{e.d.}_{M,B}+\delta_{e} \mathcal{D}^{e.d.}_{M,B}\leq 0 \label{bd:energyED4}\\
		&\label{bd:energyTD4}
		\frac{\dd }{\dd t} E^{T.d.}_{M,B}+\delta_d\mathcal{D}^{T.d}_{M,B}\leq 0.
\end{align}
\end{proposition}
 We prove \eqref{bd:energyED4}  and  \eqref{bd:energyTD4} in Section \ref{subsec:ED} and Section \ref{subsec:TD} respectively.  Having these estimates at hand, in Section \ref{subsec:decay} we present the proof of the decay estimates for the linearized problem given in Theorem \ref{thm:LinDecEst}.

 To simplify the notation, since all the norms are $L^2_v$ based, in the rest of this section we will always omit the subscript $L^2_v$ and we write $H^s_{*}$ instead of $H^s_{v,*}$. Moreover, thanks to \eqref{def:Link} we know that the problem decouples in $k$.
 Hence, the factors $\brak{k}^\alpha$ in the definition of the energies play no role in the estimates  (but are crucial in the nonlinear problem).

\subsection{Monotonicity estimate in the enhanced dissipation regime}
\label{subsec:ED}
The aim of this section is to prove \eqref{bd:energyED4} in Proposition \ref{prop:mono}. As a consequence of the properties given in Section \ref{sec:preliminaries}, we first give some basic energy inequalities which are necessary to control the time-derivative of $E^{e.d.}_{M,B}$ \eqref{def:Eed}.
\begin{lemma} For any $M>2|q_{\gamma,s}|$, $\beta\geq 0$, there hold the energy inequalities
\begin{align}
	\label{bd:en0}&\frac12 \frac{\dd }{\dd t}\norm{\jap{v}^{M}Z^\beta \hat{f}}^2+\frac{\nu}{100}  \cA[\jap{v}^{M}Z^\beta \hat{f}] \leq\nu \mathcal{R}^1_{\beta}, \\	
	\notag&\frac12 \frac{\dd }{\dd t}\norm{\jap{v}^{M+q_{\gamma,s}}\nabla_v Z^\beta \hat{f}}^2+\frac{\nu}{100}\cA[\jap{v}^{M+q_{\gamma,s}}\nabla_v Z^\beta \hat{f}]\\
	\label{bd:env}		&\qquad \qquad \leq \left|\jap{\jap{v}^{M+q_{\gamma,s}}Z^\beta(ik \hat{f}),\jap{v}^{M+q_{\gamma,s}}\nabla_vZ^\beta \hat{f}}\right|+ \nu \mathcal{R}^2_{\beta},\\
			\label{bd:enmixed}&\frac{\dd}{\dd t}\Re \jap{\jap{v}^{M+q_{\gamma,s}}Z^\beta(ik \hat{f}),\jap{v}^{M+q_{\gamma,s}}\nabla_vZ^\beta \hat{f}}+|k|^2\norm{\jap{v}^{M+q_{\gamma,s}} Z^\beta  \hat{f}}^2\leq  \nu \mathcal{R}^3_{\beta}.\end{align}
where the remainder terms are given by
\begin{align}
\label{def:R1}	&\mathcal{R}^1_{\beta}=C_{\beta,1}\bigg(\normL{\mu^\delta Z^\beta \hat{f}}^2+\sum_{|\beta_1|\leq |\beta|-1}\cA[\jap{v}^{M} Z^{\beta_1} \hat{f}]\bigg),\\
\label{def:R2}	&	\mathcal{R}^2_{\beta}= C_{\beta,2}\bigg(\normL{\mu^\delta \nabla_vZ^\beta \hat{f}}^2+\cA[\jap{v}^{M+q_{\gamma,s}}Z^\beta \hat{f}]+\sum_{|\beta_1|\leq |\beta|-1}\cA[\jap{v}^{M+q_{\gamma,s}}\nabla_v Z^{\beta_1} \hat{f}]\bigg),\\
\notag&\mathcal{R}^3_{\beta}=2\left|\jap{\jap{v}^{M+q_{\gamma,s}}\nabla_vZ^\beta \hat{f},\jap{v}^{M+q_{\gamma,s}}Z^\beta \cL (ik \hat{f})}\right|\\
 \label{def:R3}	 &\qquad \quad +2\left|\jap{[\nabla_v,\jap{v}^{M+q_{\gamma,s}}]Z^\beta \hat{f},\jap{v}^{M+q_{\gamma,s}}Z^\beta \cL (ik \hat{f})}\right|,
\end{align}
with $0<C_{\beta,1},C_{\beta,2}$ being fixed constants depending only on $\beta$.
\end{lemma}

\begin{proof}
Bounds \eqref{bd:en0} and \eqref{bd:env} are a consequence of \eqref{bd:lowerWZ}. For \eqref{bd:enmixed}, since 
\begin{align}
	&\de_t\nabla_xZ^\beta f=-v\cdot \nabla_x \nabla_xZ^\beta f-\nu\nabla_x Z^\beta \cL f,\\
	&\de_t\nabla_v Z^\beta f+\nabla_xZ^\beta f=-v\cdot \nabla_x \nabla_vZ^\beta f-\nu\nabla_v Z^\beta \cL f ,
\end{align}
using the antisymmetry of $v\cdot \nabla_x$ we get
\begin{align}
	\frac{\dd}{\dd t}&\Re\jap{\jap{v}^{M+q_{\gamma,s}}Z^\beta(ik \hat{f}),\jap{v}^{M+q_{\gamma,s}}\nabla_vZ^\beta \hat{f}}+|k|^2\normL{\jap{v}^{M+q_{\gamma,s}}Z^\beta \hat{f}}^2\\
	&= 
	-\nu \Re \jap{\jap{v}^{M+q_{\gamma,s}}Z^\beta \cL (ik \hat{f}),\jap{v}^{M+q_{\gamma,s}}\nabla_vZ^\beta \hat{f}}\\
	&\quad-\nu\Re \jap{\jap{v}^{M+q_{\gamma,s}}Z^\beta(ik \hat{f}),\jap{v}^{M+q_{\gamma,s}}\nabla_v(Z^\beta \cL \hat{f})}.
\end{align}
Integrating by parts in $v$ the last term above, and moving the $ik$ on the term containing $\cL$, we prove \eqref{bd:enmixed}.
\end{proof}

We are now ready to prove the first bound in Proposition \ref{prop:mono}.
\begin{proof}[Proof of Proposition \ref{bd:energyED4}]
	From \eqref{def:Eed} and \eqref{bd:en0}-\eqref{bd:enmixed}, we get that 
\begin{align}
\notag \frac{\dd}{\dd t}E^{e.d.}_{M,B}&+\frac{1}{200}\sum_{\alpha+|\beta|\leq B}\frac{2^{-\mathtt{C}\beta}\brak{k}^\alpha}{\jap{\nu t}^{2\beta}}\bigg( \nu\cA[\jap{v}^{M}Z^\beta \hat{f}]+\nu a_{\nu,k} \cA[\jap{v}^{M+q_{\gamma,s}}\nabla_vZ^\beta \hat{f}]\\
		&\quad+ b_{\nu,k}|k|^2\norm{\jap{v}^{M+q_{\gamma,s}} Z^\beta  \hat{f}}^2_{}\bigg)\\
		\label{bd:mixEd}&\leq  \sum_{\alpha+|\beta|\leq B}\frac{2^{-\mathtt{C}\beta}\brak{k}^\alpha}{\jap{\nu t}^{2\beta}} a_{\nu,k}\left|\jap{\jap{v}^{M+q_{\gamma,s}}Z^\beta(ik \hat{f}),\jap{v}^{M+q_{\gamma,s}}\nabla_vZ^\beta \hat{f}}\right|\\
		\label{bd:Ris}&\quad +  \sum_{\alpha+|\beta|\leq B}\frac{2^{-\mathtt{C}\beta}\brak{k}^\alpha}{\jap{\nu t}^{2\beta}}( \nu\mathcal{R}^1_{\beta}+\nu a_{\nu,k}\mathcal{R}^2_{\beta}+\nu b_{\nu,k}\mathcal{R}^3_{\beta}),
	\end{align}
where we have neglected the term with a negative sign on the right-hand side coming from the time derivative of $\jap{\nu t}^{-2\beta}$. In the rest of the proof, we highlight  all the necessary restrictions on the coefficient and in the end we show that it is possible to choose the coefficients in order to satisfy said constraints.
We first bound the error terms in \eqref{bd:Ris}, which can be directly controlled with the available {anisotropic} dissipation appearing on the left-hand side of the inequalty.
Then we control the remaining mixed inner product error term in \eqref{bd:mixEd}.

\medskip \noindent $\diamond$ \textbf{Bound on} $\cR^1_{\beta}$ \eqref{def:R1}. The sum containing lower order derivatives is controlled with the available dissipation coming from the $\beta-1$ terms. Namely, it is enough to impose that 
\begin{equation}
	\label{res0}
	 2^{-\mathtt{C}\beta}C_{\beta,1}\ll \frac{2^{-\mathtt{C}(\beta-1)}}{200} \quad \Longrightarrow \quad  2^{-\mathtt{C}}\ll \frac{1}{C_{\beta,1}}.
\end{equation}
 For the other term in $\cR^1_{\beta}$, since $M+q_{\gamma,s}\geq 0$  we have
\begin{align}
	\label{bd:kernel}
	\normL{\mu^{\delta}Z^\beta \hat{f}}^2
	\lesssim \frac{1}{|k|^2} 	\norm{\jap{v}^{M+q_{\gamma,s}}Z^\beta (ik \hat{f})}^2.
\end{align}
Thus, to control $\cR^1_{\beta}$, we need to impose the following restrictions on the coefficients
\begin{equation}
	\label{res1}
	\frac{\nu C_{\beta,1}}{|k|^2}\ll b_{\nu,k}.
\end{equation}

\medskip \noindent $\diamond$ \textbf{Bound on} $\cR^2_{\beta}$ \eqref{def:R2}. To control the first term in the definition of $\cR^2_{\beta}$, as in \cite[(6.14)]{alexandre2011global}, we exploit the following interpolation inequality: for any $\ell\geq -\gamma$, $\delta>0$ there exists $C_\delta$ such that
\begin{align}
	\normL{\jap{v}^\ell\nabla_v g}_{L^2_{\gamma/2}}^2&\leq \delta \normL{\jap{v}^\ell\nabla_v g}_{H^s_{\gamma/2}}^2+ C_\delta\normL{\jap{v}^\ell \nabla_v g}_{H^{-1}_{\gamma/2}}^2,
	\label{bd:interpolation}
	\lesssim \delta\cA[\jap{v}^\ell \nabla_v g]+C_\delta\normL{\jap{v}^\ell  g}_{L^2_{\gamma/2}}^2.
\end{align}
The bounds above follow by the Gagliardo-Nirenberg inequality and straightforward commutator estimates to handle the weights. 
With the inequality \eqref{bd:interpolation} at hand, we get
\begin{align}
				C_{\beta,2}\normL{\mu^\delta\nabla_v Z^\beta \hat{f}}^2&\leq \frac{1}{400}\cA[\jap{v}^{M+q_{\gamma,s}}\nabla_v Z^\beta \hat{f}]+\frac{\widetilde{C}_{\beta,2}}{|k|^2} \norm{\jap{v}^{M+q_{\gamma,s}} Z^\beta (ik \hat{f})}^2.
\end{align}
Moreover, since $q_{\gamma,s}\leq 0$ we have
\begin{equation}
	\cA[\jap{v}^{M+q_{\gamma,s}}Z^\beta f]\lesssim \cA[\jap{v}^{M}Z^\beta f].
\end{equation}
Hence, we require that
\begin{equation}
	\label{res2}
	\frac{\nu a_{\nu,k}\widetilde{C}_{\beta,2}}{|k|^2}\ll  b_{\nu,k}, \qquad a_{\nu,k}C_{\beta,2}\ll1,  \qquad  2^{-\mathtt{C}}C_{\beta,2}\ll1, 
\end{equation}
in order to be able to absorb the errors terms with the dissipation.

\medskip \noindent $\diamond$ \textbf{Bound on} $\cR^3_{\beta}$ \eqref{def:R3}. From the upper bound in \eqref{bd:coerL}, we get 
$$
|\jap{\cL g,h}|\lesssim \sqrt{\cA[g]}\sqrt{\cA[f]}.
$$ Hence, using $M>|q_{\gamma,s}|$ and $-d<\gamma<0$, from the Cauchy-Schwarz inequality we deduce
\begin{equation}
	\nu b_{\nu,k} \cR^3_{\beta}\leq \frac{\nu}{1600} \cA[\jap{v}^MZ^\beta \hat{f}]+\nu C |k|^2(b_{\nu,k})^2\big( \cA[\jap{v}^{M+q_{\gamma,s}}\nabla_vZ^\beta \hat{f}]+\cA[\jap{v}^{M}Z^\beta \hat{f}]\big)
\end{equation} 
Consequently, the  following restriction on the coefficients is needed
\begin{equation}
	\label{res3}
	C |k|^2(b_{\nu,k})^2\ll \min\{a_{\nu,k},1\}.
\end{equation}
Collecting the estimates on the remainders made above, under the restrictions \eqref{res1},\eqref{res2} and \eqref{res3}, we can absorb all the $\mathcal{R}^i_\beta$ error terms to get
	\begin{align}
	\notag \frac{\dd}{\dd t}E^{e.d.}_{M,B}&+\frac{1}{400}\sum_{\alpha+|\beta|\leq B}\frac{2^{-\mathtt{C}\beta}\brak{k}^\alpha}{\jap{\nu t}^{2\beta}}\bigg( \nu\cA[\jap{v}^{M}Z^\beta \hat{f}]+\nu a_{\nu,k} \cA[\jap{v}^{M+q_{\gamma,s}}\nabla_vZ^\beta \hat{f}]\\ &\hspace{4cm}+b_{\nu,k}|k|^2\norm{\jap{v}^{M+q_{\gamma,s}} Z^\beta \hat{f}}^2\bigg) \\
		\label{bd:func1}&\leq \sum_{\alpha+|\beta|\leq B}\frac{2^{-\mathtt{C}\beta}\brak{k}^\alpha}{\jap{\nu t}^{2\beta}} a_{\nu,k}\left|\jap{\jap{v}^{M+q_{\gamma,s}}Z^\beta(ik \hat{f}),\jap{v}^{M+q_{\gamma,s}}\nabla_vZ^\beta \hat{f}}\right|.	
\end{align}

\medskip \noindent $\diamond$ \textbf{Bound on the mixed inner product.} The last error term appearing in \eqref{bd:func1} is the most delicate to control. Indeed, we do not have $\nu$ as a smallness parameter and it is the first term where we have to explicitly deal with the softness of the potential. To overcome the latter issue, it is crucial to use the fact that $v$-derivatives are controlled with weaker velocity weights. For instance, in the Landau case, one has $s=1$ and $q_{-3,1}=-3/2$. Thus, we can directly control this term combining the Cauchy-Schwarz inequality with $a^2_{\nu,k}/b_{\nu,k}\ll \nu$; the proof for $s=1$ is omitted for brevity, as it is more straightforward (see also \cite{CLN21}, where $q_{-3,1}=-4$ though). 
We begin by noting the following lower bounds, obtained from Proposition \ref{lemma:coercive} 
\begin{align}
	\notag  &\nu \cA[\jap{v}^{M}Z^\beta \hat{f}]+\nu a_{\nu,k} \cA[\jap{v}^{M+q_{\gamma,s}}\nabla_vZ^\beta \hat{f}]\\
	\label{bd:lwdissEd} &\qquad\qquad\gtrsim \nu \norm{\jap{v}^{M}Z^\beta \hat{f}}^2_{H^s_{\gamma/2}}+\nu a_{\nu,k} \norm{\jap{v}^{M+q_{\gamma,s}}\nabla_vZ^\beta \hat{f}}^2_{H^s_{\gamma/2}},
\end{align}
For a general $0<s<1$, we have to handle carefully the anisotropy of the dissipation. We consider a dyadic decomposition of $\mathbb{R}^d$ and apply the Cauchy-Schwarz inequality to get 
\begin{align}
	a_{\nu,k}&\left|\jap{\jap{v}^{M+q_{\gamma,s}}Z^\beta(ik \hat{f}),\jap{v}^{M+q_{\gamma,s}}\nabla_vZ^\beta \hat{f}}\right|\leq \\
&\sum_{j=0}^{\infty}a_{\nu,k}|k| \norm{\mathbbm{1}_{\{2^{j}\leq\brak{v}\leq 2^{j+1}\}}\brak{v}^{M+q_{\gamma,s}}Z^\beta \hat{f}}\norm{\mathbbm{1}_{\{2^{j}\leq\brak{v}\leq 2^{j+1}\}}\brak{v}^{M+q_{\gamma,s}}\nabla_v Z^\beta \hat{f}}\\
:=&\, \sum_{j=0}^{\infty}\mathcal{R}_{j}.
\end{align}
Using the Young's inequality, we get
\begin{align}
\label{bd:EDRj}
\mathcal{R}_j\leq\, &  \frac{b_{\nu,k}}{1600}|k|^2\norm{\mathbbm{1}_{\{2^{j}\leq\brak{v}\leq 2^{j+1}\}}\brak{v}^{M+q_{\gamma,s}}Z^\beta \hat{f}}^2\\
&+1600\frac{a^2_{\nu,k}}{b_{\nu,k}}\norm{\mathbbm{1}_{\{2^{j}\leq\brak{v}\leq 2^{j+1}\}}\brak{v}^{M+q_{\gamma,s}}\nabla_v Z^\beta \hat{f}}^2.
\end{align}
By the choice of the coefficients in \eqref{def:abnuk}, notice that
\begin{equation}
\label{eq:coa2b}
\frac{a^2_{\nu,k}}{b_{\nu,k}}=\frac{a_0^{1+s}}{b_0}\nu a^{1-s}_{\nu,k}.
\end{equation}
Combining \eqref{eq:coa2b} with the Gagliardo-Nirenberg inequality, since $q_{\gamma,s}\leq0$, we obtain 
\begin{align}
1600&\frac{a^2_{\nu,k}}{b_{\nu,k}}\norm{\mathbbm{1}_{\{2^{j}\leq\brak{v}\leq 2^{j+1}\}}\brak{v}^{M+q_{\gamma,s}}\nabla_v Z^\beta \hat{f}}^2\leq \\
&C\frac{a_0^{1+s}}{b_0}\bigg(2^{-(j+1)\gamma}\nu a_{\nu,k}\norm{\mathbbm{1}_{\{2^{j}\leq\brak{v}\leq 2^{j+1}\}}\brak{v}^{M+q_{\gamma,s}}\nabla_v Z^\beta \hat{f}}^2_{H^s_{\gamma/2}}\bigg)^{1-s}\\
&\times \bigg(2^{-(j+1)\gamma+2jq_{\gamma,s}}\nu\norm{\mathbbm{1}_{\{2^{j}\leq\brak{v}\leq 2^{j+1}\}}\brak{v}^{M} Z^\beta \hat{f}}^2_{H^s_{\gamma/2}}\bigg)^{s}.
\end{align}
From the inequality above,  we deduce that 
\begin{align}
&\sum_{j=0}^{+\infty}\mathcal{R}_j\leq \,  \frac{b_{\nu,k}}{1600}|k|^2\norm{\brak{v}^{M+q_{\gamma,s}}Z^\beta \hat{f}}^2\\
\notag&+C\frac{a_0^{1+s}}{b_0}\bigg(\nu a_{\nu,k}\norm{\brak{v}^{M+q_{\gamma,s}}\nabla_v Z^\beta \hat{f}}^2_{H^s_{\gamma/2}}\bigg)^{1-s}\bigg(\nu\norm{\brak{v}^{M} Z^\beta \hat{f}}^2_{H^s_{\gamma/2}}\bigg)^{s}  \sum_{j=0}^{+\infty}2^{-(j+1)\gamma}2^{2sjq_{\gamma,s}}.
\end{align}
The series above is convergent for any $q_{\gamma,s}<\gamma/(2s)$, which is guaranteed by the choice we made in \eqref{def:qgammas}. Thus, using \eqref{bd:lwdissEd},  
we conclude  
\begin{align}
\sum_{j=0}^{+\infty}\mathcal{R}_j\leq\, &\frac{b_{\nu,k}}{1600}|k|^2\norm{\brak{v}^{M+q_{\gamma,s}}Z^\beta \hat{f}}^2 \\
\label{bd:mixedRj}&+\widetilde{C}\frac{a_0^{1+s}}{b_0}\left( \nu a_{\nu,k}\cA[\brak{v}^{M+q_{\gamma,s}}\nabla_v Z^\beta \hat{f}]+\nu \cA[\brak{v}^{M} Z^\beta \hat{f}]\right).
\end{align}
The term above can be absorbed on the left-hand side of \eqref{bd:func1} upon choosing
\begin{equation}
\label{bd:mixedconstr}
a_0^{1+s}\ll b_0.
\end{equation}

By the definition of the coefficients in \eqref{def:abnuk}, satisfying the constraints \eqref{res0}, \eqref{res1}, \eqref{res2}, \eqref{res3} and \eqref{bd:mixedconstr} is equivalent to imposing
\begin{equation}
\label{bd:constEd}
\begin{split}&\mathtt{C}\gg 1, \qquad \left(\frac{\nu}{|k|}\right)^{\frac{2s}{1+2s}}\leq \delta_0^{\frac{2s}{1+2s}}\ll b_0, \qquad a_0 \delta_0^{\frac{2(1+s)}{(1+2s)}}\ll b_0,\\
&b_0^2\ll a_0, \qquad a_0^{1+s}\ll b_0.
\end{split}
\end{equation}
Hence, $\mathtt{C}$ is simply some sufficiently large (universal) constant to absorb the error terms coming from the commutators of $\cL$. To satisfy the restrictions on $a_0,b_0,\delta_0$, let $0<\kappa_0\ll1$. One can then choose
\begin{equation}
a_0=\frac{1}{1000}\kappa_0^{\frac{1}{1+2s}}, \qquad b_0=\sqrt{\kappa_0 a_0}, \qquad \delta_0=\frac{1}{1000}\min\big\{\kappa_0^{\frac{1+s}{2s}},\kappa_0^{\frac{s}{2(1+s)}}\big\}.
\end{equation}
It is not hard to verify that this choice satisfies all the constraints in \eqref{bd:constEd}. Finally, using  \eqref{bd:mixedRj} in \eqref{bd:func1}, we arrive at
	\begin{align}
	\label{bd:enhanced} \frac{\dd}{\dd t}E^{e.d.}_{M,B}+\frac{1}{1600}\sum_{\alpha+|\beta|\leq B}\frac{2^{-\mathtt{C}\beta}\brak{k}^\alpha}{\jap{\nu t}^{2\beta}}\bigg(& \nu \cA[\jap{v}^{M}Z^\beta \hat{f}]+\nu a_{\nu,k} \cA[\jap{v}^{M+q_{\gamma,s}}\nabla_vZ^\beta \hat{f}]\\
	&+  \frac{b_{\nu,k}}{1600}|k|^2\norm{\jap{v}^{M+q_{\gamma,s}} Z^\beta \hat{f}}^2\bigg)\leq 0.
\end{align}
In view of the definition of $\mathcal{D}^{e.d.}_{M,B}$ \eqref{def:Ded}, the monotonicity estimate \eqref{bd:energyED4} is proved.
\end{proof}

\subsection{Monotonicity estimate in the Taylor dispersion regime} 
\label{subsec:TD}
We now turn our attention to the the proof of \eqref{bd:energyTD4}. As explained in Section \ref{sec:outLin}, we have to combine the micro-macro energy approach and the hypocoercivity scheme. We first deal with the macroscopic quantities in Section \ref{subsec:macro} and then we present the bounds for microscopic ones in Section \ref{subsec:micro}. 

\subsubsection{Bounds on the macroscopic quantities}
\label{subsec:macro}
Recall that 
\begin{align}
	&\bP f= \left(\rho(t,x)+\sfm(t,x)\cdot v+ \sfe(t,x)(|v|^2-d)\right) \sqrt{\mu},
\end{align}
where $(\rho,\sfm,\sfe)$ satisfy the hydrodynamic system \eqref{eq:detrho0}-\eqref{eq:dte0}. To gain dissipation for the macroscopic variables, it is necessary to introduce the equations satisfied by $\Theta[(I-\bP)f], \Lambda[(I-\bP)f]$ (that are terms appearing on the right-hand side of \eqref{eq:detm0} and \eqref{eq:dte0} respectively) which are given 
\begin{align}
		\label{eq:Theta0}	&\de_t(\Theta[(I-\bP)f]+2\sfe I)+\nabla_x\sfm+(\nabla_x\sfm)^T=-\Theta[ik \cdot v(I-\bP)\hat{f}+\nu\cL(I-\bP)f],\\
	\label{eq:Lambda0}&\de_t\Lambda[(I-\bP)f]+\nabla_x \sfe=-\Lambda[ik \cdot v(I-\bP)\hat{f}+\nu\cL(I-\bP)f].
\end{align}
We recall that $\Theta, \Lambda$ are the higher-order moments projections defined in \eqref{def:Theta} and \eqref{def:Lambda} respectively.
Looking at the structure of the hydrodynamic system \eqref{eq:detrho0}-\eqref{eq:dte0} and \eqref{eq:Theta0}-\eqref{eq:Lambda0}, it is natural to try to exploit mixed inner products to recover dissipation for the macroscopic variables. Indeed, loosely speaking, one has the following 
\begin{align}
	&\frac{\dd}{\dd t}(\sfm \cdot \nabla_x\rho)+|\nabla_x \rho|^2=\text{error terms}\\
	&\frac{\dd}{\dd t}((\Theta[(I-\bP)f]+2\sfe I): (\nabla_x\sfm+(\nabla_x\sfm)^T))+|\nabla_x\sfm+(\nabla_x\sfm)^T|^2=\text{error terms}\\
	\label{eq:heurT3}&\frac{\dd}{\dd t}(\Lambda[(I-\bP)f]\cdot \nabla_x \sfe)+|\nabla_x\sfe|^2=\text{error terms}
\end{align}
The idea of using mixed inner products to recover dissipation dates back at least to the PhD thesis of Kawashima \cite{kawashima1984systems} and has been successfully exploited in many different problems \cite{ADM21,guo2006boltzmann,guo2012decay,strain2012optimal,duan2009stability}. 
However, we also need to recover dissipation for terms involving $Z$-derivatives. Thus, we first need the following equivalence.
\begin{lemma}
\label{lem:equivZP}
For any $\beta\geq 0$, the following inequalities holds true
\begin{align}
	\label{bd:equivZP}
	&\normL{Z^\beta \widehat{\bP f}}_{L^2_v}^2\leq 100|(t k)^{2\beta}||(\hat{\rho},\hat{\sfm},\hat{\sfe})|^2+C_1\sum_{0\leq |\tilde{\beta}|\leq |\beta|-1}|(t k)^{2\tilde{\beta}}||(\hat{\rho},\hat{\sfm},\hat{\sfe})|^2,\\
	&\normL{Z^\beta \widehat{\bP f}}_{L^2_v}^2\geq \frac12|(t k)^{2\beta}||(\hat{\rho},\hat{\sfm},\hat{\sfe})|^2-C_2\sum_{0\leq |\tilde{\beta}|\leq |\beta|-1}|(t k)^{2\tilde{\beta}}||(\hat{\rho},\hat{\sfm},\hat{\sfe})|^2.
\end{align}
\end{lemma}
From this lemma we deduce that we can consider $\normL{Z^\beta\widehat{\bP f}}_{L^2_v}$ as being equivalent to $|t k|^\beta|(\hat{\rho},\hat{\sfm},\hat{\sfe})|$ up to lower order terms.

\begin{proof}
When $\beta=0$ the equivalence is a direct consequence of the orthogonality in $L^2_v$ of   $$(\sqrt{\mu},v_i\sqrt{\mu},(|v|^2-d)\sqrt{\mu}),$$
which is a basis for the kernel of $\cL$. 
When $|\beta|>0$, notice that
 \begin{align}
	\notag Z^{\beta}(\bP f)&=\sum_{|\beta_1|+|\beta_2|=|\beta|}C_{\beta_1,\beta_2} \bigg(((t\nabla_x)^{\beta_1}\rho)(\nabla_v^{ \beta_2}\sqrt{\mu})+((t\nabla_x)^{\beta_1}\sfe)(\nabla_v^{ \beta_2}((|v|^2-d)\sqrt{\mu}))\\
	&\qquad+\sum_{j=1}^d ((t\nabla_x)^{\beta_1}\sfm_{j})(\nabla_v^{ \beta_2}(v_j\sqrt{\mu}))\bigg)\\
	&=((t\nabla_x)^{\beta}\rho )\sqrt{\mu}+((t\nabla_x)^{\beta}\sfe)((|v|^2-s)\sqrt{\mu})+\sum_{j=1}^d ((t\nabla_x)^{\beta}\sfm_{j})(v_j\sqrt{\mu})+\mathcal{I}_{\widetilde{\beta}},
\end{align}
where the term $\mathcal{I}_{\widetilde{\beta}}$ is what it remains from the sum when $|\beta_2|\geq 1$.
Using again the orthogonality condition, the proof of \eqref{bd:equivZP} follows by the identity above and Cauchy-Schwarz inequality.
\end{proof}

The dissipation for the macroscopic variables is recovered from the mixed inner product defined in \eqref{def:Mkbeta}. In particular, we have the following adaptation of the estimate originally obtained in \cite{guo2006boltzmann,duan2009stability}. 
\begin{lemma}
	\label{lemma:mixed}
Let $\cM$ be defined as in \eqref{def:Mkbeta} and assume that $b_2\ll b_1\ll 1$. Then, there exists a universal constant $0<\tilde{c}<1$ such that
\begin{align}
	\label{bd:mixedTD}\frac{\dd }{\dd t}\cM&+\tilde{c}\sum_{0\leq|\beta|\leq B}  \frac{|k|^2}{\brak{\nu t}^{2\beta}}\normL{Z^\beta \bP \hat{f}}^2_{L^2_v}	\lesssim \nu^2 \normL{\mu^\delta(I-\bP)\hat{f}}^2_{L^2_v}.
\end{align}
\end{lemma}
This estimate is the only one where we need to crucially use the factor $\brak{\nu t}^{2\beta}$ in the definition of the energy functional, as we explain in more details after the proof (see Remark \ref{rem:nutbeta}). Before proving this lemma, we recall some basic properties of the higher order moment projections. 
\begin{lemma}
\label{lem:ThetaLambda}
Let $\Theta[\cdot],\Lambda[\cdot]$ be defined as in \eqref{def:Theta}-\eqref{def:Lambda}. Then, there exists a universal constant $0<\delta<1$ such that  
\begin{align}
	\label{bd:projhom}
	|\Theta[g]|+|\Lambda[g]|&\lesssim \normL{\mu^{\delta}g}_{L^2_v}\\
	\label{bd:projhomL}
	|\Theta[\cL(g)]|+|\Lambda[\cL(g)]|&
	\lesssim \normL{\mu^\delta g}_{L^2_v}.
\end{align}
\end{lemma}
\begin{proof}
The proof of \eqref{bd:projhom} is a direct consequence of Cauchy-Schwarz inequality and the fact that $\brak{v}^p\mu^q\in L^2_v$ for any $p\geq 0$ and $q>0$. Similarly, to prove \eqref{bd:projhomL} notice that 
\begin{equation}
|\Lambda_i[\cL(g)]|	=|\jap{(|v|^2-(d+2))v_i\sqrt{\mu},\cL g}_{L^2_v}|=|\jap{\cL(|v|^2v_i\sqrt{\mu}),g}_{L^2_v}|\lesssim \normL{\mu^\delta g}_{L^2_v}
\end{equation}
and analogous bounds holds for $\Theta$, whence proving \eqref{bd:projhomL}.
\end{proof}
We are now ready to prove Lemma \ref{lemma:mixed}. In the proof we again omit the $L^2_v$ subscript in the norms.
 \begin{proof}[Proof of Lemma \ref{lemma:mixed}]
 In \cite[Lemma 4.1]{duan2009stability} the following inequality is obtained (for $\nu=1$)
 \begin{align}
 \label{bd:macroDuan}
 \frac{1}{\nu}\frac{\dd }{\dd t} \mathcal{M}+\tilde{c}_\star\frac{|k|^2}{\nu}|(\hat{\rho},\hat{\sfe},\hat{\sfm})|^2\lesssim \nu\norm{\mu^\delta(I-\bP)\hat{f}}^2,
 \end{align}
 for some $\tilde{c}_\star$ independent of $\nu, k$. For convenience of the reader, we present below the proof of this inequality. However first, having at hand \eqref{bd:macroDuan}, we prove \eqref{bd:mixedTD}. Indeed,  since $|k|\leq \delta_0^{-1}\nu$, observe that
 \begin{equation}
 |k|^2\geq \delta_0^{2\beta}\frac{|tk|^{2\beta}}{\brak{\nu t}^{2\beta}}|k|^2,
 \end{equation}
 for any $|\beta|\geq0$. Appealing to Lemma \ref{lem:equivZP}, we get
 \begin{align}
 |k|^2|(\hat{\rho},\hat{\sfe},\hat{\sfm})|^2&\geq \frac{1}{B+1}\sum_{0\leq |\beta|\leq B}\delta_0^{2\beta}|k|^2\frac{|tk|^{2\beta}}{\brak{\nu t}^{2\beta}}|(\hat{\rho},\hat{\sfe},\hat{\sfm})|^2\\
 \notag&\geq\frac{1}{200(B+1)}\sum_{0\leq |\beta|\leq B}\delta_0^{2\beta}\bigg(\frac{|k|^2}{\brak{\nu t}^{2\beta}}\normL{Z^\beta\widehat{\bP f}}^2-C_1\sum_{0\leq |\widetilde{\beta}|\leq |\beta|-1}\frac{|tk|^{2\widetilde{\beta}}}{\brak{\nu t}^{2\widetilde{\beta}}}|(\hat{\rho},\hat{\sfe},\hat{\sfm})|^2\bigg) \\
 &\quad +\frac{1}{2(B+1)}\sum_{0\leq |\beta|\leq B}\delta_0^{2\beta}|k|^2\frac{|tk|^{2\beta}}{\brak{\nu t}^{2\beta}}|(\hat{\rho},\hat{\sfe},\hat{\sfm})|^2.
 \end{align}
 Since 
 \begin{equation}
 \sum_{0\leq |\beta|\leq B}\delta_0^{2\beta}C_1\sum_{0\leq |\widetilde{\beta}|\leq |\beta|-1}\frac{|tk|^{2\widetilde{\beta}}}{\brak{\nu t}^{2\widetilde{\beta}}}|(\hat{\rho},\hat{\sfe},\hat{\sfm})|^2\leq \delta_0^{2}BC_1\sum_{0\leq |\widetilde{\beta}|\leq B-1}\delta_0^{2\widetilde{\beta}}\frac{|tk|^{2\widetilde{\beta}}}{\brak{\nu t}^{2\widetilde{\beta}}}|(\hat{\rho},\hat{\sfe},\hat{\sfm})|^2,
 \end{equation} 
 for $\delta_0$ sufficiently small we deduce that 
 \begin{equation}
 |k|^2|(\hat{\rho},\hat{\sfe},\hat{\sfm})|^2\geq \frac{1}{200(B+1)}\sum_{0\leq |\beta|\leq B}\delta_0^{2\beta}\frac{|k|^2}{\brak{\nu t}^{2\beta}}\normL{Z^\beta\widehat{\bP f}}^2.
 \end{equation}
 Combining the inequality above with \eqref{bd:macroDuan}, we prove \eqref{bd:mixedTD} with $\tilde{c}:=\tilde{c}_\star\delta_0^{2\beta}/(200(B+1))$.

 To prove \eqref{bd:macroDuan}, applying $\nabla_x$ to \eqref{eq:dte0} and taking the Fourier transform of the resulting equation and of \eqref{eq:Lambda0}, we get 
	\begin{align}
		\label{eq:mixe1}\frac{\dd }{\dd t}\Re(\Lambda&[(I-\bP)\hat{f}]\cdot  \widehat{ik\sfe})+|k|^2|\hat{\sfe}|^2\\
	\label{eq:mixe2}		\leq\,&C_{\sfe}\big(|k|^2|\hat{\sfm}||\Lambda[(I-\bP)\hat{f}]| +|k|^2|\Lambda[(I-\bP)\hat{f}]|^2\\
	\label{eq:mixe3}	&+|k||\Lambda[ik \cdot v(I-\bP)\hat{f}]||\hat{\sfe}|+\nu|k||\Lambda[\cL((I-\bP)\hat{f}]|  |\hat{\sfe}|\big)\\
	:=\,&C_{\sfe}\big(\cI^\sfe_1+\dots +\cI^\sfe_4\big)
	\end{align}
From Young's inequality and Lemma \ref{lem:ThetaLambda} we get for some $\widetilde{C}_j$'s, 
\begin{align}
    \cI^\sfe_1&\leq\frac{b_1}{16C_{\sfe}}|k|^2|\hat{\sfm}|^2+\widetilde{C}_1|k|^2\normL{\mu^\delta(I-\bP)\hat{f}}^2,\\
    \cI^\sfe_2&\leq \widetilde{C}_2|k|^2\normL{\mu^\delta(I-\bP)\hat{f}}^2,\\
    \cI^\sfe_3&\leq\frac{|k|^2}{16C_{\sfe}}|\hat{\sfe}|^2+\widetilde{C}_3|k|^2\normL{\mu^\delta (I-\bP)\hat{f}}^2,\\
\cI^\sfe_4&\leq  \frac{|k|^2}{16C_{\sfe}} |\hat{\sfe}|^2+\widetilde{C}_4\nu^2\normL{\mu^\delta(I-\bP)\hat{f}}^2.
\end{align}
Since $|k|\lesssim \nu$, using the bounds above in \eqref{eq:mixe1} we obtain 
\begin{align}
			\label{eq:mixe}\frac{\dd }{\dd t}\Re(\Lambda[(I-\bP)\hat{f}]\cdot  ik \hat{\sfe})+\frac12|k|^2|\hat{\sfe}|^2\leq\frac{b_1}{16}|k|^2 |\hat{\sfm}|^2+\widetilde{C}\nu^2\normL{\mu^\delta(I-\bP)\hat{f}}^2.
\end{align}
Consider now the second term in $\cM$.
From \eqref{eq:Theta0} and \eqref{eq:detm0}, one obtains
\begin{align}
\label{eq:dtmixm0}	\frac{\dd }{\dd t}&\Re\bigg(\big(\Theta[(I-\bP)\hat{f}]+2(\hat{\sfe} I)\big):\big((ik\hat{\sfm}+(ik\hat{\sfm})^T)\big)\bigg)+|k\otimes\hat{\sfm}+(k\otimes\hat{\sfm})^T|^2\\
\lesssim\,&|k|^2(|\hat{\sfe}|+| \hat{\rho}|)(|\Theta[(I-\bP)\hat{f}]|+|\hat{\sfe}|)+|k|^2|\Theta[(I-\bP)\hat{f}]|(|\Theta[(I-\bP)\hat{f}]|+|\hat{\sfe}|)\\
&+|k||\hat{\sfm}|\big(|\Theta[\cF(v\cdot \nabla_x((I-\bP)f))]|+\nu\Theta[\cL((I-\bP)\hat{f})]|\big).
\end{align}
 Notice that 
\begin{align}
| k\otimes\hat{\sfm}+(k\otimes\hat{\sfm})^T|^2=\sum_{i,j=1}^d(k_j \hat{\sfm}_{i}+k_i\hat{\sfm}_{j})^2=(|k|^2|\hat{\sfm}|^2+|k\cdot \hat{\sfm}|^2).
\end{align}
Hence, using Young's inequality and Lemma \ref{lem:ThetaLambda}, similarly to \eqref{eq:mixe}, there exists $C_{\sfm}$ (independent of $\sfm$) such that
\begin{align}
	\label{eq:dtmixm}	\frac{\dd }{\dd t}&\Re\bigg(\big(\Theta[(I-\bP)\hat{f}]+2(\hat{\sfe} I)\big):\big((ik\hat{\sfm}+(ik\hat{\sfm})^T)\big)\bigg)+\frac12|k|^2|\hat{\sfm}|^2\\
	\label{eq:dtmixmC}&\qquad\leq\frac{b_2}{16}|\hat{\rho}|^2+C_{\sfm}|k|^2|\hat{\sfe}|^2+\widetilde{C}_{\sfm}\nu^2\normL{\mu^\delta(I-\bP)\hat{f}}^2.
\end{align}
Arguing analogously for the remaing mixed inner product, we infer there exists $C_{\rho}$ (independent of $\rho$) such that
\begin{align}
		\label{eq:dtmixrhoF}\frac{\dd }{\dd t}&\operatorname{Re}( \hat{\sfm}\cdot ik \hat{\rho})+\frac{|k|^2}{2}|\hat{\rho}|^2\leq C_{\rho}|k|^2 (|\hat{\sfm}|^2+|\hat{\sfe}|^2).
\end{align}
In light of \eqref{bd:equivZP}, choosing $b_2\ll b_1\ll 1$ and combining \eqref{eq:mixe}, \eqref{eq:dtmixm}, \eqref{eq:dtmixrhoF} we prove the bound \eqref{bd:macroDuan} for a suitable constant $\tilde{c}_\star$.
 \end{proof}
 
 \begin{remark}
 \label{rem:nutbeta}
In the proof of Lemma \ref{lemma:mixed}, it is crucial to exploit the factor $\brak{\nu t}^{2\beta}$ to recover dissipation for all $Z^\beta \bP f$ by using only the one available for $(\rho,\sfm,\sfe)$ (corresponding to $\beta =0$). In fact, one can also define the mixed inner product for $(t\nabla_x)^{\otimes \beta}(\rho,\sfm,\sfe)$ and try to obtain an estimate as \eqref{bd:mixedTD}. However, for $|\beta|\geq 1$ we have some dangerous error terms coming from $\nu\cL(I-\bP)f$ in \eqref{eq:Theta0}-\eqref{eq:Lambda0}, where the main errors are proportional to
\begin{equation}
\nu\normL{\mu^\delta Z^\beta (I-\bP)f}_{L^2_v}^2.
\end{equation}
For $\beta=0$ this is fine since, by the standard $L^2_v$ energy estimate, in the dissipation functional we  have a term that scales as $\nu\normL{\brak{v}^{\gamma/2+s}(I-\bP)f}_{L^2_v}^2$. On the other hand, the available dissipation on $Z^\beta(I-\bP)f$ has at most a factor $\nu^{-1}|k|^2$ in front, which is much smaller than $\nu$ in the regime $|k|\ll \nu$. To overcome this difficulty, one possibility is to exploit the estimate 
\begin{equation}
\normL{\mu^\delta Z^\beta (I-\bP)f}_{L^2_v}^2\lesssim \brak{\nu t}^2\sum_{0\leq |\widetilde{\beta}|\leq |\beta|-1} \normL{\mu^\delta Z^{\widetilde{\beta}} (I-\bP)f}_{L^2_v}^2.
\end{equation}
The loss $\brak{\nu t}^2$ can indeed be easily controlled if we divide by $\brak{\nu t}^{2\beta}$ each mixed inner product related to  $(t\nabla_x)^{\otimes \beta}(\rho,\sfm,\sfe)$.  Hence, dividing by $\brak{\nu t}^{2\beta}$ seems to be necessary to handle frequencies $|k|\ll \nu$, where the effect of the phase mixing generated by the transport can be too weak with respect to the collisional effects. 
 \end{remark}
 
 \subsubsection{Bounds on the microscopic part}
\label{subsec:micro}
In the remains to control the microscopic quantities in $E^{T.d.}_{M,B}$. Recall that $(I-\bP)f$ satisfies 
\begin{equation}
\label{eq:BoltzTayI-P}
\de_t(I-\bP) f+v\cdot \nabla_x (I-\bP)f+\nu\cL(I-\bP)f=\bP(v\cdot \nabla_x f)-v\cdot \nabla_x \bP f.
\end{equation}
In the following lemma, we collect the bounds that are needed to recover the dissipation on the microscopic part of $f$, which is a key ingredient to prove \eqref{bd:energyTD}. \begin{lemma}
	\label{lemma:microTD}
	Let $M>2\max\{|q_{\gamma,s}|,|\gamma|/2+s\}$, $\beta\geq 0$ and $\tilde{c}$ be the constant in \eqref{bd:mixedTD}. There exists $0<\tilde{\delta}<1$ such that the following inequalities holds true: for the unweighted vector fields of the full solution we have
	\begin{align}
		\label{bd:enknuZf}\frac12\frac{|k|}{\nu}\frac{\dd }{\dd t}&\normL{Z^{\beta}\hat{f}}^2_{L^2_v}+\tilde{c}| k| \cA[Z^{\beta}(I-\bP)\hat{f}]\leq\, \tilde{\delta}\nu\sum_{|\tilde{\beta}|\leq |\beta|}\cA[Z^{\tilde{\beta}}(I-\bP)\hat{f}]\\
		&+C_{\tilde{\delta}}\frac{|k|^2}{\nu}\sum_{|\beta_1|\leq|\beta|-1}\normL{Z^{\beta_1}\bP \hat{f}}^2_{L^2_v}+C|k|\sum_{|\beta_1|\leq |\beta|-1}\cA[Z^{\beta_1}(I-\bP)\hat{f}].
	\end{align}	
For the unweighted vector fields of the microscopic part we get
	\begin{align}
	\label{bd:enZI-Pf}\frac{1}{2}\frac{\dd }{\dd t}\normL{Z^{\beta}(I-\bP)\hat{f}}^2_{L^2_v}+\tilde{c}\nu \cA[Z^{\beta}(I-\bP)\hat{f}]
	\leq\,	&C_{\tilde{\delta}}\frac{|k|^2}{\nu}\sum_{|\tilde{\beta}|\leq|\beta|}\normL{Z^{\tilde{\beta}}\bP \hat{f}}^2_{L^2_v}\\
	&+C(\nu+|k|)\sum_{|\beta_1|\leq |\beta|-1}\cA[Z^{\beta_1}(I-\bP)\hat{f}],
\end{align}	
and
	\begin{align}
	\notag\frac{1}{2}\frac{\dd }{\dd t}&\normL{Z^{\beta}\nabla_v(I-\bP)\hat{f}}^2_{L^2_v}+\tilde{c}\nu \cA[Z^{\beta}\nabla_v(I-\bP)\hat{f}]
	\leq\,	{C}\nu (\cA[Z^{\beta}(I-\bP)\hat{f}]+\cA[\jap{v}^MZ^{\beta}(I-\bP)\hat{f}])\\
	\label{bd:enZI-Pfnabla}	&+{C}\frac{|k|^2}{\nu}\sum_{|\tilde{\beta}|\leq|\beta|}\normL{Z^{\tilde{\beta}}\bP \hat{f}}^2_{L^2_v}+C(\nu+|k|)\sum_{|\beta_1|\leq |\beta|-1}\cA[Z^{\beta_1}\nabla_v(I-\bP)\hat{f}]. 
\end{align}	
For the weighted vector fields of the microscopic part we obtain
	\begin{align}
	\label{bd:enZvMI-Pf}\frac{1}{2}\frac{\dd }{\dd t}&\normL{\jap{v}^MZ^{\beta}(I-\bP)\hat{f}}^2_{L^2_v}+\tilde{c}\nu \cA[\jap{v}^MZ^{\beta}(I-\bP)\hat{f}]
	\leq C_{\tilde{\delta}}\nu \cA[Z^\beta(I-\bP)\hat{f}] \\
	&+C_{\tilde{\delta}}\frac{|k|^2}{\nu}\sum_{|\tilde{\beta}|\leq|\beta|}\normL{Z^{\tilde{\beta}}\bP \hat{f}}^2_{L^2_v}+C(\nu+|k|)\sum_{|\beta_1|\leq |\beta|-1}\cA[Z^{\beta_1}(I-\bP)\hat{f}],
\end{align}	
and
\begin{align}
	\label{bd:enZvMI-Pfnabla}\frac{1}{2}\frac{\dd }{\dd t}&\normL{\jap{v}^{M+q_{\gamma,s}}Z^{\beta}\nabla_v(I-\bP)\hat{f}}^2_{L^2_v}+\tilde{c}\nu \cA[\jap{v}^{M+q_{\gamma,s}}\nabla_vZ^{\beta}(I-\bP)\hat{f}]
	\\
	&\leq\tilde{\delta}\frac{|k|^2}{\nu}\normL{\jap{v}^{M+q_{\gamma,s}}(I-\bP)\hat{f}}^2_{L^2_v}\\
	&\quad +C_{\tilde{\delta}}\nu (\cA[Z^\beta\nabla_v(I-\bP)\hat{f}] +\cA[Z^{\beta}(I-\bP)\hat{f}]+\cA[\jap{v}^MZ^{\beta}(I-\bP)\hat{f}])\\
	&\quad +C_{\tilde{\delta}}\frac{|k|^2}{\nu}\sum_{|\tilde{\beta}|\leq|\beta|}\normL{Z^{\tilde{\beta}}\bP \hat{f}}^2_{L^2_v}+C(\nu+|k|)\sum_{|\beta_1|\leq |\beta|-1}\cA[Z^{\beta_1}(I-\bP)\hat{f}].
\end{align}	
Finally, for the microscopic mixed inner product we have
\begin{align}
	\label{bd:enTDmixedI-P}
	\frac{1}{\nu}\frac{\dd}{\dd t}&\Re \jap{\jap{v}^{M+q_{\gamma,s}}Z^\beta (\nabla_x(I-\bP)f)_k,\jap{v}^{M+q_{\gamma,s}}Z^\beta\nabla_v (I-\bP)_k}_{L^2_v}\\
	&+\frac{|k|^2}{\nu}\normL{\jap{v}^{M+q_{\gamma,s}}(I-\bP)\hat{f}}^2_{L^2_v}\\
	&\leq C\nu (\cA[\brak{v}^{M+q_{\gamma,s}}Z^\beta\nabla_v(I-\bP)\hat{f}] +\cA[\brak{v}^MZ^{\beta}(I-\bP)\hat{f}])\\
	&\quad +C\frac{|k|^2}{\nu}\sum_{|\tilde{\beta}|\leq|\beta|}\normL{Z^{\tilde{\beta}}\bP \hat{f}}^2_{L^2_v}+C(\nu+|k|)\sum_{|\beta_1|\leq |\beta|-1}\cA[Z^{\beta_1}(I-\bP)\hat{f}].
\end{align}
\end{lemma}
\begin{proof} 
We omit the subscript $L^2_v$ for convenience of notation. From \eqref{def:Link} 
we compute that 
	\begin{align}
		\label{eq:enIdZf}
		\frac12\frac{\dd }{\dd t}\normL{Z^\beta \hat{f}}^2&+\nu \jap{\cL(Z^{\beta}(I-\bP)\hat{f}),Z^{\beta}(I-\bP)\hat{f}}=-\nu\Re\jap{Z^\beta\cL(I-\bP)\hat{f},Z^{\beta}\bP \hat{f}}\\
	\notag	&\quad\quad+\nu\Re\jap{[\cL,Z^\beta](I-\bP)\hat{f}, Z^{\beta}(I-\bP)\hat{f}}.
	\end{align}
Appealing to \eqref{bd:coerL}, since $(a-b)^2\geq a^2/2-2b^2$, we get 
\begin{align}
	\jap{\cL(Z^{\beta}(I-\bP)\hat{f}),Z^{\beta}(I-\bP)\hat{f}}&\geq C_1\cA[(I-\bP)Z^{\beta}(I-\bP)\hat{f}]\\
	\label{bd:lwZI-Pf}	&\geq \frac{C_1}{2}\cA[Z^\beta(I-\bP)\hat{f}]-2C_1\cA[\bP(Z^{\beta}(I-\bP)\hat{f})].
\end{align}
By the definition of $Z^\beta$ and the fact that $\bP((t\nabla_x)^{\beta}(I-\bP)\hat{f})=0$ for any multi-index $\beta$, notice that
\begin{align}
	\bP(Z^\beta(I-\bP)\hat{f})&=\bP\bigg(\sum_{\substack{|\beta_1|+|\beta_2|=|\beta|\\
|\beta_1|\leq |\beta|-1}}C_{\beta_1,\beta_2}(itk)^{\beta_1}(\nabla_v)^{\beta_2}(I-\bP)\hat{f}\bigg)\\
&=\bP\bigg(\sum_{\substack{|\beta_1|+|\beta_2|=|\beta|\\
		|\beta_1|\leq |\beta|-1}}\sum_{|\beta_3|+|\beta_4|=|\beta_1|}C_{\beta_1,\beta_2,\beta_3,\beta_4}(\nabla_v)^{\beta_2+\beta_4}Z^{\beta_3}(I-\bP)\hat{f}\bigg).
\end{align}
Since we are doing the projection onto the kernel of $\cL$, we can safely move all the $v$-derivatives and weights on $\sqrt{\mu}$. Thus
\begin{align}
	\label{bd:PZI-P}
	\cA[\bP(Z^{\beta}(I-\bP)\hat{f})]\lesssim \sum_{|\beta_1|\leq {|\beta|-1}}\cA[Z^{\beta_{1}}(I-\bP)\hat{f}].
\end{align}
From the commutation properties \eqref{bd:commZ}-\eqref{bd:commZL2} and \eqref{bd:equivA}, we obtain
\begin{equation}
	\label{bd:commZI-PI-P}
	\big|\jap{[\cL,Z^\beta](I-\bP)\hat{f}, Z^{\beta}(I-\bP)\hat{f}}\big|\leq \frac{C_1}{200}\cA[{Z^{\beta}(I-\bP)\hat{f}}]+ C\sum_{|\beta_1|\leq |\beta|-1}\cA[Z^{\beta_{1}}(I-\bP)\hat{f}].
\end{equation}
It remains to control the term involving the macroscopic part, where we crucially exploit the cancellation
\begin{equation}
	\label{eq:nicecanc}
	\jap{(t\nabla_x)^{\beta}\cL(I-\bP)\hat{f},(t\nabla_x)^\beta\bP \hat{f}}=0
\end{equation}
to write 
\begin{align}
\label{eq:japmicmac1}	&\jap{Z^\beta\cL(I-\bP)\hat{f},Z^{\beta}\bP \hat{f}}=\sum_{\substack{|\beta_1|+|\beta_2|=|\beta|\\
		|\beta_1|\leq |\beta|-1}}C_{\beta_1,\beta_2}\jap{\cL[(itk)^\beta(I-\bP)\hat{f}],(itk)^{\beta_1}(\nabla_v)^{\beta_2}\bP \hat{f}}\\
\label{eq:japmicmac2}	&\quad +\sum_{\substack{|\beta_1|+|\beta_2|=|\beta|\\
			|\beta_1|\leq |\beta|-1}}C_{\beta_1,\beta_2}\jap{(itk)^{\beta_1}(\nabla_v)^{\beta_2}\cL[(I-\bP)\hat{f}],(itk)^{\beta}\bP \hat{f}}\\
	\notag	&\quad +\sum_{\substack{|\beta_1|+|\beta_2|=|\beta|\\
				|\beta_3|+|\beta_4|=|\beta|\\
				|\beta_1|\leq |\beta|-1, \, |\beta_3|\leq |\beta|-1}}C^{\beta_1,\beta_2}_{\beta_3,\beta_4}\jap{(itk)^{\beta_1}(\nabla_v)^{\beta_2}\cL[(I-\bP)\hat{f}],(itk)^{\beta_3}(\nabla_v)^{\beta_4}\bP \hat{f}}.
\end{align} 
Since $(t\nabla_x)^{\beta}=(Z-\nabla_v)^\beta$, moving all the $v$-derivatives on the Maxwellian,  it is not hard to show that 
\begin{equation}
	\notag
\big|\jap{\cL[(itk)^\beta(I-\bP)\hat{f}],(itk)^{\beta_1}(\nabla_v)^{\beta_2}\bP \hat{f}}\big|\lesssim |tk|^{\beta_1}|(\hat{\rho},\hat{\sfm},\hat{\sfe})|\sum_{|\tilde{\beta}|\leq |\beta|}\sqrt{\cA[Z^{\tilde{\beta}}(I-\bP)\hat{f}]}
\end{equation}
For the term in \eqref{eq:japmicmac2}, we first move $(itk)^{\beta}$ on $I-\bP$ and $(itk)^{\beta_1}(\nabla_v)^{\beta_2}$ on $\bP$ and then argue as above. The remaining terms are lower order and can be controlled analogously to finally obtain 
\begin{align}
\label{bd:ZI-PZP}	\big|\jap{Z^\beta\cL(I-\bP)\hat{f},Z^{\beta}\bP \hat{f}}\big|\lesssim\,& \sqrt{\cA[Z^\beta(I-\bP)\hat{f}]}\sum_{|\beta_1|\leq |\beta|-1}|tk|^{\beta_1}|(\hat{\rho},\hat{\sfm},\hat{\sfe})|\\
	&+\sum_{|\beta_1|\leq |\beta|-1}\sqrt{\cA[Z^{\beta_1}(I-\bP)\hat{f}]}|tk|^{\beta_1}|(\hat{\rho},\hat{\sfm},\hat{\sfe})|.
\end{align}
Combining \eqref{eq:enIdZf}, \eqref{bd:lwZI-Pf}, \eqref{bd:PZI-P}, \eqref{bd:commZI-PI-P} and \eqref{bd:ZI-PZP} we have 
	\begin{align}
	\label{bd:nukZf}
	\frac12\frac{|k|}{\nu}\frac{\dd }{\dd t}&\normL{Z^\beta \hat{f}}^2+|k|\frac{C_1}{100} \notag\cA[Z^\beta(I-\bP)\hat{f}]\lesssim |k|\sqrt{\cA[Z^{\beta_1}(I-\bP)\hat{f}]}\sum_{|\beta_1|\leq |\beta|-1}|tk|^{\beta_1}|(\hat{\rho},\hat{\sfm},\hat{\sfe})|\\
	&+|k|\sum_{|\beta_1|\leq |\beta|-1}\sqrt{\cA[Z^{\beta_1}(I-\bP)\hat{f}]}|tk|^{\beta_1}|(\hat{\rho},\hat{\sfm},\hat{\sfe})|+|k|\sum_{|\beta_1|\leq |\beta|-1}\cA[Z^{\beta_1}(I-\bP)\hat{f}].
\end{align}
From Young's inequality, for any $\tilde{\delta}>0$ we have 
\begin{align}
	\label{bd:err1}|k|&\sqrt{\cA[Z^{\beta_1}(I-\bP)\hat{f}]}\sum_{|\beta_1|\leq |\beta|-1}|(tk)^{\beta_1}(\hat{\rho},\hat{\sfm},\hat{\sfe})|\leq\, \tilde{\delta} \nu \cA[Z^{\beta}(I-\bP)\hat{f}]\\
	\notag&\quad +\tilde{\delta}^{-1}\frac{|k|^2}{\nu}\sum_{|\beta_1|\leq |\beta|-1}|tk|^{2\beta_1}|(\hat{\rho},\hat{\sfm},\hat{\sfe})|^2.
\end{align}
Arguing analogously for the remaining terms and taking into account Lemma \ref{lem:equivZP}, we prove \eqref{bd:enknuZf}.

To prove \eqref{bd:enZI-Pf}, from \eqref{eq:BoltzTayI-P} we compute  
	\begin{align}
	\label{eq:enIdZI-Pf}\frac12\frac{\dd }{\dd t}&\normL{Z^\beta (I-\bP)\hat{f}}^2+\nu \jap{\cL(Z^{\beta}(I-\bP)\hat{f}),Z^{\beta}(I-\bP)\hat{f}}\\
	\label{eq:enIdZI-Pf0}	&=\Re\jap{Z^\beta\bP(ik\cdot v\hat{f}),Z^\beta(I-\bP)\hat{f}} -	\Re\jap{Z^\beta(ik\cdot v\bP \hat{f}),Z^\beta(I-\bP)\hat{f}}\\
	&\qquad +\nu\jap{[\cL,Z^\beta](I-\bP)\hat{f}, Z^{\beta}(I-\bP)\hat{f}}.
\end{align}
For the last term in \eqref{eq:enIdZI-Pf0}, using the Cauchy-Schwarz inequality and standard properties of the Maxwellian we get 
\begin{align}
	\label{bd:ZvPZI-P}\big|\jap{Z^\beta(ik\cdot v  \bP \hat{f}),Z^\beta(I-\bP)\hat{f}}\big|\leq \,&\tilde{\delta} \nu\cA[Z^\beta(I-\bP)\hat{f}]\\
	&+\tilde{\delta}^{-1}\frac{|k|^2}{\nu}\sum_{|\tilde{\beta}|\leq |\beta|}|tk|^{2\tilde{\beta}}|(\hat{\rho},\hat{\sfm},\hat{\sfe})|^2.
\end{align} 
Notice that here we pay a large constant in front of the macroscopic dissipation up to order $|\beta|$. This is why we have to choose $c_1\ll c_0$ in the definition of the energy functional \eqref{def:ETd}.

To control the first term on the right-hand side of \eqref{eq:enIdZI-Pf0}, we can again exploit a cancellation as in \eqref{eq:nicecanc} to remove the highest order  $t\nabla_x$-derivatives. In particular, we have
\begin{equation}
	\label{eq:canc}
	\jap{(tik)^\beta\bP(ik\cdot v\hat{f}),(tik)^\beta(I-\bP)\hat{f}}=0.
\end{equation}
Hence, writing $(ik\cdot v\hat{f})=(ik\cdot v\bP\hat{f})+(ik\cdot v(I-\bP)\hat{f}) $, and arguing as done in \eqref{eq:japmicmac1} to get \eqref{bd:ZI-PZP}, we infer 
\begin{align}
	\label{bd:ZPvZI-P}\big|&\jap{Z^\beta\bP(ik\cdot v\hat{f}),Z^\beta(I-\bP)\hat{f}}\big|\leq \,\tilde{\delta} (\nu+|k|)\cA[Z^\beta(I-\bP)\hat{f}]\\
	\label{bd:ZPvZI-P1}	&+C_{\tilde{\delta}}\frac{|k|^2}{\nu}\sum_{|\beta_1|\leq |\beta|-1}|tk|^{2\beta_1}|(\hat{\rho},\hat{\sfm},\hat{\sfe})|^2+C|k|\sum_{|\beta_1|\leq |\beta|-1}\cA[Z^{\beta_1}(I-\bP)\hat{f}].
\end{align} 
Therefore, from the energy identity \eqref{eq:enIdZI-Pf} we first combine \eqref{bd:lwZI-Pf}, \eqref{bd:PZI-P}, \eqref{bd:commZI-PI-P},  \eqref{bd:ZvPZI-P} and \eqref{bd:ZPvZI-P}.  Then, appealing to Lemma \ref{lem:equivZP} and using that $|k|\leq \delta_0^{-1}\nu$ we prove \eqref{bd:enZI-Pf} upon choosing $\tilde{\delta}$ in \eqref{bd:ZvPZI-P} and \eqref{bd:ZPvZI-P} sufficiently small.

To get \eqref{bd:enZvMI-Pf}, we can proceed as done to obtain \eqref{bd:enZI-Pf} but we need to handle the commutator between the weight and $\cL$. In particular, the energy identity has the same structure of \eqref{eq:enIdZI-Pf} with the extra term 
\begin{equation}
	\nu \jap{[\cL,\jap{v}^M]Z^\beta(I-\bP)\hat{f},Z^\beta(I-\bP)\hat{f}}.
\end{equation}
Appealing to \eqref{bd:commWl}, we get 
\begin{equation}
	\notag
		\nu\big| \jap{[\cL,\jap{v}^M]Z^\beta(I-\bP)\hat{f},Z^\beta(I-\bP)\hat{f}}\big|\lesssim\nu \norm{\mu^{\delta}Z^\beta(I-\bP)\hat{f}}^2\lesssim \nu \cA[Z^\beta(I-\bP)\hat{f}],
\end{equation}
whence proving \eqref{bd:enZvMI-Pf}.

To handle the terms with the $v$-derivatives, observe that 
\begin{align}
	\notag\de_t&\nabla_v(I-\bP)f+v\cdot \nabla_x \nabla_v(I-\bP)f+\nu\cL\nabla_v(I-\bP)f\\
	\label{eq:BoltzTayI-Pv}&=-\nabla_x f+ \nu[\cL,\nabla_v](I-\bP)f+\nabla_v\bP (v\cdot \nabla_x f)-v\cdot \nabla_x \nabla_v\bP f.
\end{align}
On the left-hand side we have exactly the same structure we had without $\nabla_v$. On the right-hand side of \eqref{eq:BoltzTayI-Pv}, the first three terms are different with respect to the case without $v$-derivatives. The term $\nabla_x f$ is the most dangerous one. Notice that here we can always pay a large constant for error terms involving the dissipation without $v$-derivatives since we choose $2^{-\mathtt{C}_1}\ll 2^{-\mathtt{C}_0}$ in the definition of the energy functional \eqref{def:ETd}. 

To prove \eqref{bd:enZI-Pfnabla} and \eqref{bd:enZvMI-Pfnabla}, observe that by the commutation properties in Lemma \ref{lemma:commutation}, for $j=0,1$ we have 
\begin{align}
\label{bd:Tayeasy}
\nu	\big|&\brak{\brak{v}^{j(M+q_{\gamma,s})}Z^\beta[\cL,\nabla_v](I-\bP)\hat{f},\brak{v}^{j(M+q_{\gamma,s})}Z^\beta(I-\bP)\hat{f}}\big|\\
&\leq \tilde{\delta}\nu \cA[\brak{v}^{j(M+q_{\gamma,s})}Z^\beta(I-\bP)\hat{f}]+C_{\tilde{\delta}}\nu \sum_{|\tilde{\beta}|\leq |\beta|}\cA[\jap{v}^{jM}Z^{\tilde{\beta}}(I-\bP)\hat{f}],
\end{align}
where we also used that $q_{\gamma,s}\leq0$ to obtain the last terms inside the sum above.

For the error terms generated by $\nabla_v \bP(v\cdot \nabla_x f)=\nabla_v \bP(v\cdot \nabla_x \bP f)+\nabla_v \bP(v\cdot \nabla_x (I-\bP)f)$, moving all the $v$-derivatives and weights on the term containing $\bP$,  we get that
\begin{align}
	\big|&\brak{\jap{v}^{j(M+q_{\gamma,s})}\nabla_v \bP(ik\cdot v \hat{f}),\jap{v}^{j(M+q_{\gamma,s})}\nabla_vZ^\beta(I-\bP)\hat{f}}\big|\\
	&\lesssim (|k|+\nu)\cA[Z^\beta(I-\bP)\hat{f}]+\frac{|k|^2}{\nu}\sum_{|\tilde{\beta}|\leq|\beta|}|tk|^{2\tilde{\beta}}|(\hat{\rho},\hat{\sfm},\hat{\sfe})|^2. 
\end{align}

For the error terms arising form $-\nabla_x f$, we argue as follows
\begin{align}
\notag	\big|&\brak{\brak{v}^{j(M+q_{\gamma,s})}Z^\beta(ik \hat{f}),\brak{v}^{j(M+q_{\gamma,s})}Z^\beta \nabla_v(I-\bP)\hat{f}}\big|\leq\, \\
&\big|\brak{\brak{v}^{j(M+q_{\gamma,s})}Z^\beta(ik \bP \hat{f}),\brak{v}^{j(M+q_{\gamma,s})}Z^\beta \nabla_v(I-\bP)\hat{f}}\big|\\
	&+\big|\brak{\brak{v}^{j(M+q_{\gamma,s})}Z^\beta(ik(I- \bP) \hat{f}),\brak{v}^{j(M+q_{\gamma,s})}Z^\beta \nabla_v(I-\bP)\hat{f}}\big|:=\mathcal{I}^j_1+\mathcal{I}^j_2
\end{align}
For $\mathcal{I}^j_1$, integrating by parts in $v$ and using the properties of the Maxwellian we have 
\begin{equation}
	\mathcal{I}_1^j\leq \tilde{\delta}\nu\cA[Z^\beta(I-\bP)\hat{f}]+C_{\tilde{\delta}}\frac{|k|^2}{\nu}\sum_{|\tilde{\beta}|\leq |\beta|}|tk|^{2\tilde{\beta}}|(\hat{\rho},\hat{\sfm},\hat{\sfe})|^2.
\end{equation}
To handle $\mathcal{I}^0_2$, since  $M>2|\gamma|/2+s$ and $|k|\leq \delta_0^{-1}\nu$, notice that 
\begin{align*}
	\mathcal{I}_2^0\leq\,&\tilde{\tilde{\delta}} \nu\normL{Z^\beta\nabla_v(I-\bP)\hat{f}}_{L^2_{s+\gamma/2}}^2+C_{\tilde{\tilde{\delta}}}\nu \normL{Z^\beta\jap{v}^{|s+\gamma/2|}(I-\bP)\hat{f}}_{L^2}^2\\
	\leq \,&\tilde{\delta} \nu\cA[Z^\beta\nabla_v(I-\bP)\hat{f}]+C_{\tilde{\delta}}\nu \cA[Z^\beta\jap{v}^{M}(I-\bP)\hat{f}].
\end{align*}
For $\mathcal{I}^1_2$, we proceed as done to control the mixed inner product in the enhanced dissipation regime, see \eqref{bd:mixedRj}. Namely, appealing to the Cauchy-Schwarz inequality we get 
\begin{align}
	\notag\mathcal{I}^1_2&\leq\sum_{j=0}^{+\infty} |k|\norm{\mathbbm{1}_{\{2^j\leq \jap{v}\leq 2^{j+1}\}}\brak{v}^{M+q_{\gamma,s}}Z^\beta(I-\bP)\hat{f}}\norm{\mathbbm{1}_{\{2^j\leq \jap{v}\leq 2^{j+1}\}}\brak{v}^{M+q_{\gamma,s}}Z^\beta\nabla_v(I-\bP)\hat{f}}\\
	&:=\sum_{j=0}^{+\infty}\widetilde{\cR}_j
\end{align}
Combining the Young's inequality, the Gagliardo-Nirenberg inequality and using $q_{\gamma,s}<\gamma/(2s)$, as in \eqref{bd:EDRj}-\eqref{bd:mixedRj}, we obtain that
\begin{align}
\sum_{j=0}^{+\infty}\widetilde{\mathcal{R}}_j\leq\, &\tilde{\delta}\frac{|k|^2}{\nu}\norm{\brak{v}^{M+q_{\gamma,s}}Z^\beta (I-\bP)\hat{f}}^2 \\
\label{bd:mixedRTj}&+\nu\left(\tilde{\delta}   \cA[\brak{v}^{M+q_{\gamma,s}}\nabla_v Z^\beta (I-\bP)\hat{f}]+ C_{\tilde{\delta}}\cA[\brak{v}^{M} Z^\beta (I-\bP)\hat{f}]\right).
\end{align}

Hence, the proofs of \eqref{bd:enZI-Pfnabla} and \eqref{bd:enZvMI-Pfnabla} follows by computing the time derivatives and using the estimates \eqref{bd:Tayeasy}-\eqref{bd:mixedRTj} (upon choosing $\tilde{\delta}$ sufficiently small).

Finally, to prove \eqref{bd:enTDmixedI-P} notice that the good term on the left-hand side of \eqref{bd:enTDmixedI-P} is given from the $-\nabla_x f$ in \eqref{eq:BoltzTayI-Pv} since $[\jap{v}^{M+q_{\gamma,s}}\nabla_x,\bP]=0$. Similarly to \eqref{bd:enmixed}, the terms arising from the transport $v\cdot\nabla_x$ cancel out. All the remaining error terms can be treated in a similar way to what we did to get \eqref{bd:enZI-Pfnabla}-\eqref{bd:enZvMI-Pfnabla}. For instance, we control the following term as 
\begin{align*}
	\big|&\brak{\brak{v}^{M+q_{\gamma,s}}\nabla_vZ^\beta (I-\bP)\hat{f},\brak{v}^{M+q_{\gamma,s}}Z^\beta\cL (ik(I-\bP)\hat{f})}\big|\leq\\
	& C\nu(\cA[\jap{v}^{M+q_{\gamma,s}}\nabla_v(I-\bP)\hat{f}]+\cA[\jap{v}^{M}(I-\bP)\hat{f}]),
\end{align*}
where we used $|k|\leq \delta_0^{-1}\nu$ and $q_{\gamma,s}<0$. The remaining error terms can be handled analogously.
\end{proof}

We are now ready to present the proof of the monotonicity estimate \eqref{bd:energyTD4}, whence concluding the proof of Proposition \ref{prop:mono}.
\begin{proof}[Proof of \eqref{bd:energyTD4}]
	Recall the definition of $E^{T.d.}_{M,N}$ in \eqref{def:ETd} and $\mathcal{D}^{T.d.}_{M,N}$ in \eqref{def:DTD}. When computing the time-derivative of $E^{T.d.}_{M,N}$ we neglect the negative terms on the left-hand side appearing from the time derivative of $\jap{\nu t}^{-2\beta}$. Up to the constant independent of $\mathtt{C}_j$ in front of each term, the first term in \eqref{def:DTD} directly follows by the standard $L^2_v$-energy estimate for the Boltzmann equation  while the others are a suitable linear combination of the good terms in Lemma \ref{lemma:mixed} and Lemma \ref{lemma:microTD}. 
The error terms are  absorbed thanks to the choice of the constants $c_i\ll c_{i+1}$ and $1\ll \mathtt{C}_0\ll \mathtt{C}_1$. Finally, $\delta_d$ can be chosen to scale as the smallest constant in $E^{T.d.}_{M,N}$, namely $\delta_d=c_3 \tilde{c}/1000$ upon choosing $\mathtt{C}_0$ sufficiently large.
\end{proof}
\subsection{Decay estimates}
\label{subsec:decay}
In this section, we aim at proving the decay estimates in Theorem \ref{thm:LinDecEst}. Thanks to the monotonicity estimates in Proposition \eqref{prop:mono}, it is enough to reconstruct the energy functional from the available (anisotropic) dissipation. In the hard potential case $2s+\gamma\geq0$, this is a relatively simple consequence of some interpolation inequalities. On the other hand, for soft potentials the dissipation degenerates for large velocities. To overcome this problem, the standard procedure (e.g. \cite{caflisch1980boltzmann,DV05,MR2366140}) is to use the control on higher-order moments. Indeed, the monotonicity estimates \eqref{bd:energyED4} and \eqref{bd:energyTD4} are true also for the functionals  $E^{*}_{M+M',B}$ for any $M'>0$. Recalling the definition of $\mathcal{E}_{M,B}$ given in \eqref{def:linEnergy}, i.e.
\begin{align*}
\mathcal{E}_{M,B}(t)= \mathbbm{1}_{\nu/|k|\leq\delta_0}E^{e.d.}_{M,B}(t)+\mathbbm{1}_{\nu/|k|>\delta_0}E^{T.d.}_{M,B}(t),
\end{align*}
we want to prove the following. 
\begin{lemma}
\label{lem:splitdecay}
Let $\mathcal{E}_{M,B}$ be the functional defined in \eqref{def:linEnergy} and define $\varpi=(|\gamma|(2-s)+2s|q_{\gamma,s}|)/(1+s)$. Then, for any $M'>0$, there exists constants $c,C>0$ such that for all $R\geq1$,  
\begin{equation}
\label{bd:decaweight}	\frac{\dd }{\dd t} \mathcal{E}_{M,B}(t)\leq-c\frac{\lambda_{\nu,k}}{R^{\varpi}}\mathcal{E}_{M,B}(t)+C\frac{\lambda_{\nu,k}}{R^{|\gamma+2s|+2M'}}\mathcal{E}_{M+M',B}(0),
\end{equation}
where $\lambda_{\nu,k}$ is defined in \eqref{def:lambdanuk_intro}.
\end{lemma}
Having at hand the inequality \eqref{bd:decaweight}, whose proof we postpone at the end of this section, we are ready to prove Theorem \ref{thm:LinDecEst}.

\begin{proof}[Proof of Theorem \ref{thm:LinDecEst}]
For $0\leq t\leq1$ we have nothing to prove, therefore we assume $t>1$ in the sequel. For simplicity of notation, we write $\lambda$ instead of $\lambda_{\nu,k}$. 
Let $0<p<1$ to and choose $R$ in \eqref{bd:decaweight} as 
\begin{equation}
	R=\brak{\lambda t}^{\frac{p}{\varpi}}.
\end{equation} 
Then, calling $\delta_p=c(1-p)/2$, 
 combining the inequality \eqref{bd:decaweight} with the choice of $R$ we get
\begin{align}
\label{bd:decMp}
	\frac{\dd}{\dd t}\left(\e^{\delta_p(\lambda t)^{1-p}}\mathcal{E}_{M,B}(t)\right)\lesssim  \frac{\lambda\e^{\delta_p(\lambda t)^{1-p}}}{\brak{\lambda t}^{\frac{p(|\gamma+2s|+2M')}{\varpi}}}\mathcal{E}_{M+M',B}(0).
\end{align}
Define 
\begin{equation}
	M''=\frac{p(|\gamma+2s|+2M')}{\varpi},
\end{equation}
Choosing $M'$ sufficiently large, we have $M''>2$. Integrating in time \eqref{bd:decMp}, we get 
\begin{align}
	\label{bd:pfendeca}
	\mathcal{E}_{M,B}(t)\lesssim \e^{-\delta_p(\lambda t)^{1-p}}\mathcal{E}_{M,B}(0)+  \mathcal{E}_{M+M',B}(0)\e^{-\delta_p(\lambda t)^{1-p}}\int_0^t\frac{\lambda\e^{\delta_p(\lambda \tau)^{1-p}}}{\brak{\lambda \tau}^{M''}}\dd \tau.
\end{align}
To control the last integral in the inequality above, first observe that 
\begin{align}
	\int_0^t\frac{\lambda\e^{\delta_p(\lambda \tau)^{1-p}}}{\brak{\lambda \tau}^{M''}}\dd \tau&=\left(\int_0^{\lambda t/2}+\int_{\lambda t/2}^{\lambda t}\right)\frac{\e^{\delta_ps^{1-p}}}{\brak{s}^{M''}}\dd s\lesssim \e^{\delta_p(\lambda t/2)^{1-p}}+\frac{1}{\brak{\lambda t/2}^{M''}}\int_{\lambda t/2}^{\lambda t}\e^{\delta_ps^{1-p}}\dd s.
\end{align}
Having that
\begin{align}
	\int_{\lambda t/2}^{\lambda t}\e^{\delta_ps^{1-p}}\dd s=\int_{\lambda t/2}^{\lambda t}s^p(s^{-p}\e^{\delta_ps^{1-p}})\dd s\lesssim (\lambda t)^p\e^{\delta_p(\lambda t)^{1-p}},
\end{align}
 we  obtain
 \begin{equation}
 	\e^{-\delta_p(\lambda t)^{1-p}}\int_0^t\frac{\lambda e^{c(\lambda \tau)^{1-p}}}{\brak{\lambda \tau}^{M''}}\dd \tau\lesssim \e^{-\delta_p((\lambda t)^{1-p}-(\lambda t/2)^{1-p})}+\frac{(\lambda t)^p}{\brak{\lambda t/2}^{M''}}\lesssim \frac{1}{\brak{\lambda t}^{\widetilde{M}}},
 \end{equation}
where $\widetilde{M}=M''-p>1$ and in the last inequality we used that $0<p<1$ to bound the exponential term with the polynomially decaying one. Therefore, combining the bound above with \eqref{bd:pfendeca}, we finally get
 \begin{align}
 	\mathcal{E}_{M,B}(t)\lesssim \e^{-\delta_p(\lambda t)^{1-p}}\mathcal{E}_{M,B}(0)+  \frac{1}{\brak{\lambda t}^{\widetilde{M}}}\mathcal{E}_{M+M',B}(0)
 \end{align}
whence proving \eqref{bd:lindecay}.
\end{proof}
\smallskip

It thus remain to prove Lemma \ref{lem:splitdecay}. 

\begin{proof}[Proof of Lemma \ref{lem:splitdecay}]
From the definition of $\mathcal{E}_{M,B}(t)$ \eqref{def:linEnergy} and $\mathcal{D}_{M,B}(t)$ \eqref{def:linDiss}, we study the enhanced dissipation regime ($\nu/|k|\leq\delta_0$) and the Taylor dispersion one ($\nu/|k|>\delta_0$) separately.

 \medskip

\noindent $\diamond$ \textbf{Taylor dispersion regime.} When $\nu/|k|>\delta_0$, recall the definitions of the energy $E^{T.d.}_{M,N}$ and dissipation $\cD^{T.d.}_{M,N}$ given in \eqref{def:ETd} and \eqref{def:DTD} respectively. We need to reconstruct the energy functional from the available dissipation. Thanks to the equivalence \eqref{bd:equivETd}, it is enough to reconstruct only the positive terms in $E^{T.d.}_{M,B}$.

Exploiting the integrability properties of the Maxwellian and  $M+q_{\gamma,s}>1$, we get
\begin{align}
	\label{bd:lowTD1}
	\frac{|k|^2}{\nu}\big(&\normL{Z^{\beta}\bP \hat{f}}^2+\normL{\brak{v}^{M+q_{\gamma,s}}Z^\beta(I-\bP)\hat{f}}^2\big)\gtrsim |k|(\nu^{-1}|k|\normL{Z^{\beta}\hat{f}}^2)\\
	&+\frac{|k|^2}{\nu}\big(\mathbbm{1}_{\beta=0}\normL{\hat{f}}^2+\normL{Z^\beta(I-\bP)\hat{f}}^2+\normL{\brak{v}^{M+q_{\gamma,s}}Z^\beta(I-\bP)\hat{f}}^2).
\end{align}
Next, first note the following
	\begin{equation}
		\label{bd:recon}
		\norm{\chi_{|v|\leq R} g}_{L^2}^2\geq \norm{g}_{L^2}^2-\norm{\chi_{|v|>R}g}_{L^2}^2\geq \norm{g}_{L^2}^2-R^{-2M'}\normL{\jap{v}^{M'}g}_{L^2}^2,
	\end{equation}
	for any $M'_\star>0$.
To recover the weighted term without $v$-derivatives in $E^{T.d.}_{M,B}$, we use the last term in \eqref{def:DTD}. Namely, in view of Proposition \ref{lemma:coercive}, since $\nu\gtrsim |k|^2/\nu$, we have 
 \begin{align}
 	\label{bd:lowTD2}
 \notag	&\nu \cA[\jap{v}^MZ^\beta(I-\bP)\hat{f}]\gtrsim \frac{|k|^2}{\nu}\normL{\brak{v}^{M}\mathbbm{1}_{|v|\leq R}Z^\beta(I-\bP)\hat{f}}^2_{L^2_{\gamma/2+s}}\\
 &\hspace{2cm}\gtrsim \frac{\nu^{-1}|k|^2}{R^{|\gamma
 +2s|}}\big(\normL{\brak{v}^{M}Z^\beta (I-\bP)\hat{f}}^2-\normL{\mathbbm{1}_{|v|>R}\brak{v}^MZ^\beta(I-\bP)\hat{f}}^2\big)\\
 &\hspace{2cm} \gtrsim \frac{\nu^{-1}|k|^2}{R^{|\gamma+2s|}}\bigg(\normL{\brak{v}^{M}Z^\beta (I-\bP)\hat{f}}^2-\frac{1}{R^{2M'}}E^{T.d.}_{M+M',B}(0)\bigg).
 \end{align}
In the last inequality we used that $E^{T.d}_{M,N}$, thanks to Proposition \ref{prop:mono}, is non-increasing for any $M$ and \eqref{bd:recon}.  
Analogously, we reconstruct the piece of $E^{T.d.}_{M,N}$ involving $v$-derivatives as follows
\begin{align}
	&\cA[\jap{v}^{j(M+q_{\gamma,s})}\nabla_vZ^\beta (I-\bP)\hat{f}]\gtrsim \norm{\mathbbm{1}_{|v|\leq R}\jap{v}^{j(M+q_{\gamma,s})}\nabla_vZ^\beta (I-\bP)\hat{f}}^2_{L^2_{\gamma/2+s}}\\
	\label{bd:AdissTD}
	&\gtrsim \frac{1}{R^{|\gamma+2s|}} \left(\norm{\jap{v}^{j(M+q_{\gamma,s})}\nabla_vZ^\beta(I-\bP) \hat{f}}^2_{L^2}-\frac{1}{R^{2M'}}E^{T.d.}_{M+M',B}(0)\right).
\end{align}
Combining \eqref{bd:equivETd}, \eqref{bd:lowTD1}, \eqref{bd:lowTD2} and \eqref{bd:AdissTD}, since $|k|\gtrsim \nu^{-1}|k|^2$, we get 
\begin{equation}
\label{bd:lwDTd}
	\mathcal{D}^{T.d.}_{M,B}(t)\gtrsim \frac{\nu^{-1}|k|^2}{R^{|\gamma+2s|}} E^{T.d.}_{M,B}(t)-\frac{\nu^{-1}|k|^2}{R^{|\gamma+2s|+2M'}} E^{T.d.}_{M+M',B}(0).
\end{equation}
Therefore 
\begin{equation}
	\frac12\frac{\dd }{\dd t}E^{T.d}_{M,B}(t)\lesssim -\frac{\nu^{-1}|k|^2}{R^{|\gamma+2s|}} E^{T.d.}_{M,B}(t)+\frac{\nu^{-1}|k|^2}{R^{|\gamma+2s|+2M'}} E^{T.d.}_{M+M',B}(0).
\end{equation}
In light of the definition of $\lambda_{\nu,k}$  \eqref{def:lambdanuk_intro} and $\mathcal{E}_{M,B}$ \eqref{def:linEnergy}, since $|\gamma+2s|\leq \varpi$ the bound \eqref{bd:decaweight} is proved in the Taylor dispersion regime.

 \medskip

\noindent$\diamond$ \textbf{Enhanced dissipation regime.} 
The idea of proof is very similar to the previous one, we only need to be more careful with the right scaling of the $(\nu,k)$-dependent coefficients when reconstructing the energy functional $E^{e.d.}_{M,B}$ \eqref{def:Eed} from the dissipation $\mathcal{D}^{e.d.}_{M,N}$ \eqref{def:Ded}.
In this case, recall that 
\begin{equation}
	\label{bd:Eedequiv}
	E^{e.d.}_{M,B}(t)\approx \sum_{\alpha+|\beta|\leq B}\frac{2^{-\mathtt{C}\beta}\brak{k}^{\alpha}}{\brak{\nu t}^{2\beta}}\left(\normL{\brak{v}^MZ^{\beta}\hat{f}}^2+a_{\nu,k}\normL{\brak{v}^{M+q_{\gamma,s}}Z^\beta\nabla_v \hat{f}}^2\right).
\end{equation}
To reconstruct the term without $v$-derivatives, we cannot proceed as in \eqref{bd:lowTD2} since for large $|k|$'s we have $\nu\ll \lambda_{\nu,k}$. Thus, we have to exploit the good term generated by the mixed inner product. Namely, in view of Proposition \ref{prop:mono} we know that $E^{e.d.}_{M,B}$ is non-increasing for any $M$. Thus
\begin{align}
	\label{bd:reconL2}
	b_{\nu,k}|k|^2\norm{\jap{v}^{M+q_{\gamma,s}}Z^\beta \hat{f}}^2\gtrsim \frac{\lambda_{\nu,k}}{R^{2|q_{\gamma,s}|}}\left(\norm{\jap{v}^{M}Z^\beta \hat{f}}^2_{L^2_v}-\frac{1}{R^{2M'}}E^{e.d.}_{M+M',N}(0)\right)
\end{align}
where we recall that in this case
\begin{equation}
	\lambda_{\nu,k}=\frac{\delta_1}{b_0}|k|^2b_{\nu,k}=\delta_1 \nu^{\frac{1}{1+2s}}|k|^{\frac{2s}{1+2s}}.
\end{equation}
To reconstruct the piece with $v$-derivatives, we need to exploit an interpolation inequality. Combining Proposition \ref{lemma:coercive} with Lemma \ref{lemma:weightedHs}, we have 
\begin{align}
	\label{bd:Adiss1}
	\cA[\jap{v}^{M+q_{\gamma,s}}\nabla_vZ^\beta \hat{f}]\gtrsim \frac{1}{R^{|\gamma|(2-s)}} \norm{\chi_{|v|\leq R}\jap{v}^{M+q_{\gamma,s}}\nabla_vZ^\beta \hat{f}}^2_{H^s}.
\end{align}
From the definitions of $\lambda_{\nu,k}$ \eqref{def:lambdanuk_intro} and $a_{\nu,k}$ \eqref{def:abnuk}, notice that 
\begin{equation}
\lambda_{\nu,k}a_{\nu,k}=\delta_1^{\frac{1}{1+s}}a_0^{\frac{s}{1+s}} (\lambda_{\nu,k})^{\frac{s}{1+s}}(\nu a_{\nu,k})^{\frac{1}{1+s}}
\end{equation}
Therefore, using the Gagliardo-Nirenberg inequality we have 
\begin{align}
\label{bd:gap}
	\lambda_{\nu,k} a_{\nu,k}\norm{\nabla_v g}_{L^2(B_R)}^2&\lesssim\delta_1^{\frac{1}{1+s}}a_0^{\frac{s}{1+s}}(\nu a_{\nu,k}\norm{\nabla_v g}_{H^s(B_R)}^2)^{\frac{1}{1+s}}(\lambda_{\nu,k}\norm{g}_{L^2(B_R)}^2)^\frac{s}{1+s}\\
	&\lesssim\delta_1^{\frac{1}{1+s}}a_0^{\frac{s}{1+s}}\left( \nu a_{\nu,k} \norm{\nabla_v g}_{H^s(B_R)}^2+\lambda_{\nu,k}\norm{g}_{L^2(B_R)}^2\right).
\end{align}
Since $a_0, \delta_1\ll 1$, combining \eqref{bd:reconL2}, \eqref{bd:Adiss1} with \eqref{bd:gap}, we infer
\begin{align}
	\label{bd:recova}
	\nu a_{\nu,k}&\cA[\jap{v}^{M+q_{\gamma,s}}\nabla_v Z^\beta \hat{f}]+b_{\nu,k}|k|^2\norm{\jap{v}^{M+q_{\gamma,s}}Z^\beta \hat{f}}^2_{L^2_v} \\
	&\gtrsim \frac{\nu a_{\nu,k}}{R^{|\gamma|(2-s)}} \norm{\chi_{|v|\leq R}\jap{v}^{M+q_{\gamma,s}}\nabla_vZ^\beta \hat{f}}^2_{H^s}+\frac{\lambda_{\nu,k}}{R^{2|q_{\gamma,s}|}}\norm{\chi_{|v|\leq R}\jap{v}^{M}Z^\beta \hat{f}}^2_{L^2_v}\\
	&\gtrsim \lambda_{\nu,k}\frac{a_{\nu,k}}{R^{\varpi}}\norm{\chi_{|v|\leq R}\jap{v}^{M+q_{\gamma,s}}\nabla_v Z^\beta \hat{f}}^2_{L^2_v}\\
	&\gtrsim \frac{\lambda_{\nu,k}}{R^{\varpi}}\left(a_{\nu,k}\norm{\jap{v}^{M+q_{\gamma,s}}\nabla_v Z^\beta \hat{f}}^2_{L^2_v}-\frac{1}{R^{2M'}}E^{e.d.}_{M+M',B}(0)\right)
\end{align}
where we also used $\varpi=(|\gamma|(2-s)+2s|q_{\gamma,s}|)/(1+s)$, the inequality \eqref{bd:recon} and the fact that $E^{e.d.}_{M,B}$ is non-increasing for any $M$.
Hence, recalling the definition of $\mathcal{D}^{e.d.}_{M,N}$, combining \eqref{bd:reconL2}, \eqref{bd:recova} and \eqref{bd:Eedequiv}, we obtain
\begin{equation}
\label{bd:lwDed}
	\mathcal{D}^{e.d.}_{M,B}\gtrsim \frac{\lambda_{\nu,k}}{R^{\varpi}} E^{e.d.}_{M,B}-\frac{\lambda_{\nu,k}}{R^{\varpi+2M'}} E^{e.d.}_{M+M',B}(0).
\end{equation}
Consequently, appealing to \eqref{bd:enhanced} we get
\begin{align}
	\frac{\dd}{\dd t}E^{e.d.}_{M,B}\leq -c\frac{\lambda_{\nu,k}}{R^{\varpi}}E^{e.d.}_{M,B}+ C\frac{\lambda_{\nu,k}}{R^{\varpi+2M'}}E^{e.d.}_{M+M',N}(0),
\end{align}
whence proving \eqref{bd:decaweight} in the enhanced dissipation regime since $|\gamma+2s|\leq \varpi$.

\end{proof}

\section{Nonlinear estimates}

In this section we aim at proving Proposition \ref{prop:bootstrap}, which requires the control of several error terms arising from the time derivative of the energy functionals. We first define such error terms and we collect their bounds Propositions \ref{lem:linerr}-\ref{lem:momerrLinfty}. Having at hand the aforementioned bounds, we finally prove Proposition \ref{prop:bootstrap}. 
Recall that 
\begin{align}
\label{def:EDpf}\mathcal{E}(t) & = \int_{\abs{k} \leq \delta_0^{-1} \nu} \brak{\lambda_{\nu,k} t}^{2J}  E^{T.d.}_{M,B}(t,k) \dd k +  \int_{\abs{k} > \delta_0^{-1} \nu} \brak{\lambda_{\nu,k} t}^{2J}  E^{e.d.}_{M,B}(t,k) \dd k, \\
\label{def:ELFpf}\mathcal{E}_{LF}(t) & = \sup_{k : \abs{k} \leq \delta_0^{-1} \nu} \brak{\lambda_{\nu,k} t}^{2J'}  E^{T.d.}_{M',B'}(t,k) + \sup_{k : \abs{k} > \delta_0^{-1} \nu}  \brak{\lambda_{\nu,k} t}^{2J'} E^{e.d.}_{M',B'}(t,k),\\
\label{def:EmomD}\mathcal{E}_{mom}(t) & = \int_{\abs{k} \leq \delta_0^{-1} \nu} E^{T.d.}_{M+M_J,B}(t,k) \dd k +  \int_{\abs{k} > \delta_0^{-1} \nu}  E^{e.d.}_{M+M_J,B}(t,k) \dd k, \\ 
\label{def:EmomLF}\mathcal{E}_{mom,LF}(t) &= \sup_{k : \abs{k} \leq \delta_0^{-1} \nu} E^{T.d.}_{M'+M_{J'},B'}(t,k)  + \sup_{k : \abs{k} > \delta_0^{-1} \nu}  E^{e.d.}_{M'+M_{J'},B'}(t,k),
\end{align}
where we assume the conditions \eqref{eq:constMJB}. Associated to these energies, we have the dissipations defined as above with 
\begin{equation}
\label{def:Diss}
(\mathcal{E}_{*},E^{T.d.}_{*},E^{e.d.}_{*})\to(\mathcal{D}_{*},\mathcal{D}^{T.d.}_{*},\mathcal{D}^{e.d.}_{*}) 
\end{equation}
 Taking the time derivative of $\mathcal{E}$ and exploiting the pointwise in frequency monotonicity estimates in Proposition \ref{prop:mono}, we know that 
 \begin{align}
\notag\frac{\dd }{\dd t}\mathcal{E}\leq\,& -\delta_\star\mathcal{D}+\int_{\mathbb{R}^d}L_{J,M,B}\dd k\\
\label{eq:dtcED}&+\nu \int_{\abs{k}\leq\delta_0^{-1}\nu}\brak{ \lambda_{\nu,k} t}^{2J}NL^{T.d.}_{M,B}\dd k+\nu\int_{\abs{k}>\delta_0^{-1}\nu}\brak{ \lambda_{\nu,k} t}^{2J} NL^{e.d.}_{M,B} \dd k,
\end{align}
where $\delta_\star=\min\{\delta_e,\delta_d\}>0$ with $\delta_e,\delta_d$ being the one appearing in \eqref{bd:energyED4}-\eqref{bd:energyTD4}. The linear error term is
\begin{equation}
\label{def:Le}
L_{J,M,B}:=2J\lambda_{\nu,k}\brak{\lambda_{\nu,k} t}^{2J-1}(\mathbbm{1}_{\abs{k}\leq \delta_0^{-1}\nu}E^{T.d}_{M,B}+\mathbbm{1}_{\abs{k}> \delta_0^{-1}\nu}E^{e.d.}_{M,B}).
\end{equation}
We define the nonlinear error terms as follows: the one arising from the enhanced dissipation energy functional $E^{e.d.}_{M,B}$ is
 \begin{align}
 	\label{def:NLed}NL^{e.d.}_{M,B}  = \sum_{\alpha+|\beta|\leq B} &\frac{2^{-\mathtt{C}\beta}}{\brak{\nu t}^{2\beta}} \brak{k}^{2\alpha}\bigg(\bigg|\brak{ \brak{v}^M Z^\beta \widehat{\Gamma(f,f)}, \brak{v}^M Z^\beta \hat{f}}_{L^2_v}\bigg| \\
 	\label{def:NLed1}&  + a_{\nu,k} \bigg|\brak{ \brak{v}^{M+q_{\gamma,s}} \grad_v Z^\beta \widehat{\Gamma(f,f)}, \brak{v}^{M+q_{\gamma,s}} \grad_v Z^\beta \hat{f}}_{L^2_v} \bigg|\\
 	\label{def:NLed2}&  + b_{\nu,k} \bigg|\brak{ \brak{v}^{M+q_{\gamma,s}}  Z^\beta \widehat{ik\Gamma(f,f)}, \brak{v}^{M+q_{\gamma,s}} \grad_v Z^\beta \hat{f}}_{L^2_v} \bigg|\\
 	\label{def:NLed3}&  + b_{\nu,k} \bigg|\brak{ \brak{v}^{M+q_{\gamma,s}} \grad_v Z^\beta \widehat{\Gamma(f,f)}, \brak{v}^{M+q_{\gamma,s}}  Z^\beta ik\hat{f}}_{L^2_v}\bigg|\bigg). 
 \end{align}
 From $E^{T.d.}_{M,B}$, using that $\langle{\Gamma(f,f),f}\rangle=\langle{\Gamma(f,f),(I-\bP) f}\rangle$, we get
   \begin{align}
 \label{def:NLTd}	&NL^{T.d.}_{M,B}  = \bigg|\brak{\widehat{\Gamma(f,f)},(I-\bP) \hat{f}}_{L^2_v}\bigg|+\sum_{\alpha+|\beta|\leq B}\sum_{j=0}^1 \frac{2^{-\mathtt{C}_j\beta}}{\brak{\nu t}^{2\beta}} \bigg(\bigg|\frac{|k|}{\nu}\brak{ Z^\beta\widehat{\Gamma(f,f)}, Z^\beta \hat{f}}_{L^2_v}\bigg| \\
 \label{def:NLTd1} 	& \qquad  + c_1 \bigg|\brak{   Z^\beta(\nabla_v)^j \widehat{\Gamma(f,f)}, Z^\beta(\nabla_v)^j  (I-\bP)\hat{f})}_{L^2_v}\bigg| \\
 \label{def:NLTd2} 	& \qquad + c_2 \bigg|\brak{  \jap{v}^{M+jq_{\gamma,s}} Z^\beta(\nabla_v)^j \widehat{\Gamma(f,f)}, \jap{v}^{M+jq_{\gamma,s}} Z^\beta(\nabla_v)^j  (I-\bP)\hat{f}}_{L^2_v}\bigg|\bigg) \\
  \label{def:NLTd3}	& +\frac{c_0}{\nu}\mathcal{M}_{\Gamma}+\sum_{\alpha+|\beta|\leq B} \frac{2^{-\mathtt{C}\beta}}{\brak{\nu t}^{2\beta}}c_3\bigg(\bigg|\brak{ \brak{v}^{M+q_{\gamma,s}}  Z^\beta ik\widehat{\Gamma(f,f)}, \brak{v}^{M+q_{\gamma,s}} \grad_v Z^\beta (I-\bP)\hat{f}}_{L^2_v}\bigg| \\
  \label{def:NLTd4}	& \qquad  +  \bigg|\brak{ \brak{v}^{M+q_{\gamma,s}} \grad_v Z^\beta\widehat{\Gamma(f,f)}, \brak{v}^{M+q_{\gamma,s}}  Z^\beta  (I-\bP)ik\hat{f}}_{L^2_v}\bigg|\bigg),
 \end{align}
 where 
 \begin{align}
\label{bd:mixedNL}
	\mathcal{M}_{\Gamma}:=\,&\big|\Lambda[\widehat{\Gamma(f,f)}]\cdot  (ik \hat{\sfe})\big|+\big|\Theta[\widehat{\Gamma(f,f)}]: (ik\hat{\sfm}+(ik \hat{\sfm})^T)\big|.
\end{align}
Observe that the term above, which is given by the mixed inner product with macroscopic variables defined in \eqref{def:Mkbeta}, contains only the nonlinear terms for the equations involving the higher-order moments $\Theta$ and $\Lambda$. This is a consequence of  $\brak{\Gamma(f,f),\bP f}=0$.

The time derivative of $\mathcal{E}_{LF}$ is as \eqref{eq:dtcED} with 
\begin{align}
\label{changeDLF}
(\mathcal{E},\mathcal{D})\to(\mathcal{E}_{LF},\mathcal{D}_{LF}), \quad (J,M,B)\to (J',M',B') \quad \text{and}\quad \int \to \sup.
\end{align}
For $\mathcal{E}_{mom,*}$ instead we change \eqref{eq:dtcED} as follows
\begin{align}
\label{changeDmomD}
&(\mathcal{E},\mathcal{D})\to(\mathcal{E}_{mom},\mathcal{D}_{mom}), \quad (J,M,B)\to (0,M+M_J,B) ,\\
\label{changeDmomLF}&(\mathcal{E},\mathcal{D})\to(\mathcal{E}_{mom,LF},\mathcal{D}_{mom,LF}), \quad (J,M,B)\to (0,M'+M_{J'},B)\quad \text{and}\quad \int \to \sup
\end{align}
\begin{remark}
\label{rem:trivialEmom}
Notice that, since we are not imposing any time-decay for $\mathcal{E}_{mom,*}$, we do not have the linear error term. Namely $L_{0,M,B}=0$. 
\end{remark}
The goal is then to bound the error terms in \eqref{eq:dtcED} (with the changes \eqref{changeDLF}-\eqref{changeDmomLF}) for $\mathcal{E}$ (each) energy functional. The bounds for the linear errors are a direct consequence of the lower bounds on the dissipation functionals given in Section \ref{subsec:decay}. In particular, we show  the following in Section \ref{sec:linerr}.
\begin{proposition}
\label{lem:linerr}
Let $L_{J,M,B}$ be defined as in \eqref{def:Le}. Then, there exists constants $C_1,C_2>0$ (explicitly computable) such that
\begin{align}
&\int_0^t\int_{\mathbb{R}^d}L_{J,M,B}\dd t \dd k\leq \frac{\delta_\star}{8}\int_0^t\mathcal{D}\dd t+C_1\sup_{0\leq s\leq t}(\mathcal{E}_{mom}(s)),  \label{bd:LerrD} \\
&\int_0^t\sup_{k\in \RR^d}L_{J',M',B'}\dd t\leq \frac{\delta_\star}{8}\int_0^t\mathcal{D}_{LF}\dd t+C_2\sup_{0\leq s \leq t}(\mathcal{E}_{mom,LF}(s)), \label{bd:LerrLF}
\end{align}
where $\delta_\star>0$ is the constant in \eqref{eq:dtcED}.
\end{proposition} 

To state the bounds for the nonlinear error terms, in the following propositions we always consider $\mathcal{E}_*,\mathcal{D}_*$ to be the functionals defined in \eqref{def:EDpf}-\eqref{def:Diss} and
$NL^{e.d.}_{*}$, $NL^{T.d.}_{*}$ the nonlinear errors defined in \eqref{def:NLed}, \eqref{def:NLTd} respectively. We control separately the nonlinear errors arising from the time derivative of each of the energy functionals in \eqref{def:EDpf}-\eqref{def:EmomLF}. For the ones associated with $\mathcal{E}$, in Section \ref{sec:NLerrED} we prove the following.
\begin{proposition}
\label{prop:NLerrorL2}
The following inequality holds true
\begin{align}
	&\nu \int_{\abs{k}\leq\delta_0^{-1}\nu}\brak{ \lambda_{\nu,k} t}^{2J}NL^{T.d.}_{M,B}\dd k+\nu\int_{\abs{k}>\delta_0^{-1}\nu}\brak{ \lambda_{\nu,k} t}^{2J} NL^{e.d.}_{M,B} \dd k \notag\\
	\label{bd:NLL2}&\lesssim\sqrt{\mathcal{E}}\sqrt{\mathcal{D}}\left(\sqrt{\mathcal{D}}+\big(\nu^d\mathbbm{1}_{t\leq \nu^{-1}}+\brak{\nu/t}^{\frac{d}{2}}\mathbbm{1}_{t> \nu^{-1}}\big)\sqrt{\mathcal{E}_{LF}}\right)(t).	
\end{align}
\end{proposition}

To control the $L^\infty_k$ energy functional $\mathcal{E}_{LF}$, which is needed in \eqref{bd:NLL2} and necessary to obtain the optimal time-decay rates, in Section \ref{sec:NLerrLinfty} we show the following.
\begin{proposition}
\label{prop:NLerrorLinfty}
The following inequality holds true
\begin{align}
	&\nu \sup_{k:\abs{k}\leq\delta_0^{-1}\nu}\brak{ \lambda_{\nu,k} t}^{2J'}NL^{T.d.}_{M',B'}+\nu\sup_{k:\abs{k}>\delta_0^{-1}\nu}\brak{ \lambda_{\nu,k} t}^{2J'} NL^{e.d.}_{M',B'} \notag \\
	\label{bd:NLsup}&\lesssim\sqrt{\mathcal{D}_{LF}}\left(\sqrt{\mathcal{E}}\sqrt{\mathcal{D}}+\big(\nu^d\mathbbm{1}_{t\leq \nu^{-1}}+\brak{\nu/t}^{\frac{d}{2}}\mathbbm{1}_{t> \nu^{-1}}\big)\mathcal{E}_{LF}\right)(t).
\end{align}
\end{proposition}
\begin{remark}
As it will be clear from the proof of Proposition \ref{prop:NLerrorLinfty}, we are allowed to use $\mathcal{E}$ in \eqref{bd:NLsup} because $\mathcal{E}_{LF}$ is a lower order energy with respect to $\mathcal{E}$. Indeed, it requires less derivatives, decay and weights, a fact that is crucial to close the estimates. 
\end{remark}
We finally have to estimate the higher order moments functionals $\mathcal{E}_{mom,*}$, appearing in the bounds for the linear error terms \eqref{bd:LerrD}-\eqref{bd:LerrLF}. Recall that for the time derivative of $\mathcal{E}_{mom,*}$ we only have nonlinear error terms, since  $L_{0,M,B}=0$ as observed in Remark \ref{rem:trivialEmom}. Hence, to control the time derivative of $\mathcal{E}_{mom}$ the following proposition is sufficient.
\begin{proposition}
\label{lem:momerrL2}
The following inequality holds true
\begin{align}
	&\nu \int_{\abs{k}\leq\delta_0^{-1}\nu}NL^{T.d.}_{M+M_J,B}\dd k+\nu\int_{\abs{k}>\delta_0^{-1}\nu} NL^{e.d.}_{M+M_J,B} \dd k \notag \\
	\label{bd:NLL2mom}&\lesssim\sqrt{\mathcal{E}_{mom}}\sqrt{\mathcal{D}_{mom}}\left(\sqrt{\mathcal{D}_{mom}}+\big(\nu^d\mathbbm{1}_{t\leq \nu^{-1}}+\brak{\nu/t}^{\frac{d}{2}}\mathbbm{1}_{t> \nu^{-1}}\big)\sqrt{\mathcal{E}_{LF}}\right)(t).	
\end{align}
\end{proposition}
\noindent Similary, for the $L^\infty_k$ functional $\mathcal{E}_{mom,LF}$ we obtain the following.
\begin{proposition}
The following inequality holds true
\label{lem:momerrLinfty} 
\begin{align}
	\label{bd:NLsupmom}&\nu \sup_{k:\abs{k}\leq\delta_0^{-1}\nu}NL^{T.d.}_{M'+M_{J'},B'}+\nu\sup_{k:\abs{k}>\delta_0^{-1}\nu} NL^{e.d.}_{M'+M_{J'},B'}  \\
	\notag &\lesssim\sqrt{\mathcal{D}_{mom,LF}}\bigg((\sqrt{\mathcal{E}_{mom}}+\sqrt{\mathcal{E}})\sqrt{\mathcal{D}_{mom}}+\big(\nu^d\mathbbm{1}_{t\leq \nu^{-1}}
	+\brak{\nu/t}^{\frac{d}{2}}\mathbbm{1}_{t> \nu^{-1}}\big)\sqrt{\mathcal{E}_{LF}\mathcal{E}_{mom,LF}}\bigg)(t).
\end{align}
\end{proposition}
The proofs of Propositions \ref{lem:momerrL2} and \ref{lem:momerrLinfty} are given in Section \ref{sec:NLmom} and are obtained via minor modifications in the arguments done to arrive at \eqref{bd:NLL2} and \eqref{bd:NLsup}.

We first show how Propositions \ref{lem:linerr}-\ref{lem:momerrLinfty} imply the main bootstrap Proposition \ref{prop:bootstrap}. 

\begin{proof}[Proof of Proposition \ref{prop:bootstrap}]
We begin by proving that under the bootstrap assumptions \eqref{bd:bootEmomD}-\eqref{bd:bootELF}, the inequality \eqref{bd:bootEd} holds true with half the constant on the right-hand side of \eqref{bd:bootEd}. Indeed, integrating in time \eqref{eq:dtcED} and using the bound \eqref{bd:LerrD} and \eqref{bd:NLL2}, we have
\begin{align}
\mathcal{E}(t)&+\frac78\delta_\star\int_0^t\mathcal{D}(\tau)\dd \tau\leq \mathcal{E}(0)+C_1\sup_{0\leq s\leq t}(\mathcal{E}_{mom}(s))\\
&+C\int_0^t\bigg(\sqrt{\mathcal{E}}\mathcal{D}+\left(\nu^d\mathbbm{1}_{t\leq \nu^{-1}}+\brak{\nu/\tau}^{\frac{d}{2}}\mathbbm{1}_{t> \nu^{-1}}\right)\sqrt{\mathcal{E}_{LF}}\sqrt{\mathcal{E}}\sqrt{\mathcal{D}}\bigg)(\tau)\dd \tau.
\end{align}
From the bootstrap assumptions \eqref{bd:bootEmomD}-\eqref{bd:bootELF}, for any $t\in[0,T]$ we get
\begin{align}
&C_1\sup_{0\leq s\leq t}(\mathcal{E}_{mom}(s))\leq 4C_1\eps^2,\\
&C\int_0^t(\sqrt{\mathcal{E}}\mathcal{D})(\tau)\dd \tau\leq \eps\widetilde{C} \int_0^t\mathcal{D}(\tau)\dd \tau,\\
\notag&C\int_0^t\left(\nu^d\mathbbm{1}_{\tau\leq \nu^{-1}}+\brak{\nu/\tau}^{\frac{d}{2}}\mathbbm{1}_{\tau> \nu^{-1}}\right)(\sqrt{\mathcal{E}_{LF}}\sqrt{\mathcal{E}}\sqrt{\mathcal{D}})(\tau)\dd \tau\\
\label{bd:pfmain1}&\qquad \leq \eps^3\widetilde{C}_1\int_0^t\left(\nu^d\mathbbm{1}_{\tau\leq \nu^{-1}}+\brak{\nu/\tau}^{\frac{d}{2}}\mathbbm{1}_{\tau> \nu^{-1}}\right)^2\dd \tau+\eps \widetilde{C}_2\int_0^t\mathcal{D}(\tau)\dd \tau.
\end{align}
Since $d\geq 2$, we know that the first term on the right-hand side of \eqref{bd:pfmain1} is integrable in time. Thus, we obtain that
\begin{align}
\mathcal{E}(t)+\big(\frac78-\eps (\widetilde{C}+\widetilde{C}_2)\big)\delta_\star\int_0^t\mathcal{D}(\tau)\dd \tau\leq (1+4C_1+\eps \widetilde{C}_1)\eps^2.
\end{align}
Setting 
\begin{equation}
\mathsf{B}_1:=2(1+4C_1),
\end{equation}
and choosing $\eps$ small enough, we prove that \eqref{bd:bootEd} holds with half the constant on the right-hand side. The arguments for the other functionals are analogous and we omit their proof. 
\end{proof}

The rest of this section is dedicated to the proof of Propositions \ref{lem:linerr}-\ref{lem:momerrLinfty}. In particular, for the nonlinear error terms we need to compare the solution in different frequency regimes, therefore we state some technical estimates used throughout the proof.
\begin{lemma}
\label{lem:freqex}
Let $k,\xi \in \mathbb{R}^d$. For $a_{\nu,k}$ defined in \eqref{def:abnuk}, we have
\begin{equation}
\label{bd:trivcoeff}
\mathbbm{1}_{|k|\gtrsim \nu}a_{\nu,k}\lesssim 1, \qquad \mathbbm{1}_{|k|\approx |\xi|\gtrsim \nu}\frac{a_{\nu,k}}{a_{\nu,\xi}}\lesssim 1, \qquad \mathbbm{1}_{|k|\approx |\xi|\gtrsim \nu}\frac{a_{\nu,k}}{a_{\nu,k-\xi}}\lesssim 1.
\end{equation}
Let $\lambda_{\nu,k}$ be defined as in \eqref{def:lambdanuk_intro}. Then
\begin{equation}
\label{bd:exlambda}
\mathbbm{1}_{|k|\leq \delta_0^{-1}\nu}\brak{\lambda_{\nu,k} t}\lesssim\mathbbm{1}_{|\xi|\gtrsim \nu}\brak{\lambda_{\nu,\xi} t}.
\end{equation}
For any $0<J'\leq J-1$, the following inequality holds
\begin{equation}
\label{bd:exlambdaJ}
\mathbbm{1}_{|k|\leq \delta_0^{-1}\nu}\brak{\lambda_{\nu,k} t}^{J'}\lesssim \mathbbm{1}_{|\xi|\gtrsim \nu}\frac{1}{\brak{\nu t}}\brak{\lambda_{\nu,\xi} t}^{J}.
\end{equation}
Let  $E^{e.d.}_{M,B}$ and $E^{T.d.}_{M,B}$ be defined as in \eqref{def:Eed} and \eqref{def:ETd}  respectively. Then 
\begin{align}
\label{bd:equivalencenu}
&\mathbbm{1}_{|k|\sim \nu}E^{e.d.}_{M,B}(t,k)\approx\mathbbm{1}_{|k|\sim \nu} E^{T.d.}_{M,B}(t,k), \\
\label{bd:equivalencenu1}
 &\mathbbm{1}_{|k|\sim \nu}\cD^{e.d.}_{M,B}(t,k)\approx\mathbbm{1}_{|k|\sim \nu} \cD^{T.d.}_{M,B}(t,k).
\end{align}
\end{lemma}
\begin{proof}
The first and the second inequalities in \eqref{bd:trivcoeff} are immediate from  the definition of $a_{\nu,k}$. For the last one, when $|k|\approx |\xi|\gtrsim \nu$, notice that 
\begin{equation}
\frac{a_{\nu,k}}{a_{\nu,k-\xi}}=\left(\frac{|k-\xi|}{|k|}\right)^{\frac{2}{1+2s}}\lesssim 1.
\end{equation}
To prove \eqref{bd:exlambda}, from  the expression of $\lambda_{\nu,k}$, it is enough to observe that 
\begin{equation}
\mathbbm{1}_{|k|\leq \delta_0^{-1}\nu}\frac{|k|^2}{\nu}\lesssim \mathbbm{1}_{|\xi|\gtrsim\nu}\nu^{\frac{1}{1+2s}}|\xi|^{\frac{2s}{1+2s}}. 
\end{equation}
To prove \eqref{bd:exlambdaJ} notice that, whenever $|\xi|\gtrsim \nu$ we get
\begin{equation}
\langle\lambda_{\nu,\xi} t\rangle\approx \langle\nu^{\frac{1}{1+2s}}|\xi|^{\frac{2s}{1+2s}}t\rangle\gtrsim \brak{\nu t}.
\end{equation}
Hence 
\begin{equation}
\mathbbm{1}_{|k|\leq \delta_0^{-1}\nu}\brak{\lambda_{\nu,k} t}^{J'}\lesssim\mathbbm{1}_{|\xi|\gtrsim \nu}\frac{1}{\brak{\lambda_{\nu,\xi} t}^{J-J'}}\brak{\lambda_{\nu,\xi} t}^{J}\lesssim \mathbbm{1}_{|\xi|\gtrsim \nu}\frac{1}{\brak{\nu t}}\brak{\lambda_{\nu,\xi} t}^{J},
\end{equation}
which proves \eqref{bd:exlambdaJ}.

For the equivalence \eqref{bd:equivalencenu}, by  Cauchy-Schwartz inequality and the fact that 
\begin{equation}
(b_{\nu,k}|k|)^2= \frac{b_0^2}{a_0}a_{\nu,k}\ll a_{\nu,k},
\end{equation}
see \eqref{bd:constEd}, we get
\begin{align}
E^{e.d.}_{M,B}&\leq \frac12\sum_{\alpha+|\beta|\leq B} \frac{\jap{k}^\alpha}{\jap{\nu t}^{2\beta}}\bigg(2\norm{\jap{v}^MZ^\beta \hat{f}}^2_{L^2_v}+a_{\nu,k}(1+\frac{b_0^2}{a_0})\norm{\jap{v}^{M+q_{\gamma,s}}Z^\beta \nabla_v\hat{f}}^2_{L^2_v}\bigg)\\
E^{e.d.}_{M,B}&\geq 2^{-\mathtt{C}B-2}\sum_{\alpha+|\beta|\leq B} \frac{\jap{k}^\alpha}{\jap{\nu t}^{2\beta}}\bigg(\norm{\jap{v}^MZ^\beta \hat{f}}^2_{L^2_v}+2a_{\nu,k}(1-2\frac{b_0^2}{a_0})\norm{\jap{v}^{M+q_{\gamma,s}}Z^\beta \nabla_v\hat{f}}^2_{L^2_v}\bigg).
\end{align}
Therefore, when $|k|\approx \nu$ we deduce that 
\begin{equation}
\label{bd:eqEed5}
E^{e.d.}_{M,B}\approx \sum_{\alpha+|\beta|\leq B} \frac{1}{\jap{\nu t}^{2\beta}}\bigg(\norm{\jap{v}^MZ^\beta \hat{f}}^2_{L^2_v}+\tilde{c}\norm{\jap{v}^{M+q_{\gamma,s}}Z^\beta \nabla_v\hat{f}}^2_{L^2_v}\bigg),
\end{equation}
for some $\tilde{c}$ explicitly computable.

For $E^{T.d}_{M,B}$, recalling the definition of $\mathcal{M}$ \eqref{def:Mkbeta}, we first notice that
\begin{align}
\frac{1}{\nu}|\mathcal{M}|\leq C_m\frac{|k|}{\nu}\normL{\hat{f}}^2_{L^2_v}\leq C_m\delta_0^{-1}\normL{\hat{f}}^2_{L^2_v},
\end{align}
for some constant $C_m>0$. Thus, there exists also another $\tilde{C}_m>0$ such that
\begin{align}
&E^{T.d.}_{M,B}\leq(\frac12+c_0\delta_0^{-1}C_m)\normL{\hat{f}}^2_{L^2_v}\notag \\
\notag&+\sum_{|\beta|\leq B}\sum_{j=0}^1\frac{ 1}{\brak{\nu t}^{2\beta}} \bigg(\frac{|k|}{\nu}\normL{Z^\beta \hat{f}}^2_{L^2_v}+(c_1+c_3\delta_0^{-1}\tilde{C}_m)\normL{Z^\beta\nabla_v^j(I-\bP)\hat{f}}^2_{L^2_v} \notag \\ & \hspace{5cm}
+(c_2+c_3\delta_0^{-1}\tilde{C}_m)\normL{\jap{v}^{M+jq_{\gamma,s}}Z^\beta \nabla_v^j(I-\bP)\hat{f}}^2_{L^2_v}\bigg) 
\end{align}
and 
\begin{align}
&E^{T.d.}_{M,B}\geq(\frac12-c_0\delta_0^{-1}C_m)\normL{\hat{f}}^2_{L^2_v}\notag \\
&+2^{-\mathtt{C}_1B-2}\sum_{|\beta|\leq B}\sum_{j=0}^1\frac{ 1}{\brak{\nu t}^{2\beta}} \bigg(\frac{|k|}{\nu}\normL{Z^\beta \hat{f}}^2_{L^2_v}+(c_1-c_3\delta_0^{-1}\tilde{C}_m)\normL{Z^\beta\nabla_v^j(I-\bP)\hat{f}}^2_{L^2_v} \notag \\ & \hspace{5cm}
+(c_2-c_3\delta_0^{-1}\tilde{C}_m)\normL{\jap{v}^{M+jq_{\gamma,s}}Z^\beta \nabla_v^j(I-\bP)\hat{f}}^2_{L^2_v}\bigg).  
\end{align}
Having that $M+q_{\gamma,s}>1$, it follows immediately that $E^{T.d}_{M,B}\lesssim E^{e.d.}_{M,B}$ (for any $|k|\lesssim \nu$). On the other hand, notice that
\begin{equation}
\normL{Z^\beta \hat{f}}^2_{L^2_v}\geq \frac{c_2}{4+c_2}\normL{Z^\beta \bP \hat{f}}^2_{L^2_v}-\frac{c_2}{4}\normL{Z^\beta (I-\bP)\hat{f}}^2_{L^2_v},
\end{equation}
where in the last inequality we used $(a+b)^2\geq (1-(1+c_2/4)^{-1})a^2-c_2b^2/4$. From the integrability properties of the Maxwellian, we also know that 
\begin{equation}
\normL{Z^\beta \bP \hat{f}}^2_{L^2_v}\gtrsim \normL{\jap{v}^MZ^\beta \bP \hat{f}}^2_{L^2_v}+\normL{\jap{v}^M\nabla_vZ^\beta \bP \hat{f}}^2_{L^2_v}.
\end{equation}
Hence, combining the two inequalities above, when $|k|\approx \nu$ we get
\begin{align}
&E^{T.d.}_{M,B}\gtrsim\sum_{|\beta|\leq B}\sum_{j=0}^1\frac{ 1}{\brak{\nu t}^{2\beta}} \bigg(\normL{Z^\beta \bP \hat{f}}^2
+\tilde{c}\normL{\jap{v}^{M+jq_{\gamma,s}}Z^\beta \nabla_v^j(I-\bP)\hat{f}}^2\bigg) \gtrsim E^{e.d.}_{M,B}, 
\end{align}
where in the last inequality we used \eqref{bd:eqEed5}. For the dissipation energy functionals, the proof is analogous and we omit it. 
\end{proof}

We are now ready to prove Propositions \ref{lem:linerr}-\ref{lem:momerrLinfty}.
\subsection{Linear errors}
\label{sec:linerr}
To prove Proposition \ref{lem:linerr}, we rely on the pointwise-in-$k$ estimates \eqref{bd:lwDTd} and \eqref{bd:lwDed} that are needed to prove decay for the linearized problem. In fact, the factor $\brak{\lambda_{\nu,k} t}$ in the energy functionals replaces the $k$-by-$k$ decay in time (which should not be possible in the nonlinear problem) and generates the linear error term \eqref{def:Le}. It is therefore natural to expect that the same estimates giving decay in the linearized problem allow us to control this linear error term.  
\begin{proof}[Proof of Proposition \ref{lem:linerr}]
Appealing to \eqref{bd:lwDTd} and \eqref{bd:lwDed}, recalling that $$\varpi=(|\gamma|(2-s)+2s|q_{\gamma,s}|)/(1+s),$$ for any $R>0$ we have 
\begin{align}
\mathcal{D}(t)\gtrsim \, &\int_{|k|\leq \delta_0^{-1}\nu}\bigg(\frac{\lambda_{\nu,k}}{R^{\varpi}}\brak{\lambda_{\nu,k} t}^{2J}E^{T.d.}_{M,B}-\frac{\lambda_{\nu,k}}{R^{|\gamma+2s|+2M_J}}\brak{\lambda_{\nu,k} t}^{2J}E^{T.d.}_{M+M_J,B}\bigg) \dd k \notag \\
&+\int_{|k|> \delta_0^{-1}\nu}\bigg(\frac{\lambda_{\nu,k}}{R^{\varpi}}\brak{\lambda_{\nu,k} t}^{2J}E^{e.d.}_{M,B}-\frac{\lambda_{\nu,k}}{R^{|\gamma+2s|+2M_J}}\brak{\lambda_{\nu,k} t}^{2J}E^{e.d.}_{M+M_J,B} \bigg)\dd k. \label{bd:lwDNL} 
\end{align}
The idea here is to choose $R$ such that we control $L_{J,M,B}$ \eqref{def:Le} with the dissipation and the higher order moments. In particular, we need
\begin{equation}
\label{eq:RNL}
\frac{\lambda_{\nu,k}}{R^{\varpi}}\brak{\lambda_{\nu,k} t}^{2J}\sim\lambda_{\nu,k}\brak{\lambda_{\nu,k} t}^{2J-1} \qquad \Longrightarrow \qquad R^\varpi=\tilde{\delta}\brak{\lambda_{\nu,k} t}
\end{equation}
for some $\tilde{\delta}$ sufficiently small to be specified later. Plugging \eqref{eq:RNL} in \eqref{bd:lwDNL}, we know that there is a constant $C>0$ such that
\begin{align}
&\int_{\mathbb{R}^d}L_{J,M,B}\dd k\leq \tilde{\delta}C\mathcal{D} \notag \\
&\quad+C\tilde{\delta}^{1-\widetilde{M}_J}\int_{|k|\leq \delta_0^{-1}\nu}\lambda_{\nu,k}\brak{\lambda_{\nu,k} t}^{2J-\widetilde{M}_J}\left(\mathbbm{1}_{|k|\leq \delta_0^{-1}\nu}E^{T.d.}_{M+M_J,B}+\mathbbm{1}_{|k|>\delta_0^{-1}\nu}E^{e.d.}_{M+M_J,B}\right)\dd k, \label{bd:LJMB}   
\end{align}
where $\widetilde{M}_J=(|\gamma+2s|+2M_J)/\varpi>2J+2$ thanks to our choice in \eqref{eq:constMJB}.
Then observe that
\begin{align*}
&\int_0^t\int_{\mathbb{R}^d}\lambda_{\nu,k}\brak{\lambda_{\nu,k} \tau}^{2J-\widetilde{M}_J}\left(\mathbbm{1}_{|k|\leq \delta_0^{-1}\nu}E^{T.d.}_{M+M_J,B}+\mathbbm{1}_{|k|>\delta_0^{-1}\nu}E^{e.d.}_{M+M_J,B}\right)(\tau)\dd \tau\dd k\\
&=\int_{\mathbb{R}^d}\int_0^{\lambda_{\nu,k} t}\frac{1}{\brak{s}^{\widetilde{M}_J-2J}}\left(\mathbbm{1}_{|k|\leq \delta_0^{-1}\nu}E^{T.d.}_{M+M_J,B}+\mathbbm{1}_{|k|>\delta_0^{-1}\nu}E^{e.d.}_{M+M_J,B}\right)(\lambda_{\nu,k}^{-1} s)\dd s\dd k\\
&\leq \int_{\mathbb{R}^d}\sup_{0\leq s'\leq t}\left(\mathbbm{1}_{|k|\leq \delta_0^{-1}\nu}E^{T.d.}_{M+M_J,B}+\mathbbm{1}_{|k|>\delta_0^{-1}\nu}E^{e.d.}_{M+M_J,B}\right)(s') \dd k \int_0^{\infty}\frac{1}{\brak{s}^{2}} \dd s\\
&\lesssim \sup_{0\leq s'\leq t}\mathcal{E}_{mom}(s').
\end{align*}
Hence, integrating in time \eqref{bd:LJMB}, using the bound above and choosing $\tilde{\delta}$ sufficiently small we prove  \eqref{bd:LerrD}. The proof of \eqref{bd:LerrLF} is analogous.
\end{proof}
\subsection{Nonlinear errors for the main $L^2_k$ energy}
\label{sec:NLerrED}
We now turn our attention the proof of Proposition \ref{prop:NLerrorL2}. A first crucial ingredient is the trilinear estimate in Lemma \ref{lem:trilBd}. This allow us to control everything with $L^2_v$ and $\cA$ norms, which are used in the definition of the energy and dissipation functionals. Having the trilinear estimate \eqref{bd:NLZ} at hand, the main strategy of proof is to split the nonlinearity to study separately the interactions between high and low $k$-frequencies in each nonlinear term. In particular, we must take care also of \textit{very low} frequencies $|k|\ll \nu$ since the behavior in the enhanced dissipation regime is different with respect to the Taylor dispersion one. The main technical difficulty is to choose this splitting carefully to reconstruct the energy and dissipation functionals and to take advantage of the time-decay when needed.
 
 \begin{proof}[Proof of Proposition \ref{prop:NLerrorL2}] 
 First of all, appealing to the paraproduct decomposition \eqref{def:para0}-\eqref{def:para1}, we split the nonlinear term as 
\begin{equation}
\label{eq:paraNL}
	\Gamma(f,f)=\sum_{N,N'\in 2^{\ZZ}}\bigg(\Gamma(f_N,f_{< N/8})+\Gamma(f_{< N/8},f_N)+\sum_{N/8\leq N'\leq 8N}\Gamma(f_{N},f_{N'})\bigg)
\end{equation}
Due to the structure of the nonlinearity, the treatment of the first two terms inside the parenthesis of \eqref{eq:paraNL} is identical and therefore we will not discuss the term $\Gamma(f_{< N/8}, f_N)$ in the sequel. From \eqref{def:Gamma}, one has
 \begin{equation}
 \label{def:GammaF}
 \widehat{\Gamma(g,h)}(k)=
 \frac{1}{\sqrt{\mu}}
 \int_{\mathbb{R}^d} \mathcal{Q}(\sqrt{\mu}\hat{g}(\xi),\sqrt{\mu}\hat{h}(k-\xi))\dd \xi.
 \end{equation}
The nonlinear error terms $NL^{e.d.}_{M,B}$ \eqref{def:NLed}  and $NL^{T.d.}_{M,B}$ \eqref{def:NLTd} have fundamentally different frequency interactions. We thus divide the proof by controlling first the error terms in $NL^{e.d.}_{M,B}$ and then the ones in $NL^{T.d.}_{M,B}$.

\subsubsection{\textbf{Nonlinear enhanced dissipation errors}}
\label{subsec:NLEDL2}
In this case $|k|\geq \delta_0^{-1}\nu$. Consider the term $\Gamma(f_N,f_{< N/8})$. Using the notation  in \eqref{def:GammaF}, we have 
\begin{equation}
\label{bd:simN}
|k|\leq |k-\xi|+|\xi|\lesssim N, \qquad |k|\geq |\xi|-|k-\xi|\geq N/4 \quad \Longrightarrow \quad N\approx |k|\gtrsim\nu.
\end{equation}
For $\Gamma(f_N,f_{N'})$, we cannot have $|\xi|\leq (\delta_0^{-1}\nu)/64$. Indeed, if the latter inequality were true we would get $|k-\xi|\geq (63\delta_0^{-1}\nu)/64$. But this is not possible since otherwise 
\begin{equation}
N'\geq |k-\xi|/2,\qquad N\geq N'/8, \qquad |\xi|\geq N/2 \quad \Longrightarrow \quad |\xi|\geq (\delta_0^{-1}\nu)(63/2^{11})
\end{equation}
 which is a contradiction because $ 63/2^{11}>1/64$.
Namely, for $\Gamma[f_N,f_{N'}]$, we only have to study the case $$N,N'\gtrsim \nu.$$
Thanks to \eqref{bd:equivalencenu}-\eqref{bd:equivalencenu1},  we can always consider $f_N,f_{N'}$  to be in the enhanced dissipation regime. On the other hand, the factor $f_{< N/8}$ can be in either regimes. For this reason, in the following we only detail how to control the terms arising from $\Gamma(f_N,f_{< N/8})$. The bounds for the terms involving $\Gamma(f_N,f_{N'})$ are easily recovered from the ones for $\Gamma(f_N,f_{< N/8})$ when $f_{< N/8}$ is at frequencies $\gtrsim \nu$.  
 
 	To prove \eqref{bd:NLL2}, we first bound the following term in $NL^{e.d.}_{M,B}$
	\begin{equation}
	\label{def:Iaed}
		 \mathcal{I}^{\alpha,\beta}_{M}:=\mathbbm{1}_{|k|\geq \delta_0^{-1}\nu} \frac{2^{-\mathtt{C}\beta}}{\brak{\nu t}^{2\beta}}\jap{k}^{2\alpha}a_{\nu,k}\sum_{N\sim |k|}\left| \brak{ \brak{v}^{M+q_{\gamma,s}} \grad_v Z^\beta \cF(  \Gamma(f_N,f_{< N/8})), \brak{v}^{M+q_{\gamma,s}} \grad_v Z^\beta \hat{f}}_{L^2_v}\right|. 
	\end{equation}
	 Then, we split the low-frequency part as 
\begin{equation}
	f_{< N/8}=P_{\leq \delta_0^{-1}\nu}f_{< N/8}+P_{> \delta_0^{-1}\nu}f_{< N/8},
\end{equation}
and  define 
\begin{equation}
\label{def:splitIM}
\mathcal{I}^{\alpha,\beta}_{M}=\mathcal{I}^{\alpha,\beta}_{M,\gtrsim \nu}+\mathcal{I}^{\alpha,\beta}_{M,\lesssim \nu},
\end{equation}
where $\mathcal{I}^{\alpha,\beta}_{M,\lesssim \nu}$ and $\mathcal{I}^{\alpha,\beta}_{M,\gtrsim \nu}$ are the right-hand side of \eqref{def:Iaed} with $f_{< N/8}$ replaced by $P_{\leq \delta_0^{-1}\nu}f_{< N/8}$ and $P_{> \delta_0^{-1}\nu}f_{< N/8}$ respectively.

\smallskip \noindent \textbf{High-moderately low frequency interactions.} 
We proceed with the bound for $\mathcal{I}^{\alpha,\beta}_{M,\gtrsim \nu}$.
 From the trilinear estimate in Lemma \ref{lem:trilBd} and the Leibniz rule for $\Gamma$, we infer
\begin{align}
\label{bd:Ia1}	&\mathcal{I}^{\alpha,\beta}_{M,\gtrsim \nu}\lesssim\,\sum_{\substack{N\sim |k|\\ |\beta_1|+|\beta_2|\leq|\beta|\\
	j_1+j_2\leq1}} \mathbbm{1}_{|k|\geq \delta_0^{-1}\nu} a_{\nu,k}\frac{2^{-\mathtt{C}\beta}}{\brak{\nu t}^{2\beta}}\jap{k}^{2\alpha}\sqrt{\cA[\brak{v}^{M+q_{\gamma,s}} \grad_v Z^\beta \hat{f}(k)]}\\
\label{bd:Ia2}	\times &\int_{\mathbb{R}^d} \bigg(\sqrt{\cA[\brak{v}^{M+q_{\gamma,s}} (\grad_v)^{j_1} Z^{\beta_1} \hat{f}_N(\xi)]}\norm{\brak{v}^{M+q_{\gamma,s}} (\grad_v)^{j_2} Z^{\beta_2} P_{> \delta_0^{-1}\nu}\hat{f}_{< N/8}(k-\xi)}_{L^2_v}\\
\label{bd:Ia3}& \quad+\norm{\brak{v}^{M+q_{\gamma,s}} (\grad_v)^{j_1} Z^{\beta_1} \hat{f}_N(\xi)}_{L^2_v}\sqrt{\cA[\brak{v}^{M+q_{\gamma,s}} (\grad_v)^{j_2} Z^{\beta_2}P_{> \delta_0^{-1}\nu} \hat{f}_{< N/8}(k-\xi)]}\bigg)\dd \xi. 
\end{align}
Thanks to \eqref{bd:equivalencenu}, we can always control the terms containing $P_{> \delta_0^{-1}\nu} \hat{f}_{< N/8}$ with the functionals in the enhanced dissipation regime.  In particular, for the terms involving $\cA$ we can use the dissipation functional $\mathcal{D}^{e.d.}_{M,B}$ \eqref{def:Ded} and for the $L^2_v$ terms the energy $E^{e.d.}_{M,B}$ \eqref{def:Eed}. However, when we take $v$-derivatives we need to compare the coefficients $a_{\nu,k}$ at different frequencies as in \eqref{bd:trivcoeff}.
To proceed with the bounds, in the following we denote
\begin{equation}
\label{def:notg}
g^{(\beta,M,j)}=\frac{2^{-\mathtt{C}\beta/2}}{\brak{\nu t}^{\beta}}\brak{v}^{M+q_{\gamma,s}} (\grad_v)^{j} Z^{\beta} \hat{f}.
\end{equation}
Since $|\beta_1|+|\beta_2|\leq|\beta|$ and $|\xi|\approx N\approx |k|$, see \eqref{bd:simN}, we have
\begin{equation}
\frac{2^{-\mathtt{C}\beta}}{\brak{\nu t}^{2\beta}}\leq \frac{2^{-\mathtt{C}\beta/2}}{\brak{\nu t}^{\beta}}\frac{2^{-\mathtt{C}(\beta_1+\beta_2)/2}}{\brak{\nu t}^{(\beta_1+\beta_2)}}, \qquad \brak{k}\approx \brak{\xi}.
\end{equation}
Hence, with the notation \eqref{def:notg}, using \eqref{bd:trivcoeff} and the above inequalities in \eqref{bd:Ia1}, we get
\begin{align}
\label{bd:Igtrnu} 
&\mathcal{I}^{\alpha,\beta}_{M,\gtrsim\nu}\lesssim\frac{1}{\nu} \sum_{\substack{N\sim |k|\\ |\beta_1|+|\beta_2|\leq|\beta|\\ j_1+j_2\leq 1}}\mathbbm{1}_{|k|\geq \delta_0^{-1}\nu} \jap{k}^{\alpha}\sqrt{\nu a_{\nu,k}\cA[g^{(\beta,M,1)}(k)]}\\
\label{bd:Igtrnu1} 
\times&\int_{\mathbb{R}^d} \bigg(\jap{\xi}^{\alpha}\sqrt{\nu( a_{\nu,\xi})^{j_1}\cA[g^{(\beta_1,M,j_1)}_N(\xi)]}\sqrt{(a_{\nu,k-\xi})^{j_2}}\normL{P_{>\delta_0^{-1}\nu}g^{(\beta_2,M,j_2)}_{< N/8}(k-\xi)}_{L^2_v}\\
\label{bd:Igtrnu3} 
&  +\jap{\xi}^{\alpha}(\sqrt{a_{\nu,\xi}})^{j_1}\normL{g^{(\beta_1,M,j_1)}_N(\xi)}_{L^2_v}\sqrt{\nu (a_{\nu,k-\xi})^{j_2}\cA[P_{>\delta_0^{-1}\nu}g^{(\beta_2,M,j_2)}_{< N/8}(k-\xi)]}\bigg)\dd \xi.
\end{align}
We now claim that 
\begin{equation}
\label{bd:goal1}
\nu \int_{|k|>\delta_0^{-1}\nu} \sum_{\alpha+|\beta|\leq B}\jap{\lambda_{\nu,k} t}^{2J} \mathcal{I}^{\alpha,\beta}_{M,\gtrsim \nu} \dd k\lesssim \sqrt{\mathcal{E}}\mathcal{D}
\end{equation}
To prove this bound, we are going to use Cauchy-Schwartz and Young's convolution inequalities. However, we need to be careful with the number of $x$ and $Z$ derivatives we have in each term, since the sum of the two must always be less than $B$. Therefore we distinguish two cases. 

\smallskip \noindent $\diamond$ \textbf{Case} $|\beta_2|+d/2< B.$ Recall the definitions of $E^{e.d.}_{M.B}$ and $\mathcal{D}^{e.d.}_{M.B}$ given in \eqref{def:Eed} and \eqref{def:ETd} respectively. Notice that, since $\alpha+|\beta_1|\leq B$, $j_1\leq 1$ and $N\gtrsim \nu$, using also \eqref{bd:equivalencenu}-\eqref{bd:equivalencenu1}, we have 
 \begin{align}
\label{bd:L2DN}&\norm{\brak{k}^\alpha \sqrt{\nu (a_{\nu,k})^{j_1}\cA[g^{(\beta_1,M,j_1)}_{ N}]}}_{L^2_k}\lesssim \norm{\sqrt{\mathcal{D}^{e.d.}_{M,B}(k)_{N}}}_{L^2_k}\\
\label{bd:L2EN}&\norm{\brak{k}^\alpha(\sqrt{a_{\nu,k}})^{j_1}g^{(\beta_1,M,j_1)}_{N}}_{L^2_k}\lesssim \norm{\sqrt{E^{e.d.}_{M,B}(k)_{ N}}}_{L^2_k}
\end{align}
 Hence, from the quasi-orthogonality of the Littlewood-Paley projection, the Cauchy-Schwartz and the Young's convolution inequalities, combining \eqref{bd:Igtrnu} with \eqref{bd:L2DN}-
 \eqref{bd:L2EN} we get
 \begin{align}
 \nu &\int_{|k|>\delta_0^{-1}\nu} \sum_{\substack{\alpha+|\beta|\leq B\\ |\beta_2|+(d+1)/2<B}}\jap{\lambda_{\nu,k} t}^{2J} \mathcal{I}^{\alpha,\beta}_{M,\gtrsim \nu} \dd k\lesssim \sum_{\substack{N\gtrsim \nu\\ |\beta_2|+(d+1)/2<B, \, j_2\leq 1}}\norm{\brak{\lambda_{\nu,k} t}^{J}\sqrt{\mathcal{D}^{e.d.}_{M,B}(k)_{\sim N}}}_{L^2_k}\\
 \label{bd:Igtrnucase1}
&\times \bigg( \norm{\brak{\lambda_{\nu,k} t}^{J}\sqrt{\mathcal{D}^{e.d.}_{M,B}(k)_{N}}}_{L^2_k}\norm{\sqrt{(a_{\nu,k})^{j_2}}P_{>\delta_0^{-1}\nu}g^{(\beta_2,M,j_2)}_{< N/8}}_{L^1_k L^2_v}\\
\label{bd:Igtrnucase12} & \quad +\norm{\brak{\lambda_{\nu,k} t}^{J}\sqrt{E^{e.d.}_{M,B}(k)_{ N}}}_{L^2_k}\norm{\sqrt{\nu (a_{\nu,k})^{j_2}\cA[P_{>\delta_0^{-1}\nu}g^{(\beta_2,M,j_2)}_{< N/8}(k)]}}_{L^1_k}\bigg).
  \end{align}
  From H\"older's inequality we deduce
\begin{equation}
\label{bd:0case1}
\norm{\sqrt{(a_{\nu,k})^{j_2}}P_{>\delta_0^{-1}\nu}g^{(\beta_2,M,j_2)}_{< N/8}}_{L^1_k L^2_v}\lesssim_{\delta_*} \norm{\brak{k}^{d/2+\delta_*}\sqrt{a_{\nu,k}}^{j_2}P_{>\delta_0^{-1}\nu}g^{(\beta_2,M,j_2)}_{< N/8}}_{L^2_k L^2_v}
\end{equation}
where $\delta_*>0$ is an arbitrary small number. In the case under consideration, it is enough to choose $\delta_*$ such that $|\beta_2|+d/2+\delta_*\leq B$ to obtain
\begin{equation}
\label{bd:1case1}
\norm{\brak{k}^{d/2+\delta_*}\sqrt{(a_{\nu,k})^{j_2}}P_{>\delta_0^{-1}\nu}g^{(\beta_2,M,j_2)}_{< N/8}}_{L^2_k L^2_v}\lesssim \norm{\sqrt{P_{>\delta_0^{-1}\nu}E^{e.d.}_{M,B}(k)_{< N/8}}}_{L^2_k}.
\end{equation}
Analogously, we get
 \begin{equation}
\label{bd:0case10}
\norm{\sqrt{\nu(a_{\nu,k})^{j_2}\cA[P_{>\delta_0^{-1}\nu}g^{(\beta_2,M,j_2)}_{< N/8}]}}_{L^1_k}\lesssim \norm{\sqrt{P_{>\delta_0^{-1}\nu}\mathcal{D}^{e.d.}_{M,B}(k)_{< N/8}}}_{L^2_k}
\end{equation}
Therefore, combining \eqref{bd:Igtrnucase1} with \eqref{bd:1case1} and \eqref{bd:0case10}, using the quasi orthogonality of the Littlewood-Paley projections, we infer
 \begin{align}
\label{bd:sqrtED1}
 \nu\int_{|k|\geq \delta_0^{-1}\nu}\sum_{\substack{\alpha+|\beta|\leq B\\ |\beta_2|+(d+1)/2< B}}\brak{\lambda_{\nu,k} t}^{2J}\mathcal{I}^{\alpha,\beta}_{M,\gtrsim \nu}\dd k
\lesssim \sqrt{\mathcal{E}}\mathcal{D}.
\end{align}

 \smallskip \noindent $\diamond$  \textbf{Case} $|\beta_2|+d/2\geq B:$ In this case we cannot pay $x$-derivatives on the low frequency factors. 
 Then, since $\alpha+|\beta|\leq B$ and $|\beta_1|+|\beta_2|\leq|\beta|$, observe that 
 \begin{equation}
 \label{bd:coeffcase2}
\alpha+d/2+|\beta_1|\leq B+d/2-|\beta_2|\leq d< B.
\end{equation}
Hence, when applying the Young's convolution inequality to \eqref{bd:Igtrnu} we can use the $L^1_k$ norm in the factors at frequencies $N$ in \eqref{bd:Igtrnu1}-\eqref{bd:Igtrnu3}. Namely, using \eqref{bd:L2DN}-\eqref{bd:L2EN} with $$(\alpha,j_1,\beta_1,N)\to(0,j_2,\beta_2,<N/8),$$
similarly to \eqref{bd:Igtrnucase1}, one has
 \begin{align}
 \label{bd:Igtrnucase2}
 \nu &\int_{|k|>\delta_0^{-1}\nu} \sum_{\substack{\alpha+|\beta|\leq B\\ |\beta_2|+(d+1)/2\geq B}}\jap{\lambda_{\nu,k} t}^{2J} \mathcal{I}^{\alpha,\beta}_{M,\gtrsim \nu} \dd k\lesssim \sum_{\substack{N\gtrsim \nu, \, |\beta_1|+|\beta_2|\leq B\\ |\beta_2|+(d+1)/2\geq B, \, j_1\leq 1}}\norm{\brak{\lambda_{\nu,k} t}^{J}\sqrt{\mathcal{D}^{e.d.}_{M,B}(k)_{\sim N}}}_{L^2_k}\\
  & \times \bigg(\norm{\jap{\lambda_{\nu,k} t}^J\jap{k}^{\alpha}\sqrt{\nu (a_{\nu,k})^{j_1}\cA[g^{(\beta_1,M,j_1)}_{N}(k)]}}_{L^1_k}\norm{\sqrt{E^{e.d.}_{M,B}(k)_{< N/8}}}_{L^2_k}\\
& \quad +\norm{\jap{\lambda_{\nu,k} t}^J\jap{k}^{\alpha}\sqrt{(a_{\nu,k})^{j_1}}g^{(\beta_1,M,j_1)}_{N}}_{L^1_k L^2_v}\norm{\sqrt{\mathcal{D}^{e.d.}_{M,B}(k)_{< N/8}}}_{L^2_k}\bigg).
  \end{align}
Thus, thanks to \eqref{bd:coeffcase2}, we can argue as done in \eqref{bd:0case1}-\eqref{bd:0case10}  for the factors at frequencies $N$. So we conclude that
 \begin{align}
\label{bd:sqrtED2}
 \nu\int_{|k|\geq \delta_0^{-1}\nu}\sum_{\substack{\alpha+|\beta|\leq B\\ |\beta_2|+(d+1)/2\geq B}}\brak{\lambda_{\nu,k} t}^{2J}\mathcal{I}^{\alpha,\beta}_{M,\gtrsim \nu}\dd k\lesssim \sqrt{\mathcal{E}}\mathcal{D},
\end{align}
which proves the claim \eqref{bd:goal1} by combining \eqref{bd:sqrtED1} and \eqref{bd:sqrtED2}.
\begin{remark}
\label{rem:TvsR}
Observe that to get  \eqref{bd:goal1} we never used the time-decay to control the nonlinearity. The main reason for this is that all factors are in the enhanced dissipation regime. Therefore, if instead of $\RR^d$ we consider $\TT^d$, the nonlinear error term under consideration is bounded by  \eqref{bd:goal1}. Since the other error terms enjoy analogous estimates, the proof for the bound \eqref{bd:NLL2} would be almost over (one has to control the interactions between the $k=0$ mode and $k\neq 0$ separately). On the other hand, to control our energy functionals in $\RR^d$, we cannot use the available dissipation for $k$ too small and the time-decay will play a crucial role.\footnote{A similar issue was already present for the energy functional used by Strain \cite{strain2012optimal}. However, due to the structure of the energy functional in \cite{strain2012optimal}, it is not necessary to exploit the decay in time and estimates of the form $\sqrt{\mathcal{E}}\mathcal{D}$ are enough to close the argument. On the other hand, exploiting the time-decay in the nonlinear problem seems crucial to prove the Taylor dispersion, the enhanced dissipation and the optimal time-decay rates obtained by controlling the $L^\infty_k$ norm of the solution.}
\end{remark}
 
\smallskip \noindent \textbf{High-very low frequency interactions.}
We now turn our attention to $\mathcal{I}^{\alpha,\beta}_{a,\lesssim \nu}$, which we recall is defined as the right-hand side of \eqref{def:Iaed} with $f_{< N/8}$ replaced by $P_{\leq \delta_0^{-1}\nu}f_{< N/8}$.  In this case the low-frequencies $<N/8$ are at very low frequencies, namely in the Taylor dispersion regime, in which we do not  have $k$-dependent coefficients in front of $v$-derivatives in $E^{T.d.}_{M,B}$ \eqref{def:ETd}. Therefore, exploiting that $a_{\nu,k}\lesssim 1$ for $|k|\gtrsim \nu$, we bound $\mathcal{I}^{\alpha,\beta}_{a,\lesssim \nu}$ as the right-hand side of \eqref{bd:Igtrnu} with $P_{>\delta_0^{-1}\nu}\to P_{\leq\delta_0^{-1}\nu}$ and $\brak{k-\xi}\sqrt{a_{\nu,k-\xi}}\to 1 $. More precisely, recalling the notation for $g^*$ introduced in \eqref{def:notg}, we have
\begin{align}
\label{bd:Ilesnu} 
\mathcal{I}^{\alpha,\beta}_{M,\lesssim\nu}\lesssim\,& \frac{1}{\nu}\sum_{\substack{N\sim |k|\\ |\beta_1|+|\beta_2|\leq|\beta|\\ j_1+j_2\leq 1}}\mathbbm{1}_{|k|\geq \delta_0^{-1}\nu} \jap{k}^{\alpha}\sqrt{\nu a_{\nu,k}\cA[g^{(\beta,M,1)}(k)]}\\
\label{bd:Ilesnu1} 
&\times\int_{\mathbb{R}^d} \bigg(\jap{\xi}^{\alpha}\sqrt{\nu( a_{\nu,\xi})^{j_1}\cA[g^{(\beta_1,M,j_1)}_N(\xi)]}\normL{P_{\leq\delta_0^{-1}\nu}g^{(\beta_2,M,j_2)}_{< N/8}(k-\xi)}_{L^2_v}\\
\label{bd:Ilesnu3} 
&  \qquad +\jap{\xi}^{\alpha}\sqrt{(a_{\nu,\xi})^{j_1}}\normL{g^{(\beta_1,M,j_1)}_N(\xi)}_{L^2_v}\sqrt{\nu \cA[P_{\leq\delta_0^{-1}\nu}g^{(\beta_2,M,j_2)}_{< N/8}(k-\xi)]}\bigg)\dd \xi.
\end{align}
We claim that 
\begin{equation}
\label{bd:goal2}
\nu \int_{|k|>\delta_0^{-1}\nu} \sum_{\alpha+|\beta|\leq B}\jap{\lambda_{\nu,k} t}^{2J} \mathcal{I}^{\alpha,\beta}_{M,\lesssim \nu} \dd k\lesssim \sqrt{\mathcal{E}}\sqrt{\mathcal{D}}\left(\sqrt{\mathcal{D}}+\big(\nu^d\mathbbm{1}_{t\leq \nu^{-1}}+\brak{\nu/t}^{\frac{d}{2}}\mathbbm{1}_{t> \nu^{-1}}\big)\sqrt{\mathcal{E}_{LF}}\right). 
\end{equation}
Observe that the first piece on the right-hand side of \eqref{bd:goal2} is exactly as in \eqref{bd:goal1}. On the other hand, we will see that to handle the interactions between two very different frequency regimes we need to exploit all the available information about the solution.

To prove \eqref{bd:goal2}, we need to use the micro-macro decomposition, which is a key ingredient at very low frequencies. In light of \eqref{def:notg}, it is convenient to denote\footnote{This extra notation is introduced because $[Z,\bP]\neq 0$.}
\begin{equation}
\label{def:micmacg}
g^{(\beta,M,j)}_{micro}=\frac{2^{-\mathtt{C}\beta/2}}{\brak{\nu t}^{\beta}}\brak{v}^{M+q_{\gamma,s}} (\grad_v)^{j} Z^{\beta}(I-\bP) \hat{f}.
\end{equation}
To control the term involving only the $L^2_v$-norm at very low-frequencies in \eqref{bd:Ilesnu1}, we exploit the properties of the Maxwellian to get
\begin{equation}\label{bd:micmacNLed} 
\begin{aligned}
\normL{ P_{\leq\delta_0^{-1}\nu} g^{(\beta_2,M,j_2)}_{< N/8}(k-\xi)}_{L^2_v}\lesssim\, &  \sum_{0\leq|\tilde{\beta}|\leq|\beta_2|}\frac{2^{-\mathtt{C}\tilde{\beta}/2}}{\brak{\nu t}^{\tilde{\beta}}}P_{\leq \delta_0^{-1}\nu}(| t(k-\xi)|^{\tilde{\beta}}|(\hat{\rho},\hat{\sfm},\hat{\sfe})(k-\xi)_{< N/8}|)\\
&+\norm{ P_{\leq\delta_0^{-1}\nu} g^{(\beta_2,M,j_2)}_{micro}(k-\xi)_{< N/8}}_{L^2_v}.
\end{aligned}
\end{equation}
Since $E^{T.d}_{M,B}$ \eqref{def:ETd}, $\mathcal{D}^{T.d}_{M,B}$ \eqref{def:DTD} and $E^{e.d.}_{M,B}$ \eqref{def:Eed}, $\mathcal{D}^{e.d.}_{M,B}$ \eqref{def:Ded} control the same quantities for the microscopic part (as observed after the definition of $E^{T.d}_{M,B}$), we can proceed as done to obtain \eqref{bd:goal1} for all the terms with $g^*_{micro}$. We therefore do not detail such bounds. On the other hand, when macroscopic quantities are involved we must proceed in a different way. 

First of all, when $|\tilde{\beta}|=0$ we can readily reconstruct the energy functional in view of the term $\norm{ f}_{L^2_v}$ in $E^{T.d}_{M,B}$. 
 For $|\tilde{\beta}|\geq 1$, since $|k-\xi|\lesssim \nu$,  one has 
\begin{equation}
\label{bd:stupidtrick}
|t (k-\xi)|^{\tilde{\beta}}\lesssim \nu t \sqrt{\frac{|k-\xi|}{\nu}}|t (k-\xi)|^{\tilde{\beta}-1}.
\end{equation}
Thus, in view of Lemma \ref{lem:equivZP}, we deduce that
\begin{align}
\label{bd:NLmacro1}
&\frac{2^{-\mathtt{C}\tilde{\beta}/2}}{\brak{\nu t}^{\tilde{\beta}}}P_{\leq \delta_0^{-1}\nu}(| t(k-\xi)|^{\tilde{\beta}}|(\hat{\rho},\hat{\sfm},\hat{\sfe})(k-\xi)_{< N/8}|)\lesssim\normL{P_{\leq \delta_0^{-1}\nu}\hat{f}_{< N/8}(k-\xi)}_{L^2_v}\\
\label{bd:NLmacro2}
&\qquad+\nu t \frac{2^{-\mathtt{C}\tilde{\beta}/2}}{\brak{\nu t}^{\tilde{\beta}}}\sum_{1\leq|\tilde{\beta}'|\leq |\tilde{\beta}|-1} P_{\leq \delta_0^{-1}\nu}\bigg(\sqrt{\frac{|k-\xi|}{\nu}}\normL{Z^{\tilde{\beta}'}\hat{f}_{< N/8}(k-\xi)}_{L^2_v}\bigg).
\end{align}
For the first term on the right-hand side of \eqref{bd:NLmacro1} we can again exploit the term $\norm{f}_{L^2_v}$ in $E^{T.d}_{M,B}$. For the remaining ones, 
observe that, since $|\tilde{\beta}'|\leq |\tilde{\beta}|-1$ we have
\begin{equation}
\nu t\frac{2^{-\mathtt{C}\tilde{\beta}/2}}{\jap{\nu t}^{\tilde{\beta}}} \leq\frac{2^{-\mathtt{C}\tilde{\beta}'/2}}{\jap{\nu t}^{\tilde{\beta}'}},
\end{equation}
so we are able to absorb the factor $\nu t$
in \eqref{bd:NLmacro2}. Recalling that in the definition of $E^{T.d.}_{M,b}$ \eqref{def:ETd} we have the term $\sqrt{\frac{|k|}{\nu}}\norm{Z^\beta f}_{L^2_v}$, we finally obtain
\begin{equation}
\sum_{0\leq|\tilde{\beta}|\leq|\beta_2|}\frac{2^{-\mathtt{C}\tilde{\beta}/2}}{\brak{\nu t}^{\tilde{\beta}}}P_{\leq \delta_0^{-1}\nu}(| t(k-\xi)|^{\tilde{\beta}}|(\hat{\rho},\hat{\sfm},\hat{\sfe})(k-\xi)_{< N/8}|) \lesssim \sqrt{P_{\leq \delta_0^{-1}\nu}E^{T.d.}_{M,B}(k-\xi)_{< N/8}}.
\end{equation}
Notice that here we do not need to pay attention to the number of $x$-derivatives since these terms are compactly supported in the Fourier space. Therefore, the terms arising from \eqref{bd:Ilesnu1} are controlled by $\sqrt{\mathcal{E}}\mathcal{D}$ when proving \eqref{bd:goal2}.

Consider now the terms in \eqref{bd:Ilesnu3}, where very low frequencies are computed in the $\cA$-norm. By the micro-macro decomposition and standard properties of the Maxwellian, we get
 \begin{align}
\label{bd:AmicmacNLed} \sqrt{\nu\cA[ P_{\leq\delta_0^{-1}\nu} g^{(\beta_2,M,j_2)}_{< N/8}(k-\xi)]}&\lesssim \sqrt{\nu}\sum_{0\leq|\tilde{\beta}|\leq|\beta_2|}\frac{2^{-\mathtt{C}\tilde{\beta}/2}}{\brak{\nu t}^{\tilde{\beta}}}P_{\leq \delta_0^{-1}\nu}(| t(k-\xi)|^{\tilde{\beta}}|(\hat{\rho},\hat{\sfm},\hat{\sfe})(k-\xi)_{< N/8}|)\\
\label{bd:AmicmacNLed1}&\quad+\sqrt{\nu\cA\left[ P_{\leq\delta_0^{-1}\nu} g^{(\beta_2,M,j_2)}_{micro}(k-\xi)_{< N/8}\right]}.
\end{align}
As observed previously, since $E^{e.d.}_{M,B}$ and $E^{T.d.}_{M,B}$ have the same structure for microscopic quantities, we can proceed as done to obtain \eqref{bd:goal1} to control the term in \eqref{bd:AmicmacNLed1} (with the simplification that one can always pay derivatives at very low frequencies). On the other hand, for the macroscopic part we need to be more careful due to the degeneracy of the dissipation for very low $k$'s.  In fact, we first need to improve the bound \eqref{bd:NLmacro2} to
\begin{align}
\label{bd:disslow0}&\frac{2^{-\mathtt{C}\tilde{\beta}/2}}{\brak{\nu t}^{\tilde{\beta}}}P_{\leq \delta_0^{-1}\nu}(| t(k-\xi)|^{\tilde{\beta}}|(\hat{\rho},\hat{\sfm},\hat{\sfe})(k-\xi)_{< N/8}|)\lesssim \normL{P_{\leq \delta_0^{-1}\nu}\hat{f}_{< N/8}(k-\xi)}_{L^2_v}\\
\label{bd:disslow}&\qquad +\frac{1}{\sqrt{\nu}} \sum_{1\leq|\tilde{\beta}'|\leq |\tilde{\beta}|-1} \frac{2^{-\mathtt{C}\tilde{\beta}'/2}}{\brak{\nu t}^{\tilde{\beta}'}}\frac{|k-\xi|}{\sqrt{\nu}}\normL{P_{\leq \delta_0^{-1}\nu}Z^{\tilde{\beta}'}\hat{f}_{< N/8}(k-\xi)}_{L^2_v},
\end{align}
where we used 
\begin{equation}
\label{bd:stupidtrick2}
|t (k-\xi)|^{\tilde{\beta}'}= \frac{1}{\sqrt{\nu}}(\nu t) \frac{|k-\xi|}{\sqrt{\nu}}|t (k-\xi)|^{\tilde{\beta}'-1}
\end{equation}
and the fact that $1\leq|\tilde{\beta}'|\leq |\tilde{\beta}|-1$ to absorb the $\nu t$ loss. 
Notice that we can control a term like $\frac{|k|}{\sqrt{\nu}}\norm{Z^\beta f}$ with $\mathcal{D}^{T.d}_{M,B}$ \eqref{def:DTD}. In particular, from \eqref{bd:disslow0} we infer
\begin{align}
\notag \sqrt{\nu}\sum_{0\leq|\tilde{\beta}|\leq|\beta_2|}&\frac{2^{-\mathtt{C}\tilde{\beta}/2}}{\brak{\nu t}^{\tilde{\beta}}}P_{\leq \delta_0^{-1}\nu}(| t(k-\xi)|^{\tilde{\beta}}|(\hat{\rho},\hat{\sfm},\hat{\sfe})(k-\xi)_{< N/8}|)\\
\label{bd:disslow1}
&\lesssim \sqrt{\nu}\normL{P_{\leq \delta_0^{-1}\nu}\hat{f}_{< N/8}(k-\xi)}_{L^2_v}+\sqrt{P_{\leq \delta_0^{-1}\nu}\mathcal{D}^{T.d.}_{M,B}(k-\xi)_{< N/8}}.
\end{align}
Then, the last term on the right-hand side of \eqref{bd:disslow1} gives rise to terms that are bounded by $\sqrt{\mathcal{E}}\mathcal{D}$. However, for the other term we cannot exploit the dissipation functional since we would need a factor $|k-\xi|$ that can be extremely small in this regime. Thus, with the strategy outlined in \eqref{bd:micmacNLed}-\eqref{bd:disslow1} we are able to prove 
\begin{align}
\label{bd:sqrtED+int}
\nu \int_{|k|<\delta_0^{-1}\nu} \brak{\lambda_{\nu,k} t}^{2J}\mathcal{I}^{\alpha,\beta}_{M,\lesssim \nu}\dd k&\lesssim 
\sqrt{\mathcal{E}}\mathcal{D}+\sqrt{\mathcal{E}}\sqrt{\mathcal{D}} \normL{P_{\leq \delta_0^{-1}\nu}\hat{f}}_{L^1_kL^2_v}.
\end{align}
Then, since $\lambda_{\nu,k}\sim |k|^2/\nu$ for $|k|\leq\delta_0^{-1}\nu$, notice that
\begin{align}
\label{bd:integration}
\normL{P_{\leq \delta_0^{-1}\nu}\hat{f}}_{L^1_kL^2_v}&\lesssim \normL{P_{\leq \delta_0^{-1}\nu}\brak{\lambda_{\nu,k} t}^{J'} \hat{f}}_{L^\infty_kL^2_v}\int_{|k|\lesssim \nu} \frac{\dd \xi}{\brak{ \nu^{-1}|\xi|^2 t}^{J'}}\\
&\lesssim \left(\nu^d\mathbbm{1}_{t\leq \nu^{-1}}+\brak{\nu/t}^{\frac{d}{2}}\mathbbm{1}_{t> \nu^{-1}}\right)\sqrt{\mathcal{E}_{LF}}.
\end{align}
Inserting this bound in \eqref{bd:sqrtED+int} we finally prove \eqref{bd:goal2}. Combining \eqref{bd:goal1} and \eqref{bd:goal2} we finish the proof of the bound \eqref{bd:NLL2}  for the term $\mathcal{I}^{\alpha,\beta}_{M}$  \eqref{def:Iaed} in $NL^{e.d.}_{M,B}$ \eqref{def:NLed}.

\begin{remark}
\label{remL2dgeq3}
For $d\geq 3$ it is not necessary to control the energy functional $\mathcal{E}_{LF}$ in order to bound the nonlinear terms coming from the  $L^2_k$ energy functional, i.e. $\mathcal{E}$. Indeed, we can use the $L^2_k$ norm instead of the $L^\infty_k$ norm in \eqref{bd:integration} to get  
\begin{align}
\label{bd:integrationL2}
\normL{P_{\leq \delta_0^{-1}\nu}\hat{f}}_{L^1_kL^2_v}\lesssim \left(\nu^{\frac{d}{2}}\mathbbm{1}_{t\leq \nu^{-1}}+\brak{\nu/t}^{\frac{d}{4}}\mathbbm{1}_{t> \nu^{-1}}\right)\sqrt{\mathcal{E}}.
\end{align}
However, for $d=2$ this term would generate a logarithmic loss when doing estimates as in \eqref{bd:pfmain1}. Besides getting sharp decay estimate, this is the other main reason why we directly control $\mathcal{E}$ in terms of the $\mathcal{E}_{LF}$ even for $d\geq3$.
\end{remark}

Regarding the other terms in $NL^{e.d.}_{M,B}$, notice that the first term on the right-hand side of \eqref{def:NLed} is easier to control with respect to \eqref{def:Iaed} since no $v$ or $x$ derivatives are involved.
To handle the terms with $b_{\nu,k}$ in \eqref{def:NLed}, recalling the definition of $b_{\nu,k}$ \eqref{def:abnuk}, it is not hard to check that the factor $|k|^{-1}$ balances the $x$-derivative whereas the remaining $(\nu/|k|)^{\frac{1}{1+2s}}$ is exactly $\sqrt{a_{\nu,k}}$, which is the coefficient we need to reconstruct the energy and dissipation in the terms involving $v$-derivatives. Indeed, for these terms at most one factor has a $v$-derivative as opposed to $\mathcal{I}^{\alpha,\beta}_M$ where two factors can have a $v$-derivative. Therefore, taking into account of the scalings of the coefficients, the estimates for these terms are also analogous to one for $\mathcal{I}^{\alpha,\beta}_M$. This finishes the proof of the bounds \eqref{bd:NLL2} and \eqref{bd:NLsup} for the nonlinearities coming from the enhanced dissipation energy functional.

\subsubsection{\textbf{Nonlinear Taylor dispersion errors}}
\label{subsec:NLTDL2}
We now consider the nonlinear error terms related to the Taylor dispersion energy functional, see \eqref{def:NLTd}. In this regime we have $|k|\leq \delta_0^{-1}\nu$. We use again the paraproduct decomposition \eqref{eq:paraNL}. When studying the terms containing $\Gamma(f_N,f_{< N/8})$, thanks to \eqref{bd:simN} we know that both $f_N$ and $f_{< N/8}$ are in the Taylor dispersion regime (namely $|\xi|, |k-\xi|\lesssim \delta_0^{-1}\nu$). Instead, for $\Gamma(f_N,f_{N'})$ we can have 
\begin{equation}
|\xi|\leq2\delta_0^{-1}\nu \quad \Longrightarrow \quad|k-\xi|\leq 3\delta_0^{-1}\nu \qquad \text{or} \qquad |\xi|>2\delta_0^{-1}\nu \quad \Longrightarrow \quad|k-\xi|> 3\delta_0^{-1}\nu.
\end{equation}
 This means that $f_N$ and $f_{N'}$ are either in the Taylor dispersion regime or in the enhanced dissipation one, namely 
\begin{equation}
N' \sim N \gtrsim \nu, \qquad \text{or} \qquad N' \sim  N \lesssim \nu.
\end{equation}
 Consequently, we will only detail the treatment of the error terms arising from $\Gamma(f_N,f_{N'})$ since the bounds when $N\lesssim \nu$ easily adapt to $\Gamma(f_N, f_{< N/8})$. 
 
 Let us consider first the term from \eqref{def:NLTd2}, that is the analogue of \eqref{def:Iaed}, namely
	\begin{equation}
	\label{def:IaTd}
		 \mathcal{J}^{\alpha,\beta}_{M}:=\sum_{N\sim N'}\mathbbm{1}_{|k|\leq \delta_0^{-1}\nu} \frac{2^{-\mathtt{C}_1\beta}}{\brak{\nu t}^{2\beta}}\left| \brak{ \brak{v}^{M+q_{\gamma,s}} \grad_v Z^\beta \cF(  \Gamma(f_N,f_{N'})), \brak{v}^{M+q_{\gamma,s}} \grad_v Z^\beta (I-\bP)\hat{f}}_{L^2_v}\right|.
	\end{equation}
	We define 
	\begin{equation}
		\label{def:splitJM}
	\mathcal{J}^{\alpha,\beta}_{M}=\mathcal{J}^{\alpha,\beta}_{M,\lesssim \nu}+\mathcal{J}^{\alpha,\beta}_{M,\gtrsim \nu},
	\end{equation} where
	$\mathcal{J}^{\alpha,\beta}_{M,\lesssim \nu}$ and $\mathcal{J}^{\alpha,\beta}_{M,\gtrsim \nu}$ are \eqref{def:IaTd} with $N\sim N'\lesssim \nu$ and $N\sim N'\gtrsim\nu$ respectively.

\smallskip \noindent \textbf{Very low-very low frequency interactions}.	
When $N\lesssim \nu$, thanks to \eqref{bd:equivalencenu} we can consider $f_N, f_{N'}$ and $(I-\bP)f$ in the Taylor dispersion regime.  We can then apply Lemma \ref{lem:trilBd} to obtain estimates as in \eqref{bd:Ia1}. Having that $$\nu \cA[\brak{v}^{M+q_{\gamma,s}} \grad_v Z^\beta (I-\bP)\hat{f}]$$ is a term included in $\cD^{T.d}_{M,B}$ \eqref{def:DTD}, we can proceed as done to obtain the bound \eqref{bd:goal2} in the previous subsection to get 
\begin{equation}
\label{bd:sqrtTD+int}
\nu\int_{|k|<\delta_0^{-1}\nu} \sum_{\alpha+|\beta|\leq B}\brak{\lambda_{\nu,k} t}^{2J}\mathcal{J}^{\alpha,\beta}_{M,\lesssim \nu}\dd k\lesssim 
\sqrt{\mathcal{E}}\sqrt{\mathcal{D}}\bigg(\sqrt{\mathcal{D}}+\left(\nu^d\mathbbm{1}_{t\leq \nu^{-1}}+\brak{\nu/t}^{\frac{d}{2}}\mathbbm{1}_{t> \nu^{-1}}\right) \sqrt{\mathcal{E}_{LF}}\bigg).
\end{equation}
 \noindent \textbf{Very low-high frequency interactions}.	
On the other hand, for $N\gtrsim \nu$ we have to be more careful since now we need to reconstruct the $a_{\nu,k}$ coeffiecients in  $E^{e.d.}_{M,B}$ and $\mathcal{D}^{e.d.}_{M,B}$ which could potentially lead to losses of order $\nu^{-1/(1+2s)}$. This is because $v$-derivatives in the Taylor dispersion regime are not weighted by factors $\nu^{-1/(1+2s)}$ as in the enhanced dissipation one. To overcome this difficulty,  using the same notation introduced in \eqref{def:notg}, appealing to Lemma \ref{lem:trilBd} we obtain 
\begin{equation}
\begin{aligned}
\label{bd:IgtrnuTD} 
&\mathcal{J}^{\alpha,\beta}_{M,\gtrsim\nu}\lesssim \frac{1}{\nu}\sum_{\substack{N\sim N'\gtrsim \nu\\ |\beta_1|+|\beta_2|\leq|\beta|, \, j_1+j_2\leq 1}}\mathbbm{1}_{|k|\leq \delta_0^{-1}\nu} \jap{k}^{\alpha}\frac{2^{-\mathtt{C}_1\beta/2}}{\brak{\nu t}^\beta}\sqrt{\nu \cA[\brak{v}^{M+q_{\gamma,s}}\nabla_vZ^\beta (I-\bP)\hat{f}(k)]}\\
&\times\int_{\mathbb{R}^d} \bigg(\left(\frac{\jap{\xi}}{\nu}\right)^{\frac{j_1}{1+2s}}\sqrt{\nu(a_{\nu,\xi})^{j_1} \cA[g^{(\beta_1,M,j_1)}_N(\xi)]}\left(\frac{\jap{k-\xi}}{\nu}\right)^{\frac{j_2}{1+2s}} \sqrt{(a_{\nu,k-\xi})^{j_2}}\normL{g^{(\beta_2,M,j_2)}_{N'}(k-\xi)}_{L^2_v}\\
&  +\left(\frac{\jap{\xi}}{\nu}\right)^{\frac{j_1}{1+2s}}\sqrt{(a_{\nu,\xi})^{j_1}}\normL{g^{(\beta_1,M,j_1)}_N(\xi)}_{L^2_v}\left(\frac{\jap{k-\xi}}{\nu}\right)^{\frac{j_2}{1+2s}}\sqrt{\nu (a_{\nu,k-\xi})^{j_2}\cA[g^{(\beta_2,M,1)}_{N'}(k-\xi)]}\bigg)\dd \xi.
\end{aligned}
\end{equation}
To absorb the loss in $\nu$ when estimating the $L^2_k$ norm, we use the fact that we are integrating on frequencies $|k|\lesssim \nu$ to gain powers of $\nu$ from the integration. Namely, we eploit the inequality 
\begin{align}
\label{bd:easyTD}
\int_{|k|\lesssim \nu} |{G}(k)({H}*Q)(k)|\dd k\lesssim \norm{H*Q}_{L^\infty_k}\norm{G}_{L^2_k}\nu^{\frac{d}{2}}\lesssim \norm{H}_{L^2_k}\norm{Q}_{L^2_k}\norm{G}_{L^2_k}\nu^{\frac{d}{2}}.
\end{align}
Using that $N\sim N'$, we can distribute the $x$-derivative in \eqref{bd:IgtrnuTD} on the term with the lowest number of $Z$-derivatives (smaller $|\beta_i|$). 
Hence, combining the inequality \eqref{bd:exlambda} with \eqref{bd:IgtrnuTD} and using \eqref{bd:easyTD}, we get
 \begin{align}
\label{bd:sqrtTD0}
& \nu \int_{|k|\leq \delta_0^{-1}\nu}\sum_{\alpha+|\beta|\leq B}\brak{\lambda_{\nu,k} t}^{2J}\mathcal{J}^{\alpha,\beta}_{M,\gtrsim \nu}\dd k\lesssim \nu^{\frac{d}{2}-\frac{1}{1+2s}}\norm{P_{\leq \delta_0^{-1}\nu} \brak{\lambda_{\nu,k} t}^{J}\sqrt{\cD^{T.d.}_{M,B}}}_{L^2_k}\\
&\times  \bigg(  \norm{P_{> \delta_0^{-1}\nu} \brak{\lambda_{\nu,k} t}^{J}\sqrt{\cD^{e.d.}_{M,B}}}_{L^2_k} \norm{P_{> \delta_0^{-1}\nu} \sqrt{E^{e.d.}_{M,B}}}_{L^2_k} \\
&+\norm{P_{> \delta_0^{-1}\nu} \brak{\lambda_{\nu,k} t}^{J}\sqrt{E^{e.d.}_{M,B}}}_{L^2_k}\norm{P_{> \delta_0^{-1}\nu} \sqrt{\cD^{e.d.}_{M,B}}}_{L^2_k} \bigg)\\
\label{bd:sqrtTD1}&\lesssim \sqrt{\mathcal{E}}\mathcal{D},
\end{align}
where we also used that $d/2-1/(1+2s)\geq 0$ since $d\geq2$ and $s\in (0,1]$.
Combining \eqref{bd:sqrtTD+int} and \eqref{bd:sqrtTD1} we see that from the term $\mathcal{J}^{\alpha,\beta}_M$ we get bounds in agreement with \eqref{bd:NLL2}. Following the same arguments outlined for the proof of the bounds for $\mathcal{J}^{\alpha,\beta}_M$, it is not difficult to control in an analogous way the terms in \eqref{def:NLTd}-\eqref{def:NLTd4} containing at least one $(I-\bP)f$.

\smallskip
\noindent  \textbf{Macroscopic error terms.}
We are thus left with the macroscopic error terms given by
\begin{align}
\label{def:errmacro}
& \mathcal{J}^{\alpha,\beta}_{macro, N\sim N'}+\mathbbm{1}_{|k|\leq \delta_0^{-1}\nu}\frac{c_0}{\nu}\mathcal{M}_{\Gamma, N\sim N'},\\
&\mathcal{J}^{\alpha,\beta}_{macro,N\sim N'}:=\sum_{N\sim N'}\mathbbm{1}_{|k|\leq \delta_0^{-1}\nu}\frac{2^{-\mathtt{C}_0\beta}}{\brak{\nu t}^{2\beta}} \jap{k}^{2\alpha}\bigg|\frac{|k|}{\nu}\brak{ Z^\beta\cF(  \Gamma(f_{N},f_{N'})), Z^\beta \bP\hat{f}}_{L^2_v}\bigg|
\end{align}
where $\mathcal{M}_{\Gamma,N\sim N'}$ is defined as in \eqref{bd:mixedNL} with $\Gamma(f,f)$ replaced by $\sum_{N\sim N'}\Gamma(f_{N},f_{N'})$. We recall that the proof of the bounds for the terms with $\Gamma(f_{N},f_{< N/8})$ can be treated as a subcase of the one we present below, see also the discussion at beginning of the proof for the Taylor dispersion errors. 

We claim that 
\begin{align}
\label{bd:sqrtTDmacro}
\nu\int_{|k|\leq \delta_0^{-1}\nu} \sum_{\alpha+|\beta|\leq B}\brak{\lambda_{\nu,k} t}^{2J}\big(\mathcal{J}^{\alpha,\beta}_{macro,N\sim N'}+\frac{c_0}{\nu}\mathcal{M}_{\Gamma, N\sim N'}\big)\dd k&\lesssim 
\sqrt{\mathcal{E}}\mathcal{D}.
\end{align}
Indeed, using the trilinear estimate in Lemma \ref{lem:trilBd} we get
\begin{align}
 &\mathcal{J}^{\alpha,\beta}_{macro, N\sim N'}\lesssim \frac{1}{\nu} \sum_{\substack{N\sim N' \\ |\beta_1|+|\beta_2|\leq|\beta|}}\mathbbm{1}_{|k|\leq \delta_0^{-1}\nu}  \frac{2^{-\mathtt{C}_0\beta/2}}{\brak{\nu t}^{\beta}}\brak{k}^{\alpha}\frac{|k|}{\sqrt{\nu}}\normL{Z^\beta \bP \hat{f}(k)}_{L^2_v}\\
 \times&\int_{\mathbb{R}^d}\bigg(\frac{2^{-\mathtt{C}_0\beta_1/2}}{\brak{\nu t}^{\beta_1}}\sqrt{\nu \cA[Z^{\beta_1}\hat{f}_N(\xi)]}\frac{2^{-\mathtt{C}_0\beta_2/2}}{\brak{\nu t}^{\beta_2}}\norm{Z^{\beta_2}\hat{f}_N(k-\xi)}_{L^2_v} \\
 &\qquad \qquad \qquad \qquad +\frac{2^{-\mathtt{C}_0\beta_1/2}}{\brak{\nu t}^{\beta_1}}\norm{Z^{\beta_1}\hat{f}_N(\xi)}_{L^2_v}\frac{2^{-\mathtt{C}_0\beta_2/2}}{\brak{\nu t}^{\beta_2}}\sqrt{\nu \cA[Z^{\beta_2}\hat{f}_N(k-\xi)]}\bigg)\dd \xi
\end{align}
where we used that $\brak{k}^{\alpha}\lesssim 1$ for $|k|\lesssim \nu$. From the bound above, exploiting the inequality \eqref{bd:easyTD} we get
 \begin{align}
\label{bd:sqrtTD2}
 \nu \int_{|k|\leq \delta_0^{-1}\nu}\sum_{\alpha+|\beta|\leq B}\brak{\lambda_{\nu,k} t}^{2J}\mathcal{J}^{\alpha,\beta}_{macro,N\sim N'}\dd k&\lesssim \nu^{\frac{d}{2}}\norm{P_{\leq \delta_0^{-1}\nu} \brak{\lambda_{\nu,k} t}^{J}\sqrt{\cD^{T.d.}_{M,B}}}_{L^2_k}\sqrt{\mathcal{E}}\sqrt{\mathcal{D}}\\
 &\lesssim \sqrt{\mathcal{E}}\mathcal{D}.
\end{align}
Notice that here we never lose factors $\nu^{-1/(1+2s)}$ because we do not have $v$-derivatives. 
For $\mathcal{M}_{\Gamma, N\sim N'}$, in view of Lemma \eqref{lem:equivZP}, it is not hard to see that 
\begin{align}
&\mathcal{M}_{\Gamma, N\sim N'}\lesssim \frac{1}{\nu}\sum_{\substack{N\sim N' \\ |\beta_1|+|\beta_2|\leq|\beta|}}\mathbbm{1}_{|k|\leq \delta_0^{-1}\nu}  \sum_{\widetilde{\beta}\leq \beta}\frac{2^{-\mathtt{C}_1\widetilde{\beta}/2}}{\brak{\nu t}^{\widetilde{\beta}}}\brak{k}^{\alpha}\frac{|k|}{\sqrt{\nu}}\normL{Z^\beta \bP \hat{f}(k)}_{L^2_v}\\
 \times&\int_{\mathbb{R}^d}\bigg(\frac{2^{-\mathtt{C}_1\beta_1}}{\brak{\nu t}^{\beta_1}}\sqrt{\nu \cA[Z^{\beta_1}\hat{f}_N(\xi)]}\frac{2^{-\mathtt{C}_1\beta_2}}{\brak{\nu t}^{\beta_2}}\norm{Z^{\beta_2}\hat{f}_N(k-\xi)}_{L^2_v} \\
 &\qquad \qquad \qquad \qquad +\frac{2^{-\mathtt{C}_1\beta_1}}{\brak{\nu t}^{\beta_1}}\norm{Z^{\beta_1}\hat{f}_N(\xi)}_{L^2_v}\frac{2^{-\mathtt{C}_1\beta_2}}{\brak{\nu t}^{\beta_2}}\sqrt{\nu \cA[Z^{\beta_2}\hat{f}_N(k-\xi)]}\bigg)\dd \xi.
\end{align}
Therefore, also for this term we deduce a bound in agreement with \eqref{bd:sqrtTDmacro}, whence concluding the proof of Proposition \ref{prop:NLerrorL2}.
\end{proof}

\subsection{Nonlinear errors for the main $L^\infty_k$ energy}
\label{sec:NLerrLinfty}
In this section we aim at proving Proposition \ref{prop:NLerrorLinfty}. The strategy of proof is the same followed to prove Proposition \ref{prop:NLerrorL2} and many steps will actually be identical. The main differences appear when using Young's convolution inequality in Fourier space. We, therefore, provide the details only for the steps that differs from the proof of Proposition \ref{prop:NLerrorL2}. 
\begin{proof}
We split the nonlinear term as in \eqref{eq:paraNL} and we study separately the nonlinear error terms $NL^{e.d.}_{M,B}$ \eqref{def:NLed}  and $NL^{T.d.}_{M,B}$ \eqref{def:NLTd}.
\subsubsection{\textbf{Nonlinear enhanced dissipation errors}}
In light of the discussion before the inequality \eqref{def:Iaed}, we study the term
	\begin{equation}
	\label{def:Iaedsup}
		 \mathcal{I}^{\alpha,\beta}_{M'}:=\mathbbm{1}_{|k|\geq \delta_0^{-1}\nu} \frac{2^{-\mathtt{C}\beta}}{\brak{\nu t}^{2\beta}}\jap{k}^{2\alpha}a_{\nu,k}\sum_{N\sim |k|}\left| \brak{ \brak{v}^{M'+q_{\gamma,s}} \grad_v Z^\beta \cF(  \Gamma(f_N,f_{< N/8})), \brak{v}^{M'+q_{\gamma,s}} \grad_v Z^\beta \hat{f}}_{L^2_v}\right|. 
	\end{equation}
	As in \eqref{def:splitIM}, we split this term as $\mathcal{I}^{\alpha,\beta}_{M'}=\mathcal{I}^{\alpha,\beta}_{M',\gtrsim \nu}+\mathcal{I}^{\alpha,\beta}_{M',\lesssim \nu},$ where $\mathcal{I}^{\alpha,\beta}_{M',\lesssim \nu}$ and $\mathcal{I}^{\alpha,\beta}_{M',\gtrsim \nu}$ are the right-hand side of \eqref{def:Iaedsup} with $f_{< N/8}$ replaced by $P_{\leq \delta_0^{-1}\nu}f_{< N/8}$ and $P_{> \delta_0^{-1}\nu}f_{< N/8}$ respectively.
	
\smallskip \noindent \textbf{High-moderately low frequency interactions.}
We can repeat all the arguments done in \eqref{bd:Ia1}-\eqref{bd:Igtrnu3} by replacing $M$ with $M'$, since none of these bounds depend on the specific choice of the velocity weight. Hence, recalling the definition \eqref{def:notg} 
\begin{equation}
\label{def:notg0}
g^{(\beta,M,j)}=\frac{2^{-\mathtt{C}\beta/2}}{\brak{\nu t}^{\beta}}\brak{v}^{M+q_{\gamma,s}} (\grad_v)^{j} Z^{\beta} \hat{f},
\end{equation}
we have
\begin{align}
\label{bd:Igtrnusup} 
&\mathcal{I}^{\alpha,\beta}_{M',\gtrsim\nu}\lesssim \frac{1}{\nu}\sum_{\substack{N\sim |k|\\ |\beta_1|+|\beta_2|\leq|\beta|\\ j_1+j_2\leq 1}}\mathbbm{1}_{|k|\geq \delta_0^{-1}\nu} \jap{k}^{\alpha}\sqrt{\nu a_{\nu,k}\cA[g^{(\beta,M',1)}(k)]}\\
\label{bd:Igtrnu1sup} 
\times&\int_{\mathbb{R}^d} \bigg(\jap{\xi}^{\alpha}\sqrt{\nu( a_{\nu,\xi})^{j_1}\cA[g^{(\beta_1,M',j_1)}_N(\xi)]}\sqrt{(a_{\nu,k-\xi})^{j_2}}\normL{P_{>\delta_0^{-1}\nu}g^{(\beta_2,M',j_2)}_{< N/8}(k-\xi)}_{L^2_v}\\
\label{bd:Igtrnu3sup} 
&  +\jap{\xi}^{\alpha}(\sqrt{a_{\nu,\xi}})^{j_1}\normL{g^{(\beta_1,M',j_1)}_N(\xi)}_{L^2_v}\sqrt{\nu (a_{\nu,k-\xi})^{j_2}\cA[P_{>\delta_0^{-1}\nu}g^{(\beta_2,M',j_2)}_{< N/8}(k-\xi)]}\bigg)\dd \xi.
\end{align}
We then claim that 
\begin{equation}
\label{bd:goal1sup}
\nu \sup_{k\,:\,|k|>\delta_0^{-1}\nu} \sum_{\alpha+|\beta|\leq B'}\jap{\lambda_{\nu,k} t}^{2J'} \mathcal{I}^{\alpha,\beta}_{M',\gtrsim \nu} \dd k\lesssim \sqrt{\mathcal{E}}\sqrt{\mathcal{D}}\sqrt{\mathcal{D}_{LF}}.
\end{equation}
To prove this bound, when applying Young's convolution inequalities to the terms in \eqref{bd:Igtrnu1sup}-\eqref{bd:Igtrnu3sup} we can use the $L^2_k$ norm in both factors. This simplifies the proof compared to the one for Proposition \ref{prop:NLerrorL2} since we do not need to distinguish cases depending on the number of $Z$-derivatives. Indeed,
having that 
\begin{equation}
J'\leq J, \quad M'\leq M, \quad \alpha+|\beta_1|\leq B'\leq B, \quad |\beta_2|+1/2\leq B'+1\leq B,
\end{equation}
 since $N\gtrsim \nu$, we get 
 \begin{align}
\label{bd:LinftyDNsup}&\norm{\brak{\lambda_{\nu,k} t}^{J'}\brak{k}^\alpha \sqrt{\nu (a_{\nu,k})^{j_1}\cA[g^{(\beta_1,M',j_1)}_{ N}]}}_{L^2_k}\lesssim \norm{\brak{\lambda_{\nu,k} t}^{J}\sqrt{\mathcal{D}^{e.d.}_{M,B}(k)_{N}}}_{L^2_k}\\
\label{bd:LinftyENlowsup}&\norm{\sqrt{(a_{\nu,k})^{j_2}}P_{>\delta_0^{-1}\nu}g^{(\beta_2,M',j_2)}_{< N/8}}_{L^2_kL^2_v}\lesssim \norm{\sqrt{P_{>\delta_0^{-1}\nu}E^{e.d.}_{M,B}(k)_{< N/8}}}_{L^2_k}\\
\label{bd:LinftyENsup}&\norm{\brak{\lambda_{\nu,k} t}^{J'}\brak{k}^\alpha(\sqrt{a_{\nu,k}})^{j_1}g^{(\beta_1,M',j_1)}_{N}}_{L^2_kL^2_v}\lesssim \norm{\brak{\lambda_{\nu,k} t}^{J}\sqrt{E^{e.d.}_{M,B}(k)_{ N}}}_{L^2_k}\\
\label{bd:LinftyDNlowsup}&\norm{\sqrt{\nu (a_{\nu,k})^{j_2}\cA[P_{>\delta_0^{-1}\nu}g^{(\beta_2,M',j_2)}_{< N/8}]}}_{L^2_k}\lesssim \norm{\sqrt{P_{>\delta_0^{-1}\nu}\mathcal{D}^{e.d.}_{M,B}(k)_{< N/8}}}_{L^2_k}.
\end{align}
Using the Young's convolution inequality in \eqref{bd:Igtrnusup} and appealing to \eqref{bd:LinftyDNsup}-\eqref{bd:LinftyDNlowsup}, we have
 \begin{align}
&  \nu\sup_{k : \, |k|\geq \delta_0^{-1}\nu} \sum_{\alpha+|\beta|\leq B'}(\brak{\lambda_{\nu,k} t}^{2J'}\mathcal{I}^{\alpha,\beta}_{M',\gtrsim \nu})\lesssim \sum_{N\gtrsim \nu} \norm{ \brak{\lambda_{\nu,k} t}^{J'}\sqrt{\cD^{e.d.}_{M',B'}(k)_{\sim N}}}_{L^\infty_k}\\
&\times \bigg(  \norm{ \brak{\lambda_{\nu,k} t}^{J}\sqrt{\cD^{e.d.}_{M,B}(k)_{N}}}_{L^2_k} \norm{ \sqrt{P_{> \delta_0^{-1}\nu}E^{e.d.}_{M,B}(k)_{< N/8}}}_{L^2_k}\\
&\quad +\norm{ \brak{\lambda_{\nu,k} t}^{J}\sqrt{E^{e.d.}_{M,B}(k)_{N}}}_{L^2_k}\norm{ \sqrt{P_{> \delta_0^{-1}\nu}\cD^{e.d.}_{M,B}(k)_{< N/8}}}_{L^2_k} \bigg)\\
\label{bd:sqrtELF0}&\lesssim \sqrt{\mathcal{E}}\sqrt{\mathcal{D}}\sqrt{\mathcal{D}_{LF}},
\end{align}
where we also used the quasi-orthogonality of the Littlewood-Payley projections. The claim \eqref{bd:goal1sup} is proved.

\smallskip \noindent \textbf{High-very low frequency interactions.}
We now turn our attention to $\mathcal{I}^{\alpha,\beta}_{M',\lesssim \nu}$. We can repeat the computations in \eqref{bd:Ilesnu}-\eqref{bd:stupidtrick2} with $M$ replaced by $M'$. Using the Young's convolution inequality 
\begin{equation}
\norm{g*h}_{L^\infty_k}\leq\norm{g}_{L^\infty_k}\norm{h}_{L^1_k},
\end{equation}
instead of \eqref{bd:sqrtED+int} we arrive at 
\begin{equation}
\label{bd:sqrtELF+int}
 \sup_{k: \, |k|\leq \delta_0^{-1}\nu}\sum_{\alpha+|\beta|\leq B'}\nu\brak{\lambda_{\nu,k} t}^{2J'}\mathcal{I}^{\alpha,\beta}_{M',\lesssim \nu}\lesssim 
\sqrt{\mathcal{E}}\sqrt{\mathcal{D}}\sqrt{\mathcal{D}_{LF}}+\sqrt{\mathcal{E}_{LF}}\sqrt{\mathcal{D}_{LF}} \normL{P_{\leq \delta_0^{-1}\nu}\hat{f}}_{L^1_kL^2_v}.
\end{equation}
Also in this case we can argue as in \eqref{bd:integration} to get
\begin{align}
\label{bd:integrationsup}
\normL{P_{\leq \delta_0^{-1}\nu}\hat{f}}_{L^1_kL^2_v}\lesssim \left(\nu^d\mathbbm{1}_{t\leq \nu^{-1}}+\brak{\nu/t}^{\frac{d}{2}}\mathbbm{1}_{t> \nu^{-1}}\right)\sqrt{\mathcal{E}_{LF}}.
\end{align}
Plugging the bound above in \eqref{bd:sqrtELF+int} and using \eqref{bd:goal1sup} we finally prove the bound \eqref{bd:NLsup}  for the term $\mathcal{I}^{\alpha,\beta}_{M'}$  \eqref{def:Iaedsup} in $NL^{e.d.}_{M',B'}$ \eqref{def:NLed}. The other terms in $NL^{e.d.}_{M',B'}$ can be treated analogously, see the discussion at the end of Section \ref{subsec:NLEDL2}.

\subsubsection{\textbf{Nonlinear Taylor dispersion errors}}
Following the discussion in Section \ref{subsec:NLTDL2}, we study the term 
	\begin{equation}
	\label{def:IaTdsup}
		 \mathcal{J}^{\alpha,\beta}_{M'}:=\sum_{N\sim N'}\mathbbm{1}_{|k|\leq \delta_0^{-1}\nu} \frac{2^{-\mathtt{C}_1\beta}}{\brak{\nu t}^{2\beta}}\left| \brak{ \brak{v}^{M'+q_{\gamma,s}} \grad_v Z^\beta \cF(  \Gamma(f_N,f_{N'})), \brak{v}^{M'+q_{\gamma,s}} \grad_v Z^\beta (I-\bP)\hat{f}}_{L^2_v}\right|.
	\end{equation}
We split it as	 $\mathcal{J}^{\alpha,\beta}_{M}=\mathcal{J}^{\alpha,\beta}_{M,\lesssim}+\mathcal{J}^{\alpha,\beta}_{M,\gtrsim \nu}$ where
	$\mathcal{J}^{\alpha,\beta}_{M,\lesssim \nu}$ and $\mathcal{J}^{\alpha,\beta}_{M,\gtrsim \nu}$ are \eqref{def:IaTdsup} with $N\sim N'\leq \delta_0^{-1}\nu$ and $N\sim N'> \delta_0^{-1}\nu$ respectively.
	
	\smallskip \noindent \textbf{Very low-very low frequency interactions}. When $N\lesssim \nu$, we can proceed as done in the previous section (compare with the observation made before \eqref{bd:sqrtTD+int}). Thus, we get 
	\begin{align}
	\label{bd:JVLVL}
	\nu\sum_{\alpha+|\beta|\leq B'}&\sup_{k: \, |k|\leq \delta_0^{-1}\nu}\brak{\lambda_{\nu,k} t}^{2J'}\mathcal{J}^{\alpha,\beta}_{M',\lesssim \nu}\\
	&\lesssim \sqrt{\mathcal{D}_{LF}}\left(\sqrt{\mathcal{E}}\sqrt{\mathcal{D}}+\big(\nu^d\mathbbm{1}_{t\leq \nu^{-1}}+\brak{\nu/t}^{\frac{d}{2}}\mathbbm{1}_{t> \nu^{-1}}\big)\mathcal{E}_{LF}\right).
	\end{align}

 \noindent \textbf{Very low- moderately high frequency interactions}.	
For $N\gtrsim \nu$, we have the main difference with respect to the $L^2_k$ case. Since we are considering the $L^\infty_k$-norm, here we cannot gain a smallness parameter as in \eqref{bd:easyTD}. 
Hence, for the $L^\infty_k$ norm, we have to crucially use the fact that $B'<B$. Indeed, since 
\begin{equation}
\nabla_vZ^{\beta}=Z^{\beta+1}-t\nabla_x Z^\beta,
\end{equation}
by \eqref{def:IaTdsup} we can bound $\mathcal{J}^{\alpha,\beta}_{M',\gtrsim \nu}$ as
	\begin{align}
	\notag&\mathbbm{1}_{|k|\leq \delta_0^{-1}\nu} \frac{2^{-\mathtt{C}_1\beta}}{\brak{\nu t}^{2\beta}}\jap{k}^{2\alpha}\sum_{N\sim N'\gtrsim \nu}\bigg(\left| \brak{ \brak{v}^{M'+q_{\gamma,s}}  Z^{\beta+1} \cF(  \Gamma(f_N,f_{N'})), \brak{v}^{M'+q_{\gamma,s}} \grad_v Z^\beta (I-\bP)\hat{f}}_{L^2_v}\right|\\
		\label{def:IaTdsup1} & +\brak{\nu t}\left| \brak{ \brak{v}^{M'+q_{\gamma,s}} Z^\beta \cF(  \Gamma(f_N,f_{N'})), \brak{v}^{M'+q_{\gamma,s}} \grad_v Z^\beta (I-\bP)\hat{f}}_{L^2_v}\right|\bigg)
	\end{align}
	where in the last line we used that $t|k|\lesssim \brak{\nu t}$ in this regime. Thus, as opposed to \eqref{bd:IgtrnuTD}, we never have to pay negative powers of $\nu$ since we never take a $v$-derivative of the terms $f_{N}, f_{N'}$ (which are in the enhanced dissipation regime). However, we need to use one $Z$-derivative more but we can do so since $B\geq B'+1$ and we aim at controlling terms with $f_{N}, f_{N'}$ with $L^2_k$ norms. 
	\begin{remark}
	This is a term where we are exploiting the phase mixing. Namely, we are controlling a $v$-derivative in terms of a $Z$-derivative.
	\end{remark}
	Notice that in \eqref{def:IaTdsup1} we have to lose a factor $\brak{\nu t}$ also in the term containing $Z^{\beta+1}$ and not only in \eqref{def:IaTdsup}. This is because we have to divide $Z^{\beta+1}$ by $\brak{\nu t}^{\beta+1}$. More precisely, instead of 
\eqref{bd:IgtrnuTD} we get
\begin{equation}
\begin{aligned}
\label{bd:IgtrnuTDsup} 
&\mathcal{J}^{\alpha,\beta}_{M',\gtrsim\nu}\lesssim \frac{1}{\nu} \brak{\nu t}\sum_{\substack{N\sim N'\gtrsim \nu\\ |\beta_1|+|\beta_2|\leq|\beta|+1}}\mathbbm{1}_{|k|\leq \delta_0^{-1}\nu}\jap{k}^{\alpha}\frac{2^{-\mathtt{C}_1\beta}}{\brak{\nu t}^\beta}\sqrt{\nu \cA[\brak{v}^{M'+q_{\gamma,s}}\nabla_vZ^\beta (I-\bP)\hat{f}(k)]}\\
&\qquad\qquad\qquad\times\int_{\mathbb{R}^d} \bigg(\sqrt{\nu \cA[g^{(\beta_1,M',0)}_N(\xi)]}\normL{g^{(\beta_2,M',0)}_{N'}(k-\xi)}_{L^2_v}\\
&\qquad \qquad+\normL{g^{(\beta_1,M',0)}_N(\xi)}_{L^2_v}\sqrt{\nu \cA[g^{(\beta_2,M',0)}_{N'}(k-\xi)]}\bigg)\dd \xi.
\end{aligned}
\end{equation}
We rely on \eqref{bd:exlambdaJ} to absorb the $
\brak{\nu t}$ loss in \eqref{bd:IgtrnuTDsup}. In particular, we have 
 \begin{align}
\label{bd:sqrtTD0sup}
& \nu \sum_{\alpha+|\beta|\leq B'}\sup_{k:\, |k|\leq \delta_0^{-1}\nu}\brak{\lambda_{\nu,k} t}^{2J'}\mathcal{J}^{\alpha,\beta}_{M',\gtrsim \nu}\lesssim \norm{P_{\leq \delta_0^{-1}\nu} \brak{\lambda_{\nu,k} t}^{J'}\sqrt{\cD^{T.d.}_{M',B'}}}_{L^\infty_k}\\
&\times  \bigg(  \norm{P_{> \delta_0^{-1}\nu} \brak{\lambda_{\nu,k} t}^{J}\sqrt{\cD^{e.d.}_{M,B}}}_{L^2_k} \norm{P_{> \delta_0^{-1}\nu} \sqrt{E^{e.d.}_{M,B}}}_{L^2_k} \\
&+\norm{P_{> \delta_0^{-1}\nu} \brak{\lambda_{\nu,k} t}^{J}\sqrt{E^{e.d.}_{M,B}}}_{L^2_k}\norm{P_{> \delta_0^{-1}\nu} \sqrt{\cD^{e.d.}_{M,B}}}_{L^2_k} \bigg)\\
\label{bd:sqrtTD15}&\lesssim \sqrt{\mathcal{E}}\sqrt{\mathcal{D}}\sqrt{\mathcal{D}_{LF}}.
\end{align}
Combining together \eqref{bd:JVLVL} and \eqref{bd:sqrtTD15}, we see that for the term $\mathcal{J}^{\alpha,\beta}_{M'}$ we get bounds in agreement with \eqref{bd:NLsup}. The other terms in \eqref{def:NLTd}-\eqref{def:NLTd4} that contain at least one $(I-\bP)\hat{f}$ are analogous. 

The macroscopic error terms defined in \eqref{def:errmacro} do not depend on the velocity weights. Thus, we can combine the arguments done to obtain \eqref{bd:sqrtTDmacro} with the ones to get \eqref{bd:sqrtTD15} to prove
\begin{align}
\label{bd:sqrtTDmacroLF}
\nu\sup_{k:|k|\leq \delta_0^{-1}\nu} \sum_{\alpha+|\beta|\leq B}\brak{\lambda_{\nu,k} t}^{2J}\big(\mathcal{J}^{\alpha,\beta}_{macro,N\sim N'}+\frac{c_0}{\nu}\mathcal{M}_{\Gamma, N\sim N'}\big)\dd k&\lesssim 
\sqrt{\mathcal{E}}\sqrt{\mathcal{D}}\sqrt{\mathcal{D}_{LF}},
\end{align}
whence concluding the proof of Proposition \ref{prop:NLerrorLinfty}.
\end{proof}

\subsection{Nonlinear errors for the higher order moments}
\label{sec:NLmom}
The control for the higher moments in the linearized problem is a straightforward consequence of the estimates \eqref{bd:energyED} and \eqref{bd:energyTD}. For the nonlinear problem, it is also not difficult to adapt the estimates leading to the bounds \eqref{bd:NLL2} and \eqref{bd:NLsup} to obtain \eqref{bd:NLL2mom} and \eqref{bd:NLsupmom} respectively. Since we are using higher order velocity weights, we only have to check where we need the functionals $\mathcal{E}, \mathcal{E}_{LF}$ in \eqref{bd:NLL2mom} and \eqref{bd:NLsupmom}.

\begin{proof}[Proof of Proposition \ref{lem:momerrL2}]
To prove the $L^2_k$ bound \eqref{bd:NLL2mom}, we can argue as done to prove \eqref{bd:NLL2} simply by replacing $(M,J)$ with $(M+M_J,0)$ in all the estimates. For instance, in the enhanced dissipation regime we have to bound 
\begin{equation}
\nu \int_{|k|\geq \delta_0^{-1}\nu}\sum_{\alpha+|\beta|\leq B} \mathcal{I}^{\alpha,\beta}_{M+M_J} \dd k
\end{equation}
where $\mathcal{I}^{\alpha,\beta}_{M+M_J}$ is defined in \eqref{def:Iaed}. With the same splitting introduced in \eqref{def:splitIM}, we first prove \eqref{bd:Igtrnu} with $M\to M+M_J$ and then, as in \eqref{bd:sqrtED1} with $J=0$, we get
\begin{equation}
\nu \int_{|k|\geq \delta_0^{-1}\nu}\sum_{\alpha+|\beta|\leq B} \mathcal{I}^{\alpha,\beta}_{M+M_J,\gtrsim \nu} \dd k\lesssim \sqrt{\mathcal{E}_{mom}}\mathcal{D}_{mom}.
\end{equation}
For the term $\mathcal{I}^{\alpha,\beta}_{M+M_J,\lesssim \nu}$, we also have \eqref{bd:Ilesnu} with $M\to M+M_J$. For the bound \eqref{bd:micmacNLed}, notice that for the piece with $\bP f$ the weights are irrelevant thanks to the decay of the Maxwellian. These macroscopic terms are the one that needs to be controlled with $\mathcal{E}_{LF}$. We can argue as done to obtain \eqref{bd:sqrtED+int} and \eqref{bd:integration} to prove that
\begin{align}
\nu \int_{|k|\leq \delta_0^{-1}\nu}&\sum_{\alpha+|\beta|\leq B} \mathcal{I}^{\alpha,\beta}_{M+M_J,\lesssim \nu} \dd k\\
&\lesssim \sqrt{\mathcal{E}_{mom}}\sqrt{\mathcal{D}_{mom}}\left(\sqrt{\mathcal{D}_{mom}}+\big(\nu^d\mathbbm{1}_{t\leq \nu^{-1}}+\brak{\nu/t}^{\frac{d}{2}}\mathbbm{1}_{t> \nu^{-1}}\big)\sqrt{\mathcal{E}_{LF}}\right).
\end{align}
In the Taylor dispersion regime, consider the term $\mathcal{J}^{\alpha,\beta}_{M+M_J}$ as in \eqref{def:IaTd} with its splitting \eqref{def:splitJM}. Following the reasoning before \eqref{bd:sqrtTD+int} and the estimates \eqref{bd:sqrtTD1} with $(M,J)\to (M+M_J,0)$, we get 
\begin{align}
&\sum_{\alpha+|\beta|\leq B}\int_{\mathbb{R}^d}\nu\mathcal{J}^{\alpha,\beta}_{M+M_J,\lesssim \nu}\dd k\\
&\lesssim 
\sqrt{\mathcal{E}_{mom}}\sqrt{\mathcal{D}_{mom}}\left(\sqrt{\mathcal{D}_{mom}}+\big(\nu^d\mathbbm{1}_{t\leq \nu^{-1}}+\brak{\nu/t}^{\frac{d}{2}}\mathbbm{1}_{t> \nu^{-1}}\big)\sqrt{\mathcal{E}_{LF}}\right).
\end{align}
Also here, the $\mathcal{E}_{LF}$ factor comes from macroscopic error terms, therefore insensitive to the weights.
\end{proof}

\begin{proof}[Proof of Proposition \ref{lem:momerrLinfty}]
For the proof of the $L^\infty_k$ bound \eqref{bd:NLsupmom}, we can readily follow the proof of \eqref{bd:NLsup} since the only properties we used are $M'\leq M-1, J' \leq J-1$ and $B'\leq B-1$ which are still true if we replace $M'\to M'+M_{J'}$ and $J'\to 0$ thanks to \eqref{eq:constMJB}. For instance, following \eqref{def:Iaedsup}-\eqref{bd:integrationsup} and the arguments to get \eqref{bd:JVLVL}  we obtain 
 \begin{align}
 \nu&\sup_{k : \, |k|\geq \delta_0^{-1}\nu} \sum_{\alpha+|\beta|\leq B'}\mathcal{I}^{\alpha,\beta}_{M'+M_{J'}}+\mathcal{J}^{\alpha+\beta}_{M'+M_{J'},\lesssim \nu}\\
 &\lesssim \sqrt{\mathcal{D}_{mom,LF}}\left(\sqrt{\mathcal{E}_{mom}}\sqrt{\mathcal{D}_{mom}}+\big(\nu^d\mathbbm{1}_{t\leq \nu^{-1}}+\brak{\nu/t}^{\frac{d}{2}}\mathbbm{1}_{t> \nu^{-1}}\big)\sqrt{\mathcal{E}_{LF}}\sqrt{\mathcal{E}_{mom,LF}}\right).
\end{align}
On the other hand, for $\mathcal{J}^{\alpha+\beta}_{M'+M_{J'},\gtrsim \nu}$ we get the bounds in \eqref{bd:IgtrnuTDsup} with $M'\to M'+M_{J'}$. To absorb the $\brak{\nu t}$ we have to use $\mathcal{E}$ (or $\mathcal{D}$) instead of $\mathcal{E}_{mom}$ since the latter does not contain the factor $\brak{\lambda_{\nu,k} t}$. More precisely, as done to get \eqref{bd:sqrtTD0sup}, we have
 \begin{align}
\label{bd:sqrtTDsupmom}
&  \sup_{k:\, |k|\leq \delta_0^{-1}\nu}\sum_{\alpha+|\beta|\leq B'}\nu\mathcal{J}^{\alpha,\beta}_{M'+M_{J'},\gtrsim \nu}\lesssim \norm{P_{\leq \delta_0^{-1}\nu} \sqrt{\cD^{T.d.}_{M'+M_{J'},B'}}}_{L^\infty_k}\\
&\times  \bigg(  \norm{P_{> \delta_0^{-1}\nu} \sqrt{\cD^{e.d.}_{M'+M_{J'},B'}}}_{L^2_k} \norm{P_{> \delta_0^{-1}\nu} \brak{\lambda_{\nu,k} t}\sqrt{E^{e.d.}_{M'+M_{J'},B'}}}_{L^2_k} \\
&+\norm{P_{> \delta_0^{-1}\nu} \brak{\lambda_{\nu,k} t}\sqrt{E^{e.d.}_{M'+M_{J'},B'}}}_{L^2_k}\norm{P_{> \delta_0^{-1}\nu} \sqrt{\cD^{e.d.}_{M'+M_{J'},B'}}}_{L^2_k} \bigg)\\
\label{bd:sqrtTD1mom}&\lesssim \sqrt{\mathcal{E}}\sqrt{\mathcal{D}_{mom}}\sqrt{\mathcal{D}_{mom,LF}},
\end{align}
where in the last inequality we used $M'+M_{J'}\leq M$ and $J\geq 1$.
Finally, all the error terms that do not have velocity weights are clearly bounded in the same way, for instance the ones in \eqref{def:errmacro}. Hence, the proof of Proposition \ref{lem:momerrLinfty} is over.
\end{proof}

\appendix

\section{Basic inequalities}
To handle to anisotropic dissipation in the problem under consideration in this paper, we need the following lower bound.
\begin{lemma}
	\label{lemma:weightedHs}
	Let $R>1$, $0<s\leq1$ and $\gamma<0$. Then
	\begin{equation}
		\label{bd:wHs}
		\norm{\mathbbm{1}_{|v|\leq R}g}^2_{H^s_{\gamma/2}}\gtrsim  \frac{1}{R^{|\gamma|(2-s)}}\norm{\mathbbm{1}_{|v|\leq R}g}^2_{H^s}.
	\end{equation} 
\end{lemma}
\begin{proof}
Consider first $s<1$. Let $0<\delta<1$ be a small parameter to be chosen. Then
	\begin{align}
		\label{bd:recovery0}
		\norm{\mathbbm{1}_{|v|\leq R}g}^2_{H^s_{\gamma/2}}&\geq \norm{\mathbbm{1}_{|v|\leq R}g}^2_{L^2_{\gamma/2}}+\delta \norm{\mathbbm{1}_{|v|\leq R}g}^2_{\dot{H}^s_{\gamma/2}}
	\end{align}
	Since $(a+b)^2\geq a^2/2-b^2$, we get
	\begin{align}
		\norm{\mathbbm{1}_{|v|\leq R}g}^2_{\dot{H}^s_{\gamma/2}}&=\iint
		\frac{\left(\jap{v}^\frac{\gamma}{2}(\mathbbm{1}_{|v|\leq R} g-\mathbbm{1}_{|v'|\leq R} g') +\mathbbm{1}_{|v'|\leq R} g'(\jap{v}^\frac{\gamma}{2}-\jap{v'}^\frac{\gamma}{2})\right)^2}{|v-v'|^{d+2s}}\dd v\dd v'\\
		&\geq \frac12\iint \jap{v}^\gamma\frac{(\mathbbm{1}_{|v|\leq R} g-\mathbbm{1}_{|v'|\leq R} g')^2}{|v-v'|^{d+2s}}\dd v\dd v'\\
		&\quad -\iint \left( \mathbbm{1}_{|v-v'|<R^q}+\mathbbm{1}_{|v-v'|\geq R^q}\right)\mathbbm{1}_{|v'|\leq R}(g')^2\frac{(\jap{v}^\frac{\gamma}{2}- \jap{v'}^\frac{\gamma}{2})^2}{|v-v'|^{d+2s}}\dd v\dd v'\\
		&:= I^1-I_{<R^q}^2-I_{\geq R^q}^2,
	\end{align}
where $q>0$ is a constant to be chosen later.	For $I_1$, recalling that  $\gamma<0$, one has 
	\begin{equation}
	\label{bd:I1A}
		I_1= \frac14\iint (\jap{v}^\gamma+\langle{v'}\rangle^\gamma)\frac{(\mathbbm{1}_{|v|\leq R} g-\mathbbm{1}_{|v'|\leq R} g')^2}{|v-v'|^{d+2s}}\dd v\dd v'\gtrsim \frac{1}{R^{|\gamma|}} \norm{\mathbbm{1}_{|v|\leq R}g}^2_{\dot{H}^s},
	\end{equation}
where in the last bound we used that on the support of the integral we cannot have  $|v|\geq R$ and $|v'|\geq R$.
To control $I^2_{<R^q}$, by the mean value theorem we get
	\begin{equation}
		\mathbbm{1}_{|v-v'|<1}(\jap{v}^\frac{\gamma}{2}- \jap{v'}^\frac{\gamma}{2})^2\lesssim 	\mathbbm{1}_{|v-v'|<1}|v-v'|^2(\jap{v}^{\gamma-2}+\jap{v'}^{\gamma-2})\lesssim \mathbbm{1}_{|v-v'|<1}|v-v'|^2 \jap{v'}^{\gamma}.
	\end{equation}
	Therefore, since $s<1$ we have
	\begin{equation}
	\label{bd:I2A}
		|I_{<R^q}^2|\lesssim \int \mathbbm{1}_{|v'|\leq R}\jap{v'}^{\gamma}(g')^2\dd v'\int\mathbbm{1}_{|v-v'|<R^q}\frac{1}{|v-v'|^{3-2(1-s)}}\dd v\lesssim R^{2q(1-s)}\norm{\mathbbm{1}_{|v|\leq R}g}_{L^2_{\gamma/2}}^2.
	\end{equation}
	For the remaining term 
	\begin{align}
	\label{bd:I2A1}
		|I_{\geq R^q}^2|&\lesssim  \int \mathbbm{1}_{|v'|\leq R}(g')^2\dd v'\int\mathbbm{1}_{|v-v'|>R^q}\frac{1}{|v-v'|^{d+2s}}\dd v\lesssim R^{-2qs}\norm{\mathbbm{1}_{|v|\leq R}g}_{L^2_v}^2.
	\end{align}
	 Thus, combining \eqref{bd:I1A}, \eqref{bd:I2A} and \eqref{bd:I2A1} we deduce that there is a constant $C$ such that 
	\begin{equation}
	\label{bd:HsA}
		\norm{\mathbbm{1}_{|v|\leq R}g}^2_{\dot{H}_{\gamma/2}}\geq \frac{1}{CR^{|\gamma|}}\norm{\mathbbm{1}_{|v|\leq R}g}^2_{\dot{H}^s}-CR^{2q(1-s)}\norm{\mathbbm{1}_{|v|\leq R}g}^2_{L^2_{\gamma/2}}-C R^{-2qs}\norm{\mathbbm{1}_{|v|\leq R}g}_{L^2_v}^2.
	\end{equation}
	Choosing
	\begin{equation}
	\delta=cR^{-2q(1-s)}, \qquad c=1/(4C), \qquad q=|\gamma|/2,
	\end{equation}
	and plugging \eqref{bd:HsA} in \eqref{bd:recovery0}, we get
	\begin{align}
		\notag	\norm{\mathbbm{1}_{|v|\leq R}g}^2_{H^s_{\gamma/2}}&\geq \frac34\norm{\mathbbm{1}_{|v|\leq R}g}^2_{L^2_{\gamma/2}}+\frac{1}{4C^2R^{|\gamma|(2-s)}}\norm{\mathbbm{1}_{|v|\leq R}g}^2_{\dot{H}^s}-\frac{1}{4R^{|\gamma|}}\norm{\mathbbm{1}_{|v|\leq R}g}^2_{L^2_v}\\
		\notag &\geq \frac{1}{2R^{|\gamma|}}\norm{\mathbbm{1}_{|v|\leq R}g}^2_{L^2}+\frac{1}{4C^2R^{|\gamma|(2-s)}}\norm{\mathbbm{1}_{|v|\leq R}g}^2_{\dot{H}^s}.
	\end{align}
	Since $2-s> 1$, the proof for $s<1$ is over. When $s=1$, we can explicitly compute the commutator $[\nabla_v,\brak{v}^{\gamma/2}]$ and we omit the details of this simpler proof. 
\end{proof}
For convenience of the reader, we recall here some basic inequalities used in \cite{alexandre2011global,alexandre2012boltzmann} and that we exploit in Section \ref{sec:preliminaries}.
\begin{lemma}[\cite{alexandre2012boltzmann}*{Lemma 2.5 }]
	\label{lemma:exp}
	For any $\alpha>-3$, $\rho>0,\, \alpha\in \RR$ one has 
	\begin{equation}
		\int_{\RR^3} |w|^\alpha \jap{w}^\beta\jap{w+u}^\alpha \mu^{\rho}(w+u)\dd w\approx \jap{u}^{\alpha+\beta}.
	\end{equation}
\end{lemma}
\begin{lemma}[\cite{alexandre2011global}*{Lemma 2.5}]
	\label{lemma:angular}
	For any $g \in C^1_b$ one has 
	\begin{equation}
		\int_{\SS^2}b(\cos(\theta))|g(v_*)-g(v_*')|\dd \sigma \leq C_g \jap{v-v_*}^{2s}.
	\end{equation}
When $0<s\leq 1/2$ one can replace $\jap{v-v_*}$ with $|v-v_*|$.
\end{lemma}
\begin{proof}
	By \eqref{def:v'v*}, observe that 
	\begin{equation}
		\label{eq:v*-v*'}
		|v_*'-v_*|^2=\frac{|v-v_*|^2}{2}\left(1-\frac{(v-v_*)\cdot \sigma}{|v-v_*|}\right)=\frac{|v-v_*|^2}{2}\left(1-\cos(\theta)\right)=\sin(\theta/2)^2|v-v_*|^2.
	\end{equation}
Using Taylor's formula 
\begin{equation}
	|g(v_*)-g(v_*')|\leq C_g|v_*-v_*'|= C_g\sin(\theta/2)|v-v_*|.
\end{equation}
Let $0<\delta<\pi/2$. Combining the inequality above with the property \eqref{eq:hypb}, we infer 
\begin{align*}
	\int_{\SS^2}b(\cos(\theta))|g(v_*)-g(v_*'))|\dd \sigma &\lesssim_g |v-v_*|\int_0^\delta \frac{\sin(\theta/2)}{\theta^{2s+1}}\dd \theta+\int_\delta^{\pi/2}\frac{1}{\theta^{1+2s}}\dd \theta\\
	&\lesssim_g |v-v_*|\delta^{-2s+1}+\delta^{-2s}
\end{align*} 
If $|v-v_*|<1$ then choose $\delta=1/10$ and use $|v-v_*|\leq 1\leq \jap{v-v_*}^{2s}$ . If $s\leq 1/2$ then it is also true $|v-v_*|\leq |v-v_*|^{2s}$. When $|v-v_*|>1$, take $\delta=|v-v_*|^{-1}$. 
\end{proof}

 \section*{Acknowledgments} 
The research of JB was supported by NSF Award DMS-2108633.
The research of MCZ was supported by the Royal
Society through a University Research Fellowship (URF\textbackslash
R1\textbackslash 191492).
The research of MD was supported by GNAMPA-INdAM  through the grant D86-ALMI22SCROB\_01 acronym DISFLU. 
This research was started while visiting the Newton Institute during the 2022 Frontiers in Kinetic Theory Programme; the authors would like to thank the Newton Institute for their hospitality during this time. 

\bibliographystyle{siam}
\bibliography{bibboltzmann}
\end{document}